\documentclass[12pt]{article}
\usepackage{amsmath,amssymb}
\usepackage{amsthm}
\usepackage{epsfig}
\usepackage{graphics}
\usepackage{graphicx}
\usepackage{colordvi}
\usepackage{mathrsfs} 
\usepackage{stmaryrd}
\usepackage{geometry}
\usepackage{dsfont}
\usepackage{tikz}
\usetikzlibrary{automata}
\usetikzlibrary{snakes}
\usetikzlibrary{decorations.pathmorphing}
\usepackage{adjustbox}
\usepackage{longtable}
\usepackage{comment}
\usepackage{url}
\usepackage[colorlinks=off,hidelinks]
{hyperref}
\hypersetup{
    colorlinks=true,
   linkcolor=blue!75!black,
    filecolor=blue!75!black,      
    urlcolor=blue!75!black, 
    citecolor=blue!75!black,
}
\usepackage[textsize=small]{todonotes}
\usepackage[margin=1cm, size=small]{caption}
\usepackage{multirow}
\allowdisplaybreaks[4]
\oddsidemargin
-.5in \evensidemargin -.5in\marginparwidth 0.4in
\usepackage{color,soul}
\usepackage{xcolor}

\usepackage{etoolbox}
\patchcmd{\thebibliography}{\section*}{\section}{}{}
\usepackage{tocloft}

\usepackage{chngcntr}
\counterwithin{table}{section}
\counterwithin{figure}{section}
\usepackage{caption}
\usepackage{subcaption}

\newcommand{\cedge}{{\,\leftrightsquigarrow\,}}
\newcommand{\edge}{{\;\leftrightarrow\;}}
\newcommand{\spin}{{\bs \sigma}}

\renewcommand{\Re}{{\rm Re}}
\renewcommand{\Im}{{\rm Im}}
\newcommand{\cc}{{^\circ}}
\newcommand{\1}{\mathds 1}

\newcommand{\C}{{\mathscr C}}

\newcommand{\B}{\mathscr B}

\newcommand{\T}{{\Bbb T}}

\newcommand{\vep}{{\varepsilon}}

\newcommand{\rest}{{\upharpoonright}}
\newcommand{\dint}{{\int\!\!\!\int}}
\newcommand{\ind}{{\perp\!\!\!\perp}}

\newcommand{\supp}{{\rm supp}}

\newcommand{\bs}{\boldsymbol}
\newcommand{\ms}{\mathscr}
\renewcommand{\P}{{\mathbb P}}
\newcommand{\E}{{\mathbb E}}

\newcommand{\R}{{\Bbb R}}

\renewcommand{\i}{{\mathtt  i}}
\newcommand{\defeq}{\stackrel{\rm def}{=}}

\renewenvironment{proof}[1][\proofname]{\noindent {\bfseries #1.}\;}{\hfill\ensuremath{\blacksquare}\\}
\renewcommand{\d}{{\mathrm d}}
\newcommand{\e}{{\rm e}}

\newcommand{\I}{{\rm I}}
\newcommand{\II}{{\rm II}}
\newcommand{\III}{{\rm III}}

\newcommand{\ovarphi}{\overline{\varphi}}
\newcommand{\oovarphi}{\overline{\overline{\varphi}}}
\newcommand{\hvarphi}{\widehat{\varphi}}
\newcommand{\eqspace}{\quad\;}
\renewcommand{\u}{{\mathbf u}}
\newcommand{\bi}{{\mathbf i}}
\newcommand{\bj}{{\mathbf j}}

\newcommand{\N}{{\mathbf N}}

\newcommand{\dgamma}{{\dot{\gamma}}}
\newcommand{\lv}{\Lambda_\vep}
\newcommand{\two}{{\sqrt{2}}}

\newcommand{\EM}{\gamma_{\mathsf E\mathsf M}}
\newcommand{\less}{\lesssim}
\newcommand{\more}{\gtrsim}
\newcommand{\BES}{{\rm BES}}
\newcommand{\ov}{{\overline{\vep}}}
\newcommand{\oF}{\overline{F}}
\newcommand{\wF}{\widehat{F}}

\newcommand{\s}{{\mathfrak s}}
\renewcommand{\S}{{\mathfrak S}}
\newcommand{\oS}{{\overline{\S}\mbox{}^\beta_\vep}}
\newcommand{\ooS}{{\overline{\overline{\mathfrak S}}\mbox{}^\beta_\vep}}
\newcommand{\ooA}{{\overline{\overline{A}}}}
\newcommand{\uvep}{{}^{\,\vep}}

\newcommand{\ovp}{{\ov_{\ref{oovarphi:positive}}}}
\newcommand{\oevarphi}{{\overline{\mathcal E(\hvarphi)}  }}
\usepackage{currfile}
\usepackage{fancyhdr}
\newtheoremstyle{slantthm}{10pt}{10pt}{\slshape}{}{\bfseries}{}{.5em}{\thmname{#1}\thmnumber{ #2}\thmnote{ (#3)}.}
\newtheoremstyle{slantrmk}{10pt}{10pt}{\rmfamily}{}{\bfseries}{}{.5em}{\thmname{#1}\thmnumber{ #2}\thmnote{ (#3)}.}

\begin{document}
\theoremstyle{slantthm}
\newtheorem{thm}{Theorem}[section]
\newtheorem{prop}[thm]{Proposition}
\newtheorem{lem}[thm]{Lemma}
\newtheorem{rlem}[thm]{\textcolor{red}{Lemma}}
\newtheorem{cor}[thm]{Corollary}
\newtheorem{disc}[thm]{Discussion}
\newtheorem{conj}[thm]{Conjecture}
\newtheorem*{thmm}{Theorem}

\theoremstyle{slantrmk}
\newtheorem*{notation}{Notation}
\newtheorem{ass}[thm]{Assumption}
\newtheorem{cond}[thm]{Condition}
\newtheorem{rmk}[thm]{Remark}
\newtheorem{defi}[thm]{Definition}
\newtheorem{eg}[thm]{Example}
\newtheorem{que}[thm]{Question}
\numberwithin{equation}{section}
\newtheorem{quest}[thm]{Quest}
\newtheorem{prob}[thm]{Problem}
\newtheorem{convention}[thm]{Convention}

\title{\vspace{-.5cm}
\bf The critical 2D delta-Bose gas as mixed-order asymptotics of planar Brownian motion\footnote{Support from the Natural Science and Engineering Research Council of Canada is gratefully acknowledged.}}
\author{Yu-Ting Chen\footnote{Department of Mathematics and Statistics, University of Victoria, British Columbia, Canada.}\,\,\footnote{Email: \url{chenyuting@uvic.ca}}  }
\date{\today\vspace{-.5cm}\\}

\maketitle
\abstract{We consider the 2D delta-Bose gas by a smooth mollification of the delta potential, where the coupling constant is in the critical window. The main result proves that for two particles, the approximate semigroups on $\mathscr B_b(\Bbb R^2)$ for the Schr\"odinger operator with singular interaction at the origin converge pointwise in the initial condition. This convergence extends earlier functional analytic results for the convergence in the $L_2$-norm resolvent sense. The central methods introduced here apply the excursion theory of the 2D Bessel process and the ergodicity of the winding number of planar Brownian motion.  The limiting semigroup thus shows both the Kallianpur--Robbins law for additive functionals of planar Brownian motion and Kasahara's second-order law for the fluctuations. As an application, the mode of convergence is extended to the $N$-particle delta-Bose gas for all $N\geq 3$.\vspace{.2cm}

\noindent \emph{Keywords}: Delta-Bose gas; Schr\"odinger operators; additive functionals; Bessel processes; planar Brownian motion; local times and excursion theory;  
stochastic heat equation; KPZ equation.\smallskip 

\noindent \emph{Mathematics Subject Classification (2000)}: 60J55, 60J65, 60H30
}

\setcounter{tocdepth}{2}

\tableofcontents

\section{Introduction}
In this work, we consider the quantum statistical mechanical system for two particles in the plane subject to attractive interactions. The Hamiltonian is the delta-Bose gas formally defined by
\begin{align}\label{def:deltaBose}
-\Delta_{x_1}-\Delta_{x_2}-\Lambda\delta(x_2-x_1).
\end{align}
where $\Delta_{x_j}$ is the two-dimensional Laplacian with respect to $x_j\in \R^2$, $\Lambda>0$ is a coupling constant for the interaction between two particles at $x_1,x_2$, and $\delta$ is the Dirac delta function. The construction of the semigroup solution by short-range interactions has been studied extensively by $L_2$-functional analytic methods and extends to multi-particles \cite{AGHH:2D, AGHH:Solvable, DFT:Schrodinger, DR:Schrodinger, GQT}. Our interest in this model arises from the probabilistic counterparts of its projection to $x_2-x_1$. By the Feynman--Kac formula, the corresponding semigroup can be formally presented as the expected exponential functional of the local time of planar Brownian motion. Additionally, in terms of the Kardar--Parisi--Zhang equation for interface growth, the semigroup characterizes the second moment of the stochastic heat equation in two dimensions. Based on these relations, our goal in this paper is to investigate the construction of the semigroup on $\B_b(\R^2)$ by a smooth mollification of the delta potential. The convergence holds pointwise with respect to the initial condition.

To fix ideas, we begin with the definition of the approximate Hamiltonians for \eqref{def:deltaBose} and recall the earlier results by $L_2$-methods. Fix a probability density $\phi\in \C_c(\R^2)$ and a constant $\lambda\in\R$. Define \begin{align}
\phi_\vep(x)= \vep^{-2}\phi(\vep^{-1}x),\quad 
\lv =\lv(\lambda)=\frac{2\pi}{\log \vep^{-1}}+\frac{2\pi \lambda}{(\log \vep^{-1})^2}.\label{def:J+lv}
\end{align}
As an approximation of the interaction part of the delta-Bose Hamiltonian in \eqref{def:deltaBose} for $x_2-x_1$, set
\begin{align}\label{def:Deltabeta}
H_\vep\defeq-\Delta_x -\Lambda_\vep \phi_\vep(x),
\end{align}
where $\Delta_x$ is the two-dimensional Laplacian. Then $H_\vep^\lambda$ converges to some $H^\lambda$, as self-adjoint operators on $L^2(\R^2)$, in the norm resolvent sense, that is, 
\[
(q I-H_\vep)^{-1}\xrightarrow[\vep \to 0]{} (q I-H)^{-1} \mbox{}, \quad \forall\; q:\Im (q)\neq 0
\]
as a convergence of the $L_2(\R^2)$-operator norms. Hence, for the convergence of semigroups, it holds that  $P^\beta_{\vep;t}=\e^{-tH_\vep}$ converges to $P^\beta_t=\e^{-tH}$ in the operator norm \cite[pp. 284--291]{RS}. 
The kernels  $P^\beta_t(x,y)$ can be characterized explicitly by using the Green function $G_q(x)=\int_0^\infty \e^{-q t}P_{2t}(x)\d t$ associated with the probability density $P_t(x-y)=P_t(x,y)$ of planar Brownian motion and a positive constant $\beta$:
\begin{align}\label{logbeta}
\frac{\log \beta}{2}=-\int_{\R^2}\int_{\R^2}(\log |z-z'|)\phi(z)\phi(z')\d z\d z'+(\log 2+\lambda-\EM),
\end{align}
where $\log$ is the natural logarithm and $\EM=0.57721...$ is the Euler--Mascheroni constant. 
We have
\begin{align}\label{def:Rlambda}
\int_0^\infty\e^{-q t}P^\beta_t(x,z)\d t=G_q(x-z)+\frac{4\pi}{\log (q/\beta)}G_q(x)G_q(z),\quad  \forall\;q:\Re(q)>\beta.
\end{align}
See \cite[Section~2]{AGHH:2D}, \cite[Chapter~I.5, especially pp.102--104]{AGHH:Solvable}, and  \cite[(2.7) and (2.10) for $n=2$ and $\beta=\e^{-\alpha}$]{ABD:Schrodinger} for this convergence result. 

An important application of these convergence results considers connections between the delta-Bose gas and the Kardar--Parisi--Zhang equation \cite{KPZ} in two spatial dimensions for isotropic surface growth models \cite{BC}. By the isotropy,  the Kardar--Parisi--Zhang equation can be viewed in the form of the stochastic heat equation. Moreover, by the Feynman--Kac dual relation, the semigroup solution of the multi-particle delta-Bose gas coincides with the moments of the stochastic heat equation. On the other hand, in stark contrast to the case of one dimension, the stochastic heat equation in two dimensions does not allow for construction by It\^{o}'s theory due to the singularity of paths from the noise. Recently, there have been several methods for the construction by using smooth mollifications of the noise and appropriate normalizations. Under these settings, these results prove a sharp transition that starts from Gaussian limits in the subcritical case. Convergences in the critical regime have been obtained in \cite{CSZ,Feng,GQT}. We remark that  the above Hamiltonians and the multi-particle generalizations are widely considered in one spatial dimension for the repulsive case $\Lambda<0$, known
as the Lieb--Liniger model  \cite{LL:Bose,L:Bose}. In this case, the model is exactly solvable by the Bethe ansatz,

\begin{thm}[Main result]\label{thm:main1}
$P^\beta_{\vep;t}f(x)\to P^\beta_tf(x)$ for all $t\in [0,\infty)$, $f\in\B_b(\R^2)$ and $ x\neq 0$.
\end{thm}

The proof proceeds with the Feynman--Kac representation of the approximate semigroups as exponential functionals with positive exponents of planar Brownian motion. With the scaling of $\lv$ specified above, the additive functionals approximate the local time at the origin of the planar Brownian motion and are known to converge in distribution to a standard exponential random variable by the Kallianpur--Robbins law. The main character is that despite the polarity of the origin, the coupling constant is capable of inducing non-triviality of the exponential functional in in the limit. 

The present methods are drastically different from the functional analytic methods for the results discussed above in the $L_2$-setting. The basic idea is motivated by Kasahara and Kotani's \cite{KK} process-level approach, which extends the Kallianpur--Robbins law for the convergence of additive functionals of planar Brownian motion \cite{KR1,KR2} and Kasahara's second-order law for the fluctuations \cite{K1,K2}. This approach goes back to \cite{PSV}. Specifically, we use the strong ergodicity of the Brownian angular process for a crucial reduction, showing that the approximate semigroups depend only on the two-dimensional Bessel process. See also \cite{PY,PY2} for extensive discussions of additive functionals of planar Brownian motion. For the approximate semigroups under consideration, the issue arises from using the coupling constants in the critical window. After all, proving the asymptotic behaviour we need is at the level of expectations. The exponential random variable from the Kallianpur--Robbins law has an exponential moment that blows up at criticality. This part of the proof relies heavily on the excursion theory of the two-dimensional Bessel process for precise representations of the exponential functionals by the Green functions of the Bessel process.  In particular, these methods show that the Kallianpur--Robbins law and Kasahara's second-order law are both present in the limiting semigroup. We notice that  \cite{CFH} shows a comparable result in the form of strong approximations at the process level. 

As an application, we extend the mode of convergence in Theorem~\ref{thm:main1} to the multi-particle setting. The limiting semigroup is presented as an infinite series as in the recent work of \cite{GQT} under the $L_2$-weak convergence by functional analytic methods. See also \cite{DFT:Schrodinger,DR:Schrodinger}. The starting point of the present proof begins with a probabilistic counterpart of a series expansion of the approximate semigroup from \cite{GQT}
 by Poisson calculus. The expansion shows that the multi-particle semigroups can be reduced to those for two particles. To pass the limit of the series term by term, we extend some methods from \cite{CSZ} for the case of three particles to the case of $N$-particles for all $N\geq 3$. 

% and \cite{GQT} for multiple particles. To obtain the multi-dimensional generalization of the mode of convergence in Theorem~\ref{thm:main1}, the methods from \cite{CSZ} are suitable, and we extend some of them to pass the limit of the approximate series solutions. 

\medskip

\noindent {\bf Frequently used notation.} $C(T)\in(0,\infty)$ is a constant depending only on $T$ and may change
from inequality to inequality unless otherwise indexed by labels of equations. Other constants and parameters are defined analogously. We write $A\less B$ or
$B\more A$ if $A\leq CB$ for a universal constant $C\in (0,\infty)$. $A\asymp B$ means both $A\less B$ and $B\less A$. For a process $X$, $\E_x=\E^X_x$ or $\E_\nu=\E^X_\nu$ stresses that the initial condition is the point $x$ or obeys the law $\nu$.

\section{Attractive interactions of two particles as excursions}\label{sec:twobody}
Our goal in this section is to prove Theorem~\ref{thm:main1} for a construction of the critical delta-Bose gas semigroup on $\B_b(\R^2)$ in the presence of two particles. Recall that the approximation scheme is via a smooth mollification of the delta potential. We work with the Hamiltonian in \eqref{def:Deltabeta} for the interaction part. The semigroup solution can be characterized as the following Feynman--Kac semigroup:
\begin{align}\label{def:FK}
P^\beta_{\vep;t}f(x)\defeq \E^W_{x/\two}\left[\exp\left\{\lv\int_0^t \phi_\vep(\sqrt{2}W_r)\d r\right\}f(\two W_t)\right],\quad f\in \B_b(\R^2),
\end{align}
for a two-dimensional standard Brownian motion $W$.  Here, $\sqrt{2}W$ is identified as the difference $B^2-B^1$ of two independent planar Brownian motions $B^1$ and $B^2$ under attractive interactions. 

For \eqref{def:FK}, we fix once for all $\ov=\ov(\lambda)\in (0,\e^{-\e^{100}})$ such that  $\lv=\lv(\lambda) \in (0,1)$ and $1+\lambda/\log \vep^{-1}\in (0,2]$ for all $\vep\in (0,\ov]$. Set
\begin{align}\label{def:varphi}
\varphi(x)=\phi(\two x)\quad \&\quad \varphi_\vep(x)=\phi_\vep(\two x)
\end{align}
to circumvent $\sqrt{2}$'s in \eqref{def:FK}, and fix $M_\varphi\in (0,\infty)$ such that $\supp(\varphi)\subseteq \{x;|x|\leq M_\varphi\}$. We write
\[
 A_\vep^{o}(t)\defeq\lv \int_0^t \phi_\vep(\two W_r)\d r=\lv \int_0^t \varphi_\vep(W_r)\d r.
\]

The next proposition is the starting point of our investigation. Write $P_t(x,\d y)=P_t(x,y)\d y$. The family $(P_t)$ is also understood as a semigroup of operators on $\B_b(\R^2)$. 

\begin{prop}\label{prop:FK}
For all $\vep\in (0,\ov)$, $f\in \B_b(\R^2)$, and $x\in \R^2$, it holds that
\begin{align}
\E_{x}[\e^{A^o_\vep(t)}f(W_{t})]&= P_tf(x)+\lv \int_0^t\d s\int_{\R^2} P_{s}(x,\d y) \varphi_\vep( y)\E_{y}[\e^{A^o_\vep(t-s)}f(W_{t-s})]
\label{eq:FK1}\\
\begin{split}\label{eq:FK2}
&= P_tf(x)+\lv \int_0^t \d s\int_{\R^2}P_s(x,\d y) \varphi_\vep( y)P_{t-s}f(y)\\
&\eqspace +\lv ^2\int_0^t\d s \int_{\R^2}P_{s}(x,\d y)  \varphi_\vep( y)\int_0^{t-s}\d \tau \E_y\big[\e^{A^o_{\vep}(\tau)}\varphi_\vep( W_\tau) P_{t-s-\tau}f(W_{\tau})\big].
\end{split}
\end{align}
\end{prop}
\begin{proof}
To lighten notation, write $g(y)$ for $\lv \varphi_\vep( y)$.
The proofs of \eqref{eq:FK1} and \eqref{eq:FK2} are straightforward applications of the following rules of backward differentiation and forward differentiation:  
\begin{align}
1-\e^{\int_{0}^t\d rh(r)}
&=-\int_{0}^t \d sh(s) \e^{\int_{s}^t \d rh(r)}=-\int_0^t\d s \e^{\int_0^s\d rh(r)}h(s).
\label{taylor1}
\end{align}
Indeed, the backward differentiation in \eqref{taylor1} gives
\begin{align}
\E_{x}[\e^{A^o_\vep(t)}f(W_{t})]
&= \E_{x}[f(W_{t})]+\int_0^t\d s \E_{x}\left[g( W_{s})\exp\left\{\int_s^t\d rg( W_{r})\right\}f( W_{t})\right]\notag\\
\begin{split}
&=\E_{x}[f(W_{t})]+\int_0^t\d s\int_{\R^2}P_s(x,\d y) g( y) \E_{y}\left[
\exp\left\{\int_0^{t-s}\d rg( W_{r})\right\}f(W_{t-s})\right],\label{exp1}
\end{split}
\end{align}
which proves \eqref{eq:FK1}, and we obtain from the forward differentiation in \eqref{taylor1} that
\begin{align}
 &\eqspace\E_{y}\left[\exp\left\{\int_0^{t-s}\d r g(W_{r})\right\}f( W_{t-s})\right]\notag\\
&=\E_{y}[f( W_{t-s})]+\int_0^{t-s}\d \tau\E_{y}\left[\exp\left\{\int_0^{\tau}\d rg( W_{r})\right\}g(W_{\tau}) f( W_{t-s})\right]\notag\\
&=\E_{y}[f(W_{t-s})]+\int_0^{t-s}\d \tau  \E_{y}\left[\exp\left\{\int_0^{\tau}\d r g( W_{r})\right\}g( W_\tau) P_{t-s-\tau}f(W_{\tau})\right].\label{diff:forward}
\end{align}
Combining the last equality and \eqref{exp1} proves \eqref{eq:FK2}.
\end{proof}

Observe that for nonzero $x\in \R^2$, the second term on the right-hand side of \eqref{eq:FK2} tends to zero  as $\vep\to 0$. Hence, \eqref{eq:FK2} suggests an informal approximation of the density of $P^\beta_{\vep;t}$ defined by \eqref{def:FK}:
\begin{align}\label{resemble}
P^\beta_{\vep;t}(x,z)\approx P_{2t}(x,z)+
\int_0^t\d s P_{2s}(x)  \int_0^{t-s}\d \tau \frac{\lv^2\E_0\big[\e^{A^o_\vep(\tau)}|W_\tau=0\big]}{4\pi \tau} \cdot P_{2(t-s-\tau)}(z).
\end{align}
In more detail, the approximation uses the following properties: $\E^W_{x/\two}[f(\two W_t)]=\E^W_x[f(W_{2t})]$, $\varphi_\vep(y)\d y\approx (1/2) \delta_0(\d y)$, and $P_\tau(0)=1/(2\pi \tau)$. The following proposition shows an identity taking a form similar to \eqref{resemble}. Recall that the limiting kernel $P_t^{\beta}(x,z)$ from the $L_2$-theory is characterized by \eqref{def:Rlambda}, and $\beta$ is a positive constant defined by \eqref{logbeta}.

\begin{prop}
Write $\Gamma(z)=\int_0^\infty t^{z-1}\e^{-t}\d t$ for the gamma function, and set
\begin{align}\label{def:Tbeta}
\mathfrak s^\beta(\tau)\defeq 4\pi\int_0^\infty \d u\frac{\beta^u \tau^{u-1}}{\Gamma(u)}.
\end{align}
Then 
\begin{align}\label{Lap:Tbeta}
\int_0^\infty \d \tau \e^{-q \tau}\mathfrak s^\beta(\tau)=\frac{4\pi }{\log (q/\beta)},\quad \forall\; q\in (\beta,\infty),
\end{align}
and for all nonzero $x$ and $z$ such that $x\neq z$, 
\begin{align}
P_t^\beta (x,z)
&=P_{2t}(x,z)+ \int_0^t\d sP_{2s}(x)\int_0^{t-s}\d  \tau \mathfrak s^\beta(\tau)P_{2(t-s-\tau)}(z).
\label{def:Pbeta2}
\end{align}

\end{prop}
\begin{proof}
Given $a,b\in (0,\infty)$, the one-dimensional marginals of the Gamma$(a,b)$-subordinator $X^{(a,b)}$ are given by 
\begin{align}\label{def:Gamma}
\P(X^{(a,b)}_u\in \d \tau)=f^{(a,b)}_u(\tau)\d \tau,\quad  f^{(a,b)}_u(\tau)\defeq 
 \frac{b^{au}\tau^{au-1}}{\Gamma(au)}\e^{-b\tau},\quad u,\tau> 0
\end{align}
with the Laplace transforms
\begin{align}\label{Lap:gamma}
\E[\e^{-q X^{(a,b)}_u}]=b^{au}(b+q)^{-au}=\e^{-u a\log (1+q/b)}
\end{align}
\cite[p.73]{Bertoin}. Hence,
$\mathfrak s^\beta(\tau)= 4\pi \int_0^\infty \d u f^{(1,\beta)}_u(\tau)\e^{ \beta\tau }$, 
and \eqref{Lap:gamma} proves \eqref{Lap:Tbeta}:
\begin{align*}
\int_0^\infty \d \tau \e^{-q \tau}\mathfrak s^\beta(\tau)
=4\pi \int_0^\infty \d u\E[\e^{-(q-\beta)X^{(1,\beta)}_u}]=\frac{4\pi }{\log (q/\beta)}
,\quad \forall\; q\in (\beta,\infty).
\end{align*}
The other required identity, in \eqref{def:Pbeta2}, now follows from the last equality in \eqref{Lap:Tbeta} upon inverting Laplace transforms for both sides of \eqref{def:Rlambda}. 
\end{proof}

Henceforth, the methods in the rest of this section are centered around the use of Proposition~\ref{prop:FK}.

\subsection{A-priori bounds}
We begin by proving some a-priori bounds for $P^\beta_{\vep;t}f(x)$. 
Consider a general convolution-type Gr\"onwall inequality: Given $T\in (0,\infty)$ and $a_\vep(T)\in [0,\infty)$,
\begin{align}\label{Fvep:ineq}
F_\vep(t)\leq a_{\vep}(T)+\int_0^t b_\vep(s)F_\vep(t-s)\d s,\quad \forall\;t\in [0,T],
\end{align}
where
\begin{align}
b_\vep(s)=\lv\sup_{y\in \R^2}\int_{\R^2}P_{s}(\vep y,\vep z)\varphi(z)\d z.\label{def:abe}
\end{align}
The following proposition proves exponential growth of nonnegative solutions to \eqref{Fvep:ineq}. 

\begin{lem}\label{lem:gronwall}
For all $T\in (0,\infty)$, $ q\in (q(\|\varphi\|_\infty,\lambda),\infty)$ and $\vep\in (0,\vep_{\ref{ineq:gronwall}}(\|\varphi\|_\infty,q,\lambda)\wedge\ov)$, any $[0,\infty)$-valued increasing solution $F_\vep(t)$ to \eqref{Fvep:ineq} satisfies
\begin{align}
F_\vep(T)\leq C(\|\varphi\|_\infty,q,\lambda) a_\vep(T)\e^{q T}\log \vep^{-1}.
\label{ineq:gronwall}
\end{align}
\end{lem}
\begin{proof}
We first prove a general bound on $F_\vep$ which does not use the particular form of $b_\vep(t)$ in \eqref{def:abe}:
\begin{align}\label{ineq:gronwall0}
F_\vep(T)\leq a_\vep(T)\e^{qT}\sum_{n=0}^\infty \left(\int_0^T \e^{-qs}b_\vep(s)\d s\right)^n.
\end{align}

To get this inequality, let $B_\vep(t)=\int_0^t b_\vep(s)\d s$, $\xi_1,\xi_2,\cdots$ be a sequence of i.i.d. random variables with density  $\1_{[0,T]}(t)b_\vep(t)/B_\vep(T)$, and $S_N=\sum_{n=1}^N\xi_n$. Then by \eqref{Fvep:ineq},  for all $0\leq t\leq T$ and $N\in \Bbb N$,
\begin{align}\label{G:ineq}
\begin{split}
F_\vep(t)&\leq a_\vep(T)\sum_{n=0}^{N-1}B_\vep(T)^n\P(S_n\leq t) + B_\vep(T)^N\P(S_N\leq t)\cdot \sup_{0\leq s\leq t}F_\vep(s)
\end{split}
\end{align}
 \cite[pp.22--24]{Dalang}. To use this bound as $N\to\infty$, note that for all $q\in (0,\infty)$, 
\begin{align}\label{series:q=0}
\int_0^\infty q \e^{-q t}\sum_{n=0}^\infty B_\vep(T)^n\P(S_n\leq t)\d t= \sum_{n=0}^\infty B_\vep(T)^n\E[\e^{-q \xi_1}]^n=\sum_{n=0}^\infty \left(\int_0^T \e^{-qs}b_\vep(s)\d s\right)^n,
\end{align}
where the first equality follows since $\xi_1,\xi_2,\cdots$ are i.i.d., and the second equality uses the definition of the distribution of $\xi_1$. For every $\vep\in (0,\ov)$, we can find some large $q$ such that the last series in \eqref{series:q=0} is convergent. Then the leftmost side of \eqref{series:q=0} shows that $\sum_{n=0}^\infty B_\vep(T)^n\P(S_n\leq t)<\infty$ for all $t\in [0,\infty)$ by monotonicity.  

We are ready to prove \eqref{ineq:gronwall}. Since $\sup_{0\leq s\leq T}F_\vep(s)\leq F_\vep(T)<\infty$ by the assumption on $F_\vep$, the convergence of  $\sum_{n=0}^\infty B_\vep(T)^n\P(S_n\leq t)<\infty$  implies that, by passing $N\to\infty$ for the right-hand side of \eqref{G:ineq}, $F_\vep(t)\leq a_\vep(T)\sum_{n=0}^\infty B_\vep(T)^n\P(S_n\leq t)$ for all $0\leq t\leq T$, which immediately yields \eqref{ineq:gronwall} for $q=0$. For $q\in (0,\infty)$,  we obtain \eqref{ineq:gronwall} from the last equality in \eqref{series:q=0} and the following consequence of \eqref{ineq:gronwall} with $q=0$:
\begin{align}\label{Lap:bdd}
\begin{split}
\e^{-q T}F_\vep(T)&=   \int_T^\infty q\e^{-q t}F_\vep(T)\d t\leq a_\vep(T)\int_0^\infty q \e^{-q t}\sum_{n=0}^\infty B_\vep(T)^n\P(S_n\leq t)\d t.
\end{split}
\end{align}
 
Now we use the particular form of $b_\vep$ in \eqref{def:abe}.
Fix $q\in (0,\infty)$, and recall the definition of $\lv$ in \eqref{def:J+lv}. By changing variable in $s$,
 for all $\vep\in (0,\ov)$ such that $\vep^2q<1$,
\begin{align} 
 \int_0^\infty \e^{-q s}b_\vep(s)
\d s
&\leq  \lv\int_0^{\vep^2}\e^{-q \vep^2 s}\d s\|\varphi\|_\infty +\frac{\lv}{2\pi }\int_{\vep^2}^{\infty}\frac{\e^{-\vep^2q s}}{s}\d s\int_{\R^2}\varphi(z)\d z\notag\\
&\leq  \lv\int_0^1\e^{-q \vep^2 s}\d s\|\varphi\|_\infty+\frac{\lv}{4\pi }\left(\int_{\vep^2q}^{1}\frac{1}{s}\d s+\int_{1}^{\infty}\e^{-s}\d s\right)\notag\\
&\leq \left(\frac{2\pi}{\log \vep^{-1}}+\frac{2\pi \lambda}{(\log \vep^{-1})^2}\right)\|\varphi\|_\infty +\frac{1}{4\pi }\left(\frac{2\pi}{\log \vep^{-1}}+\frac{2\pi \lambda}{(\log \vep^{-1})^2}\right)
[-\log (q \vep^2)+\e^{-1}]\notag\\
\begin{split}\label{Lap:b-bdd}
&=1+\frac{1}{\log \vep^{-1}}\left(2\pi \|\varphi\|_\infty+\frac{-\log q +\e^{-1}}{2}+\lambda\right) \\
&\quad +\frac{\lambda }{(\log \vep^{-1})^2}\left(2\pi \|\varphi\|_\infty+\frac{-\log q +\e^{-1}}{2}\right).
\end{split}
\end{align}
Applying the last inequality to \eqref{ineq:gronwall0} proves \eqref{ineq:gronwall}. The proof is complete.
\end{proof}

The growth of $b_\vep(s)$ in \eqref{def:abe} comes from the use of $\vep y$, rather than $y$ of order $1$. The following lemma shows that the appropriate setting for the application of Lemma~\ref{lem:gronwall} is  \eqref{eq:FK2}. 

\begin{prop}\label{prop:bounds}
For all $T\in (0,\infty)$ and $\vep\in (0,\vep(\|\varphi\|_\infty,\lambda))$,
\begin{align}\label{ineq:apriori}
&\eqspace \sup_{y\in \R^2}\int_0^{T}\lv^2\E_{\vep y}[\e^{A^o_\vep(t)}\varphi_\vep(W_t)]\d t\leq C(\|\varphi\|_\infty,\lambda)\lv [1+\log^+ (\vep^{-2}T)]\e^{q(\|\varphi\|_\infty,\lambda)T}.
\end{align}
\end{prop}
\begin{proof}
Set
\begin{align}
F_\vep(t')&=\sup_{y\in \R^2}\int_0^{t'}\lv^2\E_{\vep y}[\e^{A^o_\vep(t)}\varphi_\vep(W_t)]\d t,\quad
a_\vep(T)=\sup_{y\in \R^2}\int_0^T \lv ^2 \E_{\vep y}[\varphi_\vep(W_t)]\d t.\label{Fa:apriori}
\end{align}
Then \eqref{Fvep:ineq} holds if we apply \eqref{eq:FK1} with $f$ replaced by $\lv^2\varphi_\vep $ and integrate over $0\leq t\leq t'$.
In more detail, after the integration, the corresponding second term in \eqref{eq:FK1} can be written as
\begin{align*}
&\eqspace\lv\int_0^{t'}\d t \int_0^{t}\d s\int_{\R^2}\d z P_{s}(\vep y,\vep z) \varphi( z)\lv ^2\E_{\vep z}\big[\e^{A^o_\vep(t-s)}\varphi_\vep(W_{t-s})\big] \\
&= \int_0^{t'} \d s \lv\int_{\R^2}\d z P_{s}(\vep y,\vep z) \varphi( z)\int_0^{t'-s} \d t\lv^2\E_{\vep z}[\e^{A^o_\vep(t)}\varphi_\vep(W_{t})].
\end{align*}

To apply Lemma~\ref{lem:gronwall}, we only need to bound $a_\vep(T)$. Write
\begin{align}
\int_0^T \lv^2\E_{\vep y}[\varphi_\vep(W_t)]\d t&=\lv ^2\int_{\R^2} \d z\varphi(z)\int_0^{\vep^{-2}T} \frac{1}{2\pi t}\exp\left\{-\frac{|y-z|^2}{2t}\right\}\d t\notag\\
&\leq \lv^2\|\varphi\|_\infty+\lv^2\int_{\R^2}\d z\varphi(z)\int_1^{(\vep^{-2}T)\vee 1}\frac{\d t}{t}.\notag
\end{align}
It follows that 
\begin{align}\label{bounds:03}
\begin{split}
a_\vep(T)&\less C(\|\varphi\|_\infty)\lv^2 [1+\log^+ (\vep^{-2}T)].
\end{split}
\end{align}
The required bound in \eqref{ineq:apriori} now follows from Lemma~\ref{lem:gronwall} 
and \eqref{bounds:03}. 
\end{proof}

\subsection{Laplace transforms of the approximate semigroups}\label{sec:Lap}
From this point on until the end of the proof of Theorem~\ref{thm:main1}, we consider the Laplace transform in time of $P^{\beta}_{\vep;t}f(x)$ defined by \eqref{def:FK}. In this subsection, we give the first simplification by showing the central role of the following functional:
\begin{align}
\S^\beta_\vep(q,y)&\defeq\lv^2\int_0^\infty \d t\e^{-qt}\E^W_{\vep y}[\e^{A^{o}_\vep(t)}\varphi_\vep( W_t)]
=\lv^2\int_0^\infty \d t\e^{-\vep^2qt}\E^W_{y}[\e^{A_\vep(t)} \varphi(W_{t})],
\label{main}
\end{align}
where the second equality uses the Brownian scaling $W_t\stackrel{\rm (d)}{=}\vep^{-1}W_{\vep^{2}t}$ and the notation
\begin{align}
A_\vep(t)\defeq \lv \int_0^t \varphi(W_r)\d r.\label{def:A}
\end{align} 

\begin{prop}
For all $0\neq x\in \R^2$, $q\in (0,q(\|\varphi\|_\infty,\lambda))$, and $f\in \B_b(\R^2)$, it holds that
\begin{align}
\begin{split}
&\lim_{\vep\to 0}\Bigg\{\int_0^\infty\d t \e^{-qt }\E_{x}[\e^{A^o_\vep(t)}f(W_{t})]-\int_0^\infty\d t \e^{-qt } P_tf(x)\\
&\eqspace-\int_{\R^2}\d y\varphi(y) \int_0^\infty \d t\e^{-qt}P_t(x,\vep y)\times \S^\beta_\vep(q,y)\times \int_0^\infty \d t\e^{-qt}P_tf(0)\Bigg\}=0.\label{Lap:int}
\end{split}
\end{align}
Here, $\sup_{\vep\in (0,\ov)}\int_0^\infty \d t\e^{-qt}\E^W_{x}[\e^{A^o_\vep(t)}]<\infty$.
\end{prop}
\begin{proof}
We work with \eqref{eq:FK2} and divide the proof into a few steps. \medskip

\noindent {\bf Step 1.} We claim that for all $0\neq x\in \R^2$ and $f\in \B_b(\R^2)$, 
\begin{gather}
\lim_{\vep\to 0}\int_0^\infty \d t\e^{-qt}\lv \int_0^t \d s\int_{\R^2}P_s(x,\d y) \varphi_\vep( y)P_{t-s}f(y)=0,\quad \forall\; q\in (0,\infty);\label{lim:semi2}\\
\sup_{\vep\in (0,\ov)}\int_0^\infty \d t\e^{-qt}\E_{x}[\e^{A^o_\vep(t)}]<\infty,\quad \forall\; q\in (0,q_{\ref{ineq:apriori}}(\|\varphi\|_\infty,M_\varphi,\lambda)).\label{sup:semi}
\end{gather}
To see \eqref{lim:semi2}, note that
\begin{align*}
&\eqspace\int_0^\infty \d t\e^{-qt}\lv \int_0^t \d s\int_{\R^2}P_s(x,\d y) \varphi_\vep( y)P_{t-s}f(y)\\
&=\lv \int_{\R^2}\d y\varphi(y)\left(\int_0^\infty \d t\e^{-qt}P_t(x,\vep y)\right)\left(\int_0^\infty \d t\e^{-qt}P_{t}f(\vep y)\right)\xrightarrow[\vep\to 0]{}0
\end{align*}
by dominated convergence. 
As for \eqref{sup:semi}, we use \eqref{eq:FK2} and \eqref{lim:semi2} to get
\begin{align}
\begin{split}\label{Lapgibbs:bdd}
\limsup_{\vep\to 0}\int_0^\infty \d t\e^{-qt}\E_{x}[\e^{A^o_\vep(t)}]
&\leq \int_0^\infty \d t \e^{-qt}+\limsup_{\vep\to 0}\int_{\R^2}\d y\varphi(y)
\left(\int_0^\infty\d t \e^{-qt}P_t(x,\vep y)\right)\\
&\eqspace\times \int_0^\infty \d t\e^{-qt}\sup_{0<|y'|\leq M_\varphi}\int_0^t \lv^2\E_{\vep y'}[\e^{A^o_\vep(t)}\varphi_\vep(W_t)],
\end{split}
\end{align}
which is finite for all $q\in (q(\|\varphi\|_\infty,M_\varphi,\lambda),\infty)$ by \eqref{ineq:apriori}. \medskip

\noindent {\bf Step 2.}
In this step, we show that
\begin{align}\label{lim:sup000}
\sup_{z: \varphi_\vep( z)>0}\sup_{0\leq t\leq T}|P_tf( z)-P_tf(0)|\xrightarrow[\vep\to 0]{}0,\quad\forall\; T\in (0,\infty).
\end{align}
To see this property, write $P_{t}f(z)=\E[f(z+\sqrt{t}Z)]$ for $Z\sim \mathcal N(0,1)$. By truncation and a standard approximation of $L_1(\R^2,\d z)$-functions by continuous functions \cite[Theorem~1.18 on p.34]{LL:Analysis}, 
$(z,t)\mapsto P_tf(z)$ is jointly continuous on $\R^2\times \R_+$, which is enough to get \eqref{lim:sup000}.
\medskip 

\noindent {\bf Step 3.}
We are ready to prove \eqref{Lap:int}. By  \eqref{eq:FK2} and \eqref{lim:semi2}, it remains to show that 
\begin{align}\label{lim:delta1}
\begin{split}
&\lim_{\vep\to 0}\int_0^\infty \d t\e^{-qt}\lv ^2\int_0^t\d s \int_{\R^2}P_{s}(x,\d y)  \varphi_\vep( y)\int_0^{t-s}\d \tau \E_y[\e^{A^o_\vep(\tau)}\varphi_\vep( W_\tau) P_{t-s-\tau}f( W_{\tau})]\\
&\eqspace\;\;-\int_0^\infty \d t\e^{-qt}\lv ^2\int_0^t\d s \int_{\R^2}P_s(x,\d y)  \varphi_\vep( y)\int_0^{t-s}\d \tau \E_y[\e^{A^o_\vep(\tau)}\varphi_\vep( W_\tau)] P_{t-s-\tau}f(0)=0.
\end{split}
\end{align} 

Replacing $P_{t-s-\tau}f(0)$ with $\1_{\{t-s-\tau\geq L\}}$ in the second term in \eqref{lim:delta1} yields an integral of the form
\begin{align*}
&\eqspace\int_{\R^2}\d yh_0(y)\int_0^\infty \d t\e^{-qt}\int_0^t\d s  h_{1,y}(s)\int_0^{t-s}\d \tau h_{2,y}(\tau)h_3(t-s-\tau)\\
&=\int_{\R^2}\d yh_0(y)\left(\int_0^\infty \d t\e^{-qt} h_{1,y}(t)\d t\right)\left(\int_0^\infty \d t\e^{-qt} h_{2,y}(t)\d t\right)
\left(\int_0^\infty \d t\e^{-qt} h_{3}(t)\d t\right).
\end{align*}
We also note that by \eqref{ineq:apriori},
\begin{align}\label{limsup:+++1}
\limsup_{\vep\to 0}\int_0^\infty \d t\e^{-qt}\lv ^2\int_0^t\d s \int_{\R^2}P_{s}(x,\d y)  \varphi_\vep( y)\int_0^{t-s}\d \tau \E_y[\e^{A^o_\vep(\tau)}\varphi_\vep( W_\tau)]<\infty
\end{align} 
 for all $q\in (q(\|\varphi\|_\infty,M_\varphi,\lambda),\infty)$.
By \eqref{lim:sup000} and \eqref{limsup:+++1}, \eqref{lim:delta1} holds for all $q\in (q(\|\varphi\|_\infty,M_\varphi,\lambda),\infty)$. The proof is complete.
\end{proof}

{\bf  We only consider nonzero $y$ for $\S^\beta_\vep(q,y)$ in the rest of this section.} This assumption is justified by the role of $\S^\beta_\vep(q,y)$ in \eqref{Lap:int} as an integrand. 

\subsection{Asymptotic radial symmetry}\label{sec:gen}
View $W$ as a complex-valued process. Under a nonzero initial condition $W_0=y$, the skew-product representation of $W$ shows a two-dimensional standard Brownian motion $(\beta_t,\gamma_t)$ such that
\begin{align}\label{def:skew}
W_t=\rho_t\exp\{\i \theta_t\}\quad \mbox{ for }\;\rho_t=|W_t|=\exp\{\beta_{S_t^{-1}}\}\;\mbox{ and }\; \theta_t=\gamma_{S_t^{-1}},
\end{align}
where $S_t=\int_0^t\e^{2\beta_r}\d r$, or equivalently, $S_t^{-1}=\int_0^t \d r/\rho_r^2$, and
$\theta_0\in \R$ satisfies $\e^{\i \theta_0}=y/|y|$  \cite[p.193]{RY}.  Note that $S_t\to\infty$ a.s. by the occupation times formula. We denote the two-dimensional Bessel process $(\rho_t)$ by $\BES^2$. 
 A decomposition of $\varphi$ analogous to the skew-product representation is
\begin{align}\label{def:Jdec}
\varphi(x)=\ovarphi(|x|)+\hvarphi(x),\; x\in \R^2\quad\mbox{ for }\;
 \ovarphi(b)=\frac{1}{2\pi}\int_{-\pi}^\pi \varphi(b\e^{\i\theta})\d \theta,\;b\in [0,\infty).
\end{align}
A key property is $\int \varphi(x)\d x=\int \overline{\varphi}(|x|)\d x$. The rescaled functions $\ovarphi_\vep$ and $\hvarphi_\vep$ are defined as in \eqref{def:varphi}.

In this subsection, we consider three asymptotic reductions of $P^{\beta}_{\vep;t}f(x)$ as $\vep\to 0$. The first two are motivated by \cite{KK} and rely crucially on \eqref{def:skew} and \eqref{def:Jdec} such that the semigroup is shown to be asymptotically radially symmetric.

\subsubsection{Radial symmetrization of initial conditions}
The first step is to approximate $\S^\beta_\vep(q,y)$ in \eqref{main} by replacing its initial condition with the distribution
\[
\u_{|y|}(H)\,\defeq\,\int_{\T}\u (\d \theta)\1_{H}(|y|\e^{\i \theta}),\quad H\in \B(\R^2),
\]
where $\mathbf u$ is the uniform distribution on $\T=[-\pi,\pi)$. Set $\dgamma=\gamma$ mod $2\pi$.
The one-dimensional marginals of $(\dgamma_t)$ have densities $\P(\dgamma_t=\theta)$ with respect to the Lebesgue measure on $\T$ \cite[(1.2) on p.135]{KK}: 
\begin{align}\label{eq:theta}
\P(\dgamma_t=\theta)=\frac{1}{2\pi}+\frac{1}{2\pi}\sum_{n=1}^\infty \e^{-n^2t/2}\cdot 2\cos(n\theta-n\gamma_0),\quad t\in (0,\infty),\; \theta\in \T.
\end{align}
Hence, $\mathbf u$ is the equilibrium distribution of $(\dgamma_t)$, and the following rate of convergence holds: 
\begin{align}\label{mixing}
|\P(\dgamma_t=\theta)-1/(2\pi)|\leq \frac{\mathbf g(t)}{2\pi},\quad\mbox{ where } 
 \mathbf g(t)= 2\sum_{n=1}^\infty \e^{-n^2t/2}.
\end{align}
The counterpart of $\S^\beta_\vep(q,y)$ where the angle of the initial condition is distributed uniformly is
\begin{align}
\oS(q,|y|)\defeq \lv^2\int_0^\infty \d t\e^{-qt}\E^W_{\mathbf u_{\vep |y|}}[\e^{A^{o}_\vep(t)}\varphi_\vep( W_t)]=\lv^2\int_0^\infty \e^{-\vep^2qt}\E^W_{ \u_{|y|}}[\e^{A_\vep(t)} \varphi(W_{t})]\d t.
\end{align}

The next lemma shows the asymptotic equivalence of $\S^\beta_\vep(q,y)$ and $\oS(q,|y|)$. For the proof, recall that the scale function and the speed measure of $\BES^2$ can be chosen to be $s(x)=2\log x$ and $m(\d x)=\1_{(0,\infty)}(x)x\d x$, respectively \cite[p.446]{RY}. Then
for all $0<g<a<d<\infty$ and nonnegative $F$, 
\begin{align}\label{eq:exit}
\begin{split}
\E_a\left[\int_0^{T_g\wedge T_d}F(\rho_s)\d s\right]&=2\int_a^d\frac{(\log a-\log g)(\log d-\log r)}{\log d-\log g} F(r) r\d r\\
&\quad +2\int_g^a\frac{(\log r-\log g)(\log d-\log a)}{\log d-\log g} F(r) r\d r
\end{split}
\end{align}
\cite[Corollary~3.8 on p.305]{RY}.  Moreover, \eqref{eq:exit} implies
\begin{align}
\E_a\left[\int_0^{T_d}F(\rho_s)\d s\right]&=2\int_a^d(\log d-\log r) F(r) r\d r+2(\log d-\log a) \int_0^aF(r) r\d r,\label{eq:Tr}\\
\E_a\left[\int_0^{T_g}F(\rho_s)\d s\right]&=2(\log a-\log g) \int_a^\infty F(r) r\d r+2\int_g^a(\log r-\log g)F(r) r\d r,\label{eq:Tl}
\end{align}
where \eqref{eq:Tr} also uses the polarity of $0$. 

We also use the following implications of the skew-product representation \eqref{def:skew}.  Set
\begin{align}\label{def:TX}
T_a=T_a(X)\defeq \inf\{t\geq 0;X_t=a\}
\end{align}
for any real-valued continuous process $X$. Then for all $b\in (0,\infty)$, $S^{-1}_{T_b(\rho)}=T_{\log b}(\beta)$ with the notation defined in \eqref{def:TX}, and it follows from changing variables that
for any nonnegative function $g$,
\begin{align}
\int_0^{T_{b}(\rho)}g(\rho_v)\d v &=\int_0^{T_{b}(\rho)}g(\exp\{\beta_{S^{-1}_r}\})\d r
=\int_0^{S^{-1}_{T_b(\rho)}}g(\exp\{\beta_v\})\exp\{2\beta_v\}\d v\notag\\
&=\int_0^{T_{\log b}(\beta)}g(\exp\{\beta_v\})\exp\{2\beta_v\}\d v.\label{ineq:SP:g}
\end{align}

\begin{prop}\label{prop:main:c}
Given any $M\in (0,\infty)$ and $q\in (q(\|\varphi\|_\infty,\lambda),\infty)$, the following two-sided uniform bounds hold
for all $\vep\in (0,\vep(\|\varphi\|_\infty,M,\lambda)\wedge \ov)$:
\begin{align}\label{lem:main:c:bdd}
\begin{split}
 \inf_{y:0<|y|\leq M}\S^\beta_\vep(q,y) &\geq  \left( \frac{\int_{-1}^1 (1-t^2)^{-1/2}\cosh(M \sqrt{2\vep^2q }t)\d t}{\int_{-1}^1 (1-t^2)^{-1/2}\cosh(\log\log \vep^{-1} \sqrt{2q \vep^2 }t)\d t}-\frac{C(M,\lambda)}{(\log\log \vep^{-1})^{1/2}}\right)\\
&\eqspace\times \oS (q,\log\log \vep^{-1})-C(\|\varphi\|_\infty,M_\varphi)(\log \log\log \vep^{-1})\lv^2,\\
\sup_{y:0<|y|\leq M}\S^\beta_\vep(q,y) &\leq\left( \frac{\int_{-1}^1 (1-t^2)^{-1/2}\d t}{\int_{-1}^1 (1-t^2)^{-1/2}\cos(\log\log \vep^{-1} \sqrt{2\lv \|\varphi\|_\infty}t)\d t} +\frac{C(M,\lambda)}{(\log\log \vep^{-1})^{1/2}}\right)\\
&\eqspace\times \oS(q,\log\log \vep^{-1}) +C(\|\varphi\|_\infty,M_\varphi)(\log \log\log \vep^{-1})\lv^2.
\end{split}
\end{align}
In particular, since $\sup_{\vep\in (0,\ov)}\overline{\S}\mbox{}^\beta_\vep(q,\log\log \vep^{-1})<\infty$ by Proposition~\ref{prop:bounds}, we get
\begin{align}\label{main:b}
\begin{split}
&\lim_{\vep\to 0}\sup_{y:0<|y|\leq M}\big|\S^\beta_\vep(q,y)-\oS(q,|y|)\big|=0.
\end{split}
\end{align}
\end{prop}
\begin{proof}
Write $\eta_\vep=\log \log \vep^{-1}$, and decompose $\S^\beta_\vep(q,y)$ according to
\begin{align}\label{S:main}
\S_\vep^\beta(q,y) &=\lv^2 \E^W_{ y}\left[\int_0^{T_{\eta_\vep}(\rho)}\e^{-\vep^2qt}\e^{A_\vep(t)}\varphi(W_t)\d t\right]
 +\lv^2 \E^W_{ y}\left[\int_{T_{\eta_\vep}(\rho)}^\infty \e^{-\vep^2qt}\e^{A_\vep(t)}\varphi(W_t)\d t\right]\notag\\
&=\I_{\ref{S:main}}(\vep,q,y)+\II_{\ref{S:main}}(\vep,q,y).
\end{align}
We bound the last two terms separately in the next two steps. \medskip

\noindent {\bf Step 1.} To bound $\I_{\ref{S:main}}(\vep,q,y)$ over $0<|y|\leq M$, we consider
\begin{align}
&\eqspace\I_{\ref{S:main}}(\vep,q,y)
\leq \lv^2 \|\varphi\|_\infty\E^W_{y}\left[\int_0^{T_{\eta_\vep}(\rho)}\exp\left\{\lv\|\varphi\|_\infty\int_0^{t} \1_{[0,M_\varphi]}(\rho_r)\d r\right\} \1_{[0,M_\varphi]}(\rho_t)\d t\right]\notag\\
&= \lv^2 \|\varphi\|_\infty\E^\rho_{|y|}\left[\int_0^{T_{\eta_\vep}(\rho)}\exp\left\{\lv\|\varphi\|_\infty\int_{t}^{T_{\eta_\vep}(\rho)} \1_{[0,M_\varphi]}(\rho_r )\d r\right\} \1_{[0,M_\varphi]}(\rho_t)\d t\right]\notag\\
&\leq \lv^2 \|\varphi\|_\infty\E^\rho_{ |y|}\left[\int_0^{T_{\eta_\vep}(\rho)}\1_{[0,M_\varphi]}(\rho_t)\E^\beta_{ \log \rho_t}\left[\exp\left\{\lv \|\varphi\|_\infty\int_{0}^{T_{ \log \eta_\vep}(\beta)}\e^{2\beta_r}\d r\right\}\right] \d t\right]\notag\\
&\leq \lv^2 \|\varphi\|_\infty\E^\rho_{ |y|}\left[\int_0^{T_{\eta_\vep}(\rho)}\1_{[0,M_\varphi]}(\rho_t) \d t\right]\sup_{z:0<|z|\leq M}
\E^\beta_{ \log |z|}\left[\exp\left\{\lv \|\varphi\|_\infty\int_{0}^{T_{ \log \eta_\vep}(\beta)}\e^{2\beta_r}\d r\right\}\right],\label{II:equilibrium}
\end{align}
where the first equality follows from the second equality in \eqref{taylor1}, and the first inequality follows from the Markov property of $(\rho_t)$ and \eqref{ineq:SP:g}. Note that the use of $\E^\beta_{\log \rho_t}$ in the first inequality is valid since  $\beta_0=\log \rho_0$ by \eqref{def:skew}. In the rest of this proof, we only consider $\vep\in (0,\ov)$ such that $M\leq \eta_\vep<0.5\cdot j_{0,1}/\sqrt{2\lv\|\varphi\|_\infty}$ unless otherwise mentioned, where $j_{0,1}$ is the smallest positive zero of the Bessel function of the first kind.

To bound the right-hand side of \eqref{II:equilibrium}, first, note that by \eqref{eq:Tr}, for all $0<|y|\leq M$,
\begin{align*}
\E^\rho_{ |y|}\left[\int_0^{T_{\eta_\vep}(\rho)}\1_{[0,M_\varphi]}(\rho_t) \d t\right]
&\leq C(M_\varphi)\log \eta_\vep. 
\end{align*}
Additionally, for all $0<|z|\leq M$, it follows from \eqref{bm:prob4} that
\begin{align*}
\E^\beta_{ \log |z|}\left[\exp\left\{\lv \|\varphi\|_\infty\int_{0}^{T_{ \log \eta_\vep}(\beta)}\e^{2\beta_r}\d r\right\}\right]\leq \frac{\int_{-1}^1 (1-t^2)^{-1/2}\d t}{\int_{-1}^1 (1-t^2)^{-1/2}\cos( \eta_\vep\sqrt{2\lv \|\varphi\|_\infty}t)\d t}\leq C(\|\varphi\|_\infty).
\end{align*}
Applying the last two displays to \eqref{II:equilibrium} shows that
\begin{align}
\sup_{y:0<|y|\leq M}\I_{\ref{main}}(\vep,q,y)\leq C(\|\varphi\|_\infty,M_\varphi)(\log \eta_\vep)\lv^2.\label{lem:main:c:bdd1}
\end{align}

\noindent {\bf Step 2.}
To bound $\II_{\ref{main}}$, we use the strong Markov property of planar Brownian motion at $T_{\eta_\vep}(\rho)$:
\begin{align}
\II_{\ref{main}}(\vep,q,y)
&= \E_{ y}^W\left[\e^{-\vep^2q T_{\eta_\vep}(\rho)+A_\vep(T_{\eta_\vep}(\rho))}\S^\beta_\vep\big(q,\eta_\vep\exp\big\{\i \dgamma_{\int_0^{T_{\eta_\vep}(\rho)}\d v/\rho_v^2}\big\}\big)\right]\notag\\
&\leq \E^W_{y}\left[\exp\left\{ \lv\|\varphi\|_\infty\int_0^{T_{\eta_\vep}(\rho)}\1_{[0,M_\varphi]}(\rho_r)\d r\right\}\S^\beta_\vep\big(q,\eta_\vep\exp\big\{\i \dgamma_{\int_0^{T_{\eta_\vep}(\rho)}\d v/\rho_v^2}\big\}\big)\right]\notag\\
\begin{split}
&\leq \E^\rho_{|y|}\left[\exp\left\{\lv \|\varphi\|_\infty\int_0^{T_{\eta_\vep}(\rho)}\1_{[0,M_\varphi]}(\rho_r)\d r\right\}\right]\oS(q,\eta_\vep)\\
& \eqspace + \E^\rho_{ |y|}\left[\exp\left\{\lv\|\varphi\|_\infty\int_0^{T_{\eta_\vep}(\rho)}\1_{[0,M_\varphi]}(\rho_r)\d r\right\}\mathbf g\left(\int_0^{T_{\eta_\vep}(\rho)}\frac{\d v}{\rho_v^2}\right)\right]\oS(q,\eta_\vep)
\label{eq1:main:c}
\end{split}\end{align}
by \eqref{mixing}, where the last inequality follows since $(\rho_t)\ind (\gamma_t)$ and $\int_0^{T_{\eta_\vep}(\rho)}\d v/\rho_v^2>0$ a.s. 

To bound the expectations in \eqref{eq1:main:c}, note that by \eqref{ineq:SP:g},
 \begin{align*}
&\int_0^{T_{\eta_\vep}(\rho)}\frac{\d v}{\rho_v^2}=T_{\log \eta_\vep}(\beta)\quad\&\quad
\int_0^{T_{\eta_\vep}(\rho)}\1_{[0,M_\varphi]}(\rho_r)\d r\leq \eta_\vep^2T_{\log \eta_\vep}(\beta).
\end{align*}
Hence, the bound in \eqref{eq1:main:c} shows that for all $0<|y|\leq M$, 
\begin{align}
\II_{\ref{main}}(\vep,q,y)&\leq 
 \E^\rho_{|y|}\left[\exp\left\{\lv \|\varphi\|_\infty\int_0^{T_{\log \eta_\vep}(\beta)}\e^{2\beta_r}\d r\right\}\right]\oS(q,\eta_\vep)\notag\\
&\eqspace + \E^\rho_{ |y|}\left[\exp\left\{\eta_\vep^2\lv\|\varphi\|_\infty T_{\log \eta_\vep}(\beta)\right\}\mathbf g\big(T_{\log \eta_\vep}(\beta)\big)\right]\oS(q,\eta_\vep)\notag\\
\begin{split}
&\leq\frac{\int_{-1}^1 (1-t^2)^{-1/2}\d t}{\int_{-1}^1 (1-t^2)^{-1/2}\cos(\eta_\vep \sqrt{2\lv \|\varphi\|_\infty}t)\d t}\oS(q,\eta_\vep)\\
&\eqspace+ \E^\rho_{ |y|}\big[\exp\left\{\eta_\vep^2\lv\|\varphi\|_\infty T_{\log \eta_\vep}(\beta)\right\}\mathbf g\big(T_{\log \eta_\vep}(\beta)\big)\big]\oS(q,\eta_\vep),
\label{eq1:main:cc}
\end{split}
\end{align}
where the last inequality uses \eqref{bm:prob4}. Similarly, for all $0<|y|\leq M$,
\begin{align}
\II_{\ref{main}} (\vep,q,y)&\geq  \E^\rho_{|y|}[\e^{-\vep^2q T_{\eta_\vep}(\rho)}]\oS(q,\eta_\vep)- \E^\rho_{ |y|}\big[\e^{-\vep^2q T_{\eta_\vep}(\rho)}\mathbf g\big(T_{\log \eta_\vep}(\beta)\big)\big]\oS(q,\eta_\vep)\notag\\
\begin{split}
&\geq  \frac{\int_{-1}^1 (1-t^2)^{-1/2}\cosh(M \sqrt{2 \vep^2q }t)\d t}{\int_{-1}^1 (1-t^2)^{-1/2}\cosh(\eta_\vep \sqrt{2 \vep^2q }t)\d t} \oS(q,\eta_\vep)-
 \E^\rho_{ |y|}\big[\mathbf g\big(T_{\log \eta_\vep}(\beta)\big)\big]\oS(q,\eta_\vep)
\label{eq2:main:c}
\end{split}
\end{align}
by \eqref{bes:hit2}. For the expectations in \eqref{eq1:main:cc} and \eqref{eq2:main:c}, note that  for any $c\in [0,\|\varphi\|]$, 
\begin{align}
&\eqspace\sup_{\beta_0:\beta_0\leq \log M}\E^\beta_{\beta_0}\big[\e^{\eta_\vep^2\lv cT_{\log \eta_\vep}(\beta)}\mathbf g\big(T_{\log \eta_\vep}(\beta)\big)\big]=\sup_{\beta_0:\beta_0\leq  \log M}2\sum_{n=1}^\infty \E^\beta_{\beta_0}\big[\e^{-(-c\eta_\vep^2\lv +n^2/2)T_{\log \eta_\vep}(\beta)}\big]\notag\\
&=\sup_{\beta_0:\beta_0\leq  \log M}2\sum_{n=1}^\infty \e^{-\sqrt{-2c\eta_\vep^2\lv +n^2}(\log \eta_\vep-\beta_0)}\leq \frac{C(M,\lambda)}{\eta_\vep^{1/2}},
\label{eq2:main:caux}
\end{align}
where the second equality follows from \cite[Proposition~II.3.7 on p.71]{RY}, and the last inequality holds by choosing smaller $\vep$ if necessary in a way that still depends only on $(\|\varphi\|_\infty,M,\lambda)$.

By \eqref{eq1:main:cc}--\eqref{eq2:main:caux}, we get, for all $\vep\in (0, \vep(\|\varphi\|_\infty,M,\lambda))$, 
\begin{align}\label{lem:main:c:bdd2}
\begin{split}
\inf_{y:0<|y|\leq M}\II_{\ref{main}} (\vep,q,y)&\geq  \frac{\int_{-1}^1 (1-t^2)^{-1/2}\cosh(M \sqrt{2 \vep^2q }t)\d t}{\int_{-1}^1 (1-t^2)^{-1/2}\cosh(\eta_\vep \sqrt{2q \vep^2 }t)\d t}\oS(q,\eta_\vep)-\frac{C(M,\lambda)}{\eta^{1/2}_\vep}\oS(q,\eta_\vep),\\
 \sup_{y:0<|y|\leq M}\II_{\ref{main}} (\vep,q,y)&\leq\frac{\int_{-1}^1 (1-t^2)^{-1/2}\d t}{\int_{-1}^1 (1-t^2)^{-1/2}\cos(\eta_\vep \sqrt{2\lv\|\varphi\|_\infty }t)\d t} \oS(q,\eta_\vep)+ \frac{C(M,\lambda)}{\eta_\vep^{1/2}}\oS(q,\eta_\vep).
\end{split}
\end{align}

\noindent {\bf Step 3.} Putting together \eqref{lem:main:c:bdd1} and \eqref{lem:main:c:bdd2} yields the bounds in \eqref{lem:main:c:bdd}. The proof is complete.
\end{proof}

\subsubsection{Radialization of the potential}
The next step is to reduce the potential $\varphi$ in $P^\beta_{\vep;t}f(x)$ to a radially symmetric function.
The method considers an alternative aspect of $\S^\beta_\vep(q,y)$ by the function
\begin{align}\label{def:Fvep}
F_\vep(y,t)\defeq\lv \E_{\vep y}[\e^{A^o_\vep(t)}].
\end{align}
By \eqref{diff:forward},
\begin{align}\label{def:Fvep:Lap}
\mathcal LF_{\vep,y}(q)\defeq \int_0^\infty q\e^{-qt}F_\vep(y,t) \d t=\lv +\S^\beta_\vep(q,y).
\end{align}

Let us state the required approximation. Define the {\bf logarithmic energy} of $f$ by
\begin{align}\label{def:Ef}
\mathcal E(f)(y)\defeq\frac{1}{\pi}f(y)\int_{\R^2}f(z)\log \frac{1}{|y-z|}\d z. 
\end{align}
Since $\hvarphi$ is bounded and has a support in $0\leq |y|\leq M_\varphi$, $\mathcal E(\hvarphi)(y)$ has a support in the same set and 
\[
\|\mathcal E(\hvarphi)\|_\infty\leq \|\varphi\|_\infty^2\cdot\sup_{y:0\leq |y|\leq M_\varphi}\int_{|z|\leq M_\varphi}\big|\log |y-z|\big|\d z=C(\|\varphi\|_\infty,M_\varphi).
\]
We choose the approximation to be the following exponential functional of $\BES^2$:
\begin{align}\label{def:Gvep}
G_\vep(|y|,t)=\lv \E_{\vep |y|}^\rho\left[\exp\left\{\lv \int_0^t\oovarphi_\vep (\rho_r)\d r\right\}\right],
\end{align}
where, with the uniform distribution $\u$ on $\T$ and $\R^2$ identified as the complex plane, 
\begin{align}\label{def:oovarphi}
\oovarphi_\vep(|y|)&\defeq \ovarphi_\vep(|y|)+ \lv\int_{\T}\vep^{-2}\mathcal E(\hvarphi)(\vep^{-1}|y|\e^{\i \theta})\u(\d \theta).
\end{align}
We also set 
\begin{align}
 \mathcal LG_{\vep,|y|}(q)&\defeq\int_0^\infty q\e^{-qt}G_\vep(|y|,t)\d t,\label{Lap:FG}\\
\oovarphi\uvep(|y|)&\defeq \ovarphi(|y|)+ \lv\oevarphi(|y|),\quad \mbox{where}\quad \oevarphi(|y|)=\int_{\T}\mathcal E(\hvarphi)(|y|\e^{\i \theta})\u(\d \theta).
\label{def:varphi0}
\end{align}
Note that by radial symmetry, we can replace $\E^\rho_{\vep|y|}$ in \eqref{def:Gvep} with $\E^W_{\vep y}$ or $\E^{W}_{\mathbf u_{\vep |y|}}$, and we have an important positivity of $\oovarphi\uvep$ that will be used in Section~\ref{sec:radial}: for $\ovp=\vep(\|\varphi\|_\infty,M_\varphi)\in (0,\ov)$,
\begin{align}\label{oovarphi:positive}
\oovarphi\uvep(y)\geq 0,\quad \forall\;y\in \R^2,\;\vep\in (0,\ovp).
\end{align}
The property in \eqref{oovarphi:positive} follows from the definition of $\oovarphi\uvep$ upon writing out its correction term:
\[
\lv\oevarphi=\lv\int_{\T}\frac{1}{\pi}[\varphi(|y|\e^{\i\theta})-\ovarphi(|y|)]\int_{\R^2}[\varphi(z)-\ovarphi(|z|)]\log \frac{1}{|y-z|}\d z\u (\d \theta)
\]
and bounding the $\d z$-integral on the right-hand side over all $y\in \supp(\varphi)$. 

\begin{prop}\label{prop:replaceangle}
For all $M\in [M_\varphi,\infty)$, $q\in (q(\|\varphi\|,\lambda),\infty)$, and $\vep\in (0,\vep(\|\varphi\|_\infty,M_\varphi,M,q,\lambda)\wedge \ov)$, 
\begin{align}\label{ineq:replaceangle}
\sup_{y:0<|y|\leq M}|\mathcal LF_{\vep,y}(q)-\mathcal LG_{\vep,|y|}(q)|\leq C(\|\varphi\|_\infty,M_\varphi,M,q,\lambda)\lv.
\end{align}
\end{prop}

For the proof, consider the following decompositions according to \eqref{def:Jdec}:
\begin{align}
\oF_\vep(|y|,t)&=\lv \E_{\u_{\vep|y|}}[\e^{A^o_\vep(t)}],\quad \wF_\vep(y,t)=F_\vep(y,t)-\oF_\vep(|y|,t),\label{def:varF}\\
 \mathcal L\oF_{\vep,|y|}(q)&=\int_0^\infty \d t q\e^{-qt}\oF_\vep(|y|,t),\quad 
\mathcal L\widehat{F}_{\vep,y}(q)=\mathcal LF_{\vep,y}(q)-\mathcal L\oF_{\vep,|y|}(q).\label{def:varFLap}
\end{align}
By \eqref{eq:FK1}, the following expansions hold:
\begin{align}
\mathcal L F_{\vep,y}(q)&=\lv+\int_{\R^2}\d z\varphi(z)\left(\lv \int_0^\infty \d t\e^{-qt}P_t(\vep y,\vep z)\right)\mathcal LF_{\vep,z}(q),\label{Lap:F}\\
\mathcal L\oF_{\vep,|y|}(q)&=\lv+\int_{\T}\u (\d \theta)  \int_{\R^2}\d z\varphi(z)\left(\lv\int_0^\infty \d t \e^{-qt}P_t(\vep |y|\e^{\i \theta},\vep z)\right)\mathcal LF_{\vep,z}(q),\label{Lap:oF}\\
\mathcal LG_{\vep,|y|}(q)&=\lv+\int_{\T}\u (\d \theta)  \int_{\R^2}\d z\varphi(z)\left(\lv\int_0^\infty \d t \e^{-qt}P_t(\vep |y|\e^{\i \theta},\vep z)\right)\mathcal LG_{\vep,|z|}(q).\label{Lap:G}
\end{align}
The goal for the proof of Proposition~\ref{prop:replaceangle} is to extract an approximate integral equation for $\mathcal L\oF_{\vep,|y|}(q)$ from \eqref{Lap:oF} that takes the same form of \eqref{Lap:G}.  

\begin{lem}\label{lem:hFsmall}
For all $M\in (0,\infty)$, $q\in (q(\|\varphi\|,\lambda),\infty)$, and $\vep\in (0,\vep(\|\varphi\|_\infty,M_\varphi,M,q,\lambda)\wedge \ov)$,
\begin{align}\label{hF:bdd}
\sup_{y:0<|y|\leq M}|\mathcal L\wF_{\vep,y}(q)|\leq  C(\|\varphi\|_\infty,M_\varphi,M,q,\lambda)\lv .
\end{align}
\end{lem}
\begin{proof}
For this proof and those of the lemmas for Proposition~\ref{prop:replaceangle} presented below, we use the condition 
\begin{align}\label{ass:radvepy}
\vep\in(0,\ov):\vep q^{1/2}(M+2M_\varphi)/2^{1/2}\leq 1\quad \&\quad  y:|y|\leq M.
\end{align}

Expand the resolvent densities in \eqref{Lap:F} and \eqref{Lap:oF} by \eqref{asymp:gauss}. Then under \eqref{ass:radvepy}, 
\begin{align*}
&\eqspace\int_{\R^2}\d z\varphi(z)\left(\lv \int_0^\infty \d t\e^{-qt}P_t(\vep y,\vep z)\right)\mathcal LF_{\vep,z}(q)\\
&=\frac{\lv}{\pi} \int_{\R^2}\d z\varphi(z)\left(\log \vep^{-1}+\log \frac{2^{1/2}}{q^{1/2}|y-z|}-\EM+\mathcal O\left(\frac{\vep^2q|y-z|^2}{2} \log\frac{2^{1/2}} {\vep q^{1/2}|y-z|} \right)\right)\mathcal LF_{\vep,z}(q).
\end{align*}
Take the differences of both sides of \eqref{Lap:F} and \eqref{Lap:oF} by using the foregoing display. Hence, under \eqref{ass:radvepy}, we have
\[
|\mathcal L \widehat{F}_{\vep,y}(q)|\leq  C(\|\varphi\|_\infty,M_\varphi,M,q)\lv\cdot \sup_{z:0<|z|\leq M_\varphi}|\mathcal LF_{\vep,z}(q)|.
\]
The required bound now follows from the foregoing inequality and Proposition~\ref{prop:bounds}.
\end{proof}

Next, we proceed with a key decomposition of $\mathcal L\oF_{\vep,|y|}(q)$ from \eqref{Lap:oF} as follows:
\begin{align}\label{dec:F}
\begin{split}
\mathcal L\oF_{\vep,|y|}(q)&=\lv +\I_{\ref{dec:F}}(\vep,q,y)+\II_{\ref{dec:F}}(\vep,q,y)+\III_{\ref{dec:F}}(\vep,q,y)\\
&\eqspace +{\rm IV}_{\ref{dec:F}}(\vep,q,y)+{\rm V}_{\ref{dec:F}}(\vep,q,y)+{\rm VI}_{\ref{dec:F}}(\vep,q,y).
\end{split}
\end{align}
where
\begin{align}
\I_{\ref{dec:F}}(\vep,q,y)&= \int_{\T}\u (\d \theta)\int_{\R^2}\d z\ovarphi(|z|)\left(\lv\int_0^\infty \d t \e^{-qt}P_t(\vep |y|\e^{\i \theta},\vep z)\right)\mathcal LF_{\vep,z}(q),\label{rad:term1}\\
\I\I_{\ref{dec:F}}(\vep,q,y)&=\int_{\T}\u (\d \theta)\int_{\R^2}\d z\hvarphi(z)\left(\lv\int_0^\infty \d t \e^{-qt}P_t(\vep |y|\e^{\i \theta},\vep z)\right)\lv,\label{rad:term2}\\
\begin{split}\label{rad:term3}
\I\I\I_{\ref{dec:F}}(\vep,q,y)&=\int_{\T}\u (\d \theta)\int_{\R^2}\d z\hvarphi(z)\left(\lv\int_0^\infty \d t \e^{-qt}P_t(\vep |y|\e^{\i \theta},\vep z)\right) \\
&\eqspace\times\int_{\R^2}\d z'\varphi(z')\left(\lv\int_0^\infty\d t' \e^{-qt'}P_{t'}(\vep z,\vep z')\right)\mathcal L\widehat{F}_{\vep,z'}(q),
\end{split}\\
\begin{split}\label{rad:term4}
{\rm IV}_{\ref{dec:F}}(\vep,q,y)
&= \int_{\T}\u (\d \theta)\int_{\R^2}\d z\hvarphi(z)\left(\lv\int_0^\infty \d t \e^{-qt}P_t(\vep |y|\e^{\i \theta},\vep z)\right) \\
&\eqspace\times\int_{\R^2}\d z'\ovarphi(|z'|)\left(\lv\int_0^\infty\d t' \e^{-qt'}P_{t'}(\vep z,\vep z')\right)\mathcal L\oF_{\vep,|z'|}(q),
\end{split}\\
\begin{split}\label{rad:term5}
{\rm V}_{\ref{dec:F}}(\vep,q,y)
&=\int_{\T}\u (\d \theta)\int_{\R^2}\d z\hvarphi(z)\left(\lv\int_0^\infty \d t \e^{-qt}P_t(\vep |y|\e^{\i \theta},\vep z)\right) \\
&\eqspace\times\int_{\R^2}\d z'\hvarphi(z')\left(\lv\int_0^\infty\d t' \e^{-qt'}P_{t'}(\vep z,\vep z')\right)\mathcal L\oF_{\vep,|z'|}(q).
\end{split}
\end{align}
To verify the decomposition in \eqref{dec:F}, the reader may note that according to $\varphi(z')=\ovarphi(|z'|)+\hvarphi(z')$, 
\begin{align*}
{\rm IV}_{\ref{dec:F}}(\vep,q,y)+{\rm V}_{\ref{dec:F}}(\vep,q,y)&=\int_{\T}\u (\d \theta)\int_{\R^2}\d z \widehat{\varphi}(z)\left(\lv\int_0^\infty \d t \e^{-qt}P_t(\vep |y|\e^{\i \theta},\vep z)\right) \\
&\eqspace\times\int_{\R^2}\d z'\varphi(z')\left(\lv\int_0^\infty\d t' \e^{-qt'}P_{t'}(\vep z,\vep z')\right)\mathcal L\oF_{\vep,|z'|}(q).
\end{align*}
So according to $\mathcal LF_{\vep,z'}(q)=\mathcal L\widehat{F}_{\vep,z'}(q)+\mathcal L\oF_{\vep,|z'|}(q)$,
\begin{align*}
&\eqspace {\rm III}_{\ref{dec:F}}(\vep,q,y)+{\rm IV}_{\ref{dec:F}}(\vep,q,y)+{\rm V}_{\ref{dec:F}}(\vep,q,y)\\
&=\int_{\T}\u (\d \theta)\int_{\R^2}\d z\hvarphi(z)\left(\lv\int_0^\infty \d t \e^{-qt}P_t(\vep |y|\e^{\i \theta},\vep z)\right)\\
&\eqspace\times \int_{\R^2}\d z'\varphi(z')\left(\lv\int_0^\infty\d t' \e^{-qt'}P_{t'}(\vep z,\vep z')\right)\mathcal LF_{\vep,z'}(q).
\end{align*}
By the last equality, applying \eqref{Lap:F}, with $(y,z)$ replaced by $(z,z')$ there, to $\mathcal LF_{\vep,z}(q)$  shows that
\begin{align*}
&\eqspace {\rm II}_{\ref{dec:F}}(\vep,q,y)+{\rm III}_{\ref{dec:F}}(\vep,q,y)+{\rm IV}_{\ref{dec:F}}(\vep,q,y)+{\rm V}_{\ref{dec:F}}(\vep,q,y)\\
&=\int_{\T}\u (\d \theta)\int_{\R^2}\d z\hvarphi(z)\left(\lv\int_0^\infty \d t \e^{-qt}P_t(\vep |y|\e^{\i \theta},\vep z)\right)\mathcal LF_{\vep,z}(q).
\end{align*}
Finally, by using $\varphi(z)=\ovarphi(|z|)+\hvarphi(z)$ again, we obtain from the last equality that the right-hand side of \eqref{dec:F} coincides with the right-hand side of \eqref{Lap:oF}.  

In the next three lemmas, we estimate the five terms defined in \eqref{rad:term1}--\eqref{rad:term5}. 

\begin{lem}\label{lem:rad1}
For all $\vep\in (0,\ov)$, $q\in (0,\infty)$, and $y\in \R^2$,  
\begin{align}\label{eq:rad1}
\I_{\ref{dec:F}}(\vep,q,y)=\int_{\R^2}\d z\ovarphi(|z|)\left(\lv\int_0^\infty \d t \e^{-qt}P_t(\vep |y|,\vep z)\right)\mathcal L\oF_{\vep,|z|}(q).
\end{align}
\end{lem}
\begin{proof}
By the definition in \eqref{rad:term1} and the skew-product representation in \eqref{def:skew},
\begin{align*}
\I_{\ref{dec:F}}(\vep,q,y)&= \lv\int_0^\infty \d t \e^{-qt}\E^W_{\u_{\vep |y|}}\left[\vep^{-2}\ovarphi(\vep^{-1}\rho_t)\mathcal LF_{\vep,\vep^{-1}\rho_t\exp\big\{\i\dgamma_{\int_0^t \d v/\rho_v^2}\big\}}(q)\right]\\
&= \lv\int_0^\infty \d t \e^{-qt}\E_{\vep |y|}^\rho\left[\vep^{-2}\ovarphi(\vep^{-1}\rho_t)\int_\T \u(\d \theta)\mathcal LF_{\vep,\vep^{-1}\rho_t\e^{\i \theta}}(q)\right]\\
&= \lv\int_0^\infty \d t \e^{-qt}\E_{\vep |y|}^\rho\left[\vep^{-2}\ovarphi(\vep^{-1}\rho_t)\mathcal L\oF_{\vep,\vep^{-1}\rho_t}(q)\right],
\end{align*}
where the second equality follows from the independence between $(\rho_t)$ and $(\gamma_t)$ and the stationarity of $(\dgamma_t)$ under the initial condition $\u$, and the last equality follows from the definitions for $\oF_\vep$ in \eqref{def:varF} and \eqref{def:varFLap}. Rewriting the last equality in terms of the Brownian transition kernels $P_t(x,y)$ gives the required identity in \eqref{eq:rad1}. 
 \end{proof}

\begin{lem}\label{lem:rad234}
Fix $M\in (0,\infty)$ and $q\in (q(\|\varphi\|,\lambda),\infty)$. Then for $|y|\leq M$, 
\begin{align}
\I\I_{\ref{dec:F}}(\vep,q,y)&=0,\quad \forall\;\vep\in (0,\ov);\label{est:2}\\
|\I\I\I_{\ref{dec:F}}(\vep,q,y)|&\leq  C(\|\varphi\|_\infty,M_\varphi,M,q,\lambda)\lv ^2,\quad\forall \;\vep\in (0,\vep(\|\varphi\|_\infty, M_\varphi,M,q,\lambda)\wedge \ov);\label{est:3}\\
{\rm IV}_{\ref{dec:F}}(\vep,q,y)&=0,\quad \forall\; \vep\in (0,\ov). \label{est:4}
\end{align}
\end{lem}
\begin{proof}
Both of the proofs of \eqref{est:2} and \eqref{est:4} use the mean-zero property:
\begin{align}\label{mean-zero}
\int_{\R^2}\hvarphi(z)H(|z|)\d z=0
\end{align}
for any radially symmetric function $H(|z|)$.
For $\I\I_{\ref{dec:F}}(\vep,q,y)$, choose $H(|z|)$ to be 
\begin{align*}
\int_{\T}\u (\d \theta)\left(\lv\int_0^\infty \d t \e^{-qt}P_t(\vep |y|\e^{\i \theta},\vep z)\right)=\int_{\T}\u (\d \theta)\left(\lv\int_0^\infty \d t \e^{-qt}P_t(\vep |y|,\vep |z|\e^{-\i\theta})\right)
\end{align*}
by the translation invariance of $\u(\d \theta)$. For ${\rm IV}_{\ref{dec:F}}(\vep,q,y)$, choose $H(|z|)$ to be 
\begin{align*}
&\eqspace\int_\T \u(\d \theta)\left(\lv\int_0^\infty \d t \e^{-qt}P_t(\vep |y|\e^{\i \theta},\vep z)\right)\int_{\R^2}\d z'\ovarphi(|z'|)\left(\lv\int_0^\infty\d t' \e^{-qt'}P_{t'}(\vep z,\vep z')\right)\mathcal L\oF_{\vep,|z'|}(q)\\
&=\int_\T \u(\d \theta)\left(\lv\int_0^\infty \d t \e^{-qt}P_t(\vep |y|,\vep |z|\e^{-\i \theta})\right)\\
&\eqspace \times \int_0^\infty \d r'r'\int_\T \d \theta'\ovarphi(r')\left(\lv\int_0^\infty\d t' \e^{-qt'}P_{t'}(\vep |z|,\vep r'\e^{\i \theta'})\right)\mathcal L\oF_{\vep,r'}(q),
\end{align*}
where we have used the translation invariance of $\d\theta'$ and then the same property of $\u (\d \theta)$.

To get \eqref{est:3}, note that under \eqref{ass:radvepy},
\eqref{asymp:gauss}  gives 
\begin{align*}
&\eqspace\I\I\I_{\ref{dec:F}}(\vep,q,y)\\
&=\int_{\T}\u (\d \theta)\int_{\R^2}\d z\hvarphi(z) \frac{\lv}{\pi} \left(\log \vep^{-1}+\log \frac{2^{1/2}}{q^{1/2}|y\e^{\i \theta}-z|}+ \mathcal O\left(\frac{\vep^2q|y\e^{\i \theta}-z|^2}{2} 
\log\frac{2^{1/2}}{\vep q^{1/2}|y\e^{\i \theta}-z|}\right)
\right) \\
&\eqspace\times\int_{\R^2}\d z'\varphi(z')\frac{ \lv}{\pi} \left(\log \vep^{-1}+\log \frac{2^{1/2}}{q^{1/2}|z-z'|}+ \mathcal O\left(\frac{\vep^2q|z-z'|^2}{2} 
\log\frac{2^{1/2}}{\vep q^{1/2}|z-z'|}\right)
\right) \mathcal L\widehat{F}_{\vep,z'}(q).
\end{align*}
For the right-hand side, note that under \eqref{ass:radvepy},
\begin{align*}
&\int_{\T}\u (\d \theta)\int_{\R^2}\d z\hvarphi(z)  \left(\frac{\lv}{\pi}\log \vep^{-1}\right)\times \int_{\R^2}\d z'\varphi(z') \left(\frac{\lv}{\pi}\log \vep^{-1}\right) \mathcal L\widehat{F}_{\vep,z'}(q)=0,\\
&\left|\int_{\T}\u (\d \theta)\int_{\R^2}\d z\hvarphi(z)\left(\frac{\lv}{\pi}\log \frac{2^{1/2}}{q^{1/2}|y\e^{\i \theta}-z|}\right)\times \int_{\R^2}\d z'\varphi(z') \left(\frac{\lv}{\pi}\log \vep^{-1}\right) \mathcal L\widehat{F}_{\vep,z'}(q)\right|\\
&\leq C(\|\varphi\|_\infty,M_\varphi,M,q,\lambda)\lv \sup_{z':0<|z'|\leq M_\varphi}|\mathcal L\widehat{F}_{\vep,z'}(q)|,\\
&\left|\int_{\T}\u (\d \theta)\int_{\R^2}\d z\hvarphi(z)  \left(\frac{\lv}{\pi}\log \vep^{-1}\right)\times \int_{\R^2}\d z'\varphi(z')  \left(\frac{\lv}{\pi}\log \frac{2^{1/2}}{q^{1/2}|z-z'|}\right) \mathcal L\widehat{F}_{\vep,z'}(q)\right|\\
&\leq C(\|\varphi\|_\infty,M_\varphi,M,q,\lambda)\lv \sup_{z':0<|z'|\leq M_\varphi}|\mathcal L\widehat{F}_{\vep,z'}(q)|.
\end{align*}
Here, the first line uses the property $\int \hvarphi=0$. The other integrals from expanding the product in the last expression for $\I\I\I_{\ref{dec:F}}(\vep,q,y)$ can be bounded in the same way as the last two inequalities. 
We have proved \eqref{est:3} by this observation, the last two displays and \eqref{hF:bdd}.
\end{proof}

The next estimate will be used together with \eqref{eq:rad1} to get the required approximate integral equality for $\mathcal L \oF_{\vep,|y|}(q)$.

\begin{lem}\label{lem:rad5}
Fix $q,M\in (0,\infty)$. Then for $|y|\leq M$ and $\vep\in (0,\vep(\|\varphi\|_\infty, M_\varphi,M,q,\lambda)\wedge \ov)$,
\begin{align}
\begin{split}\label{rad5:est}
&\left|{\rm V}_{\ref{dec:F}}(\vep,q,y)-\int_{\R^2}\d z\big(\oovarphi\uvep(|z|)-\ovarphi(|z|)\big)\left(\lv\int_0^\infty \d t \e^{-qt}P_t(\vep |y|,\vep z)\right)\mathcal L\oF_{\vep,|z|}(q)\right|\\
&\leq C(\|\varphi\|_\infty,M_\varphi,M,q,\lambda)\lv^2.
\end{split}
\end{align}
\end{lem}
\begin{proof}
The proof uses similar arguments for proving Lemma~\ref{lem:rad234}. Again,  under \eqref{ass:radvepy},
\eqref{asymp:gauss} gives
\begin{align*}
&\eqspace{\rm V}_{\ref{dec:F}}(\vep,q,y)\\
&=\int_{\T}\u (\d \theta)\int_{\R^2}\d z\hvarphi(z)\frac{\lv}{\pi} \left(\log \vep^{-1}+\log \frac{2^{1/2}}{q^{1/2}|y\e^{\i \theta}-z|}+ \mathcal O\left(\frac{\vep^2q|y\e^{\i \theta}-z|^2}{2} 
\log\frac{2^{1/2}}{\vep q^{1/2}|y\e^{\i \theta}-z|}\right)
\right) \\
&\eqspace\times\int_{\R^2}\d z'\hvarphi(z') \frac{\lv}{\pi} \left(\log \vep^{-1}+\log \frac{2^{1/2}}{q^{1/2}|z-z'|}+ \mathcal O\left(\frac{\vep^2q|z-z'|^2}{2} 
\log\frac{2^{1/2}}{\vep q^{1/2}|z-z'|}\right)
\right) \mathcal L\oF_{\vep,|z'|}(q).
\end{align*}
To proceed, expand the product in the right-hand side of the last equality. Any term that involves one of the $\mathcal O$'s or does not use $\log \vep^{-1}$ can be bounded by $C(\|\varphi\|_\infty,M_\varphi,M,q,\lambda)\lv^2$. 

For the remaining terms, it follows from \eqref{mean-zero} that
\begin{align}
&\int_{\T}\u (\d \theta)\int_{\R^2}\d z\hvarphi(z)\left(\frac{\lv}{\pi}\log \vep^{-1}\right)\times \int_{\R^2}\d z'\hvarphi(z') \left(\frac{\lv}{\pi}\log \vep^{-1}\right) \mathcal L\oF_{\vep,|z'|}(q)=0,\label{rad:5-00}\\
&\int_{\T}\u (\d \theta)\int_{\R^2}\d z\hvarphi(z) \left(\frac{ \lv}{\pi}\log \frac{2^{1/2}}{q^{1/2}|y\e^{\i \theta}-z|}\right)\times \int_{\R^2}\d z'\hvarphi(z') \left(\frac{\lv}{\pi}\log \vep^{-1}\right) \mathcal L\oF_{\vep,|z'|}(q)=0,\label{rad:5-000}
\end{align}
and
\begin{align}
&\eqspace \int_{\T}\u (\d \theta)\int_{\R^2}\d z\hvarphi(z) \left(\frac{\lv}{\pi}\log \vep^{-1}\right)\times \int_{\R^2}\d z'\hvarphi(z') \left(\frac{\lv}{\pi}\log \frac{2^{1/2}}{q^{1/2}|z-z'|}\right) \mathcal L\oF_{\vep,|z'|}(q)\label{rad:5-0}\\
&= \int_{\R^2}\d z'\left(\frac{\lv}{\pi}\log \vep^{-1}\right)\lv\left(\frac{1}{\pi}\hvarphi(z')\int_{\R^2}\d z \hvarphi(z)\log \frac{2^{1/2}}{q^{1/2}|z-z'|}\right) \mathcal L\oF_{\vep,|z'|}(q)\notag\\
&=\int_{\R^2}\d z' \left(\frac{\lv}{\pi}\log \vep^{-1}\right) \lv\mathcal E(\widehat{\varphi})(z') \mathcal L\oF_{\vep,|z'|}(q)\notag\\
&=\int_{\R^2}\d z' \left(\frac{\lv}{\pi}\log \vep^{-1}\right) \big(\oovarphi\uvep(|z'|)-\ovarphi(|z'|)\big)\mathcal L\oF_{\vep,|z'|}(q).\label{rad:5-1}
\end{align}
Here, the second equality follows from the definition of $\mathcal E(\hvarphi)$ in \eqref{def:Ef} and the mean-zero property \eqref{mean-zero}, and the last equality uses the mean-zero property to remove the non-radial part of $\mathcal E(\widehat{\varphi})(z') $ and then the definition of $\oovarphi$ in \eqref{def:varphi0}. On the other hand, note that \eqref{asymp:gauss} gives
\begin{align}\label{rad:5-2}
\left|\frac{\lv}{\pi}\log \vep^{-1}-\lv\int_0^\infty \d t \e^{-qt}P_t(\vep |y|,\vep z')\right|\leq C(\|\varphi\|_\infty,M_\varphi,M,q,\lambda)\lv 
\end{align}
for all $0<|y|\leq M_\varphi$ and $|z'|\leq M_\varphi$. We obtain from \eqref{rad:5-1} and \eqref{rad:5-2} that under \eqref{ass:radvepy}, the main term for \eqref{rad:5-0} satisfies the estimate 
\begin{align*}
&\Bigg| \int_{\T}\u (\d \theta)\int_{\R^2}\d z\hvarphi(z) \left(\frac{\lv}{\pi}\log \vep^{-1}\right)\times \int_{\R^2}\d z'\hvarphi(z') \left(\frac{\lv}{\pi}\log \frac{2^{1/2}}{q^{1/2}|z-z'|}\right) \mathcal L\oF_{\vep,|z'|}(q)\\
&-\int_{\R^2}\d z'\left(\lv\int_0^\infty \d t \e^{-qt}P_t(\vep |y|,\vep z')\right)\big(\oovarphi\uvep(|z'|)-\ovarphi(|z'|)\big)\mathcal L\oF_{\vep,|z'|}(q)\Bigg|\leq C(\|\varphi\|_\infty,M_\varphi,M,q,\lambda)\lv^2.
\end{align*}

The required bound now follows by applying the foregoing inequality, \eqref{rad:5-00}, \eqref{rad:5-000} and the observation in the first paragraph of this proof to the expression of ${\rm V}_{\ref{dec:F}}(\vep,q,y)$ there. 
\end{proof}

\begin{proof}[End of the proof of Proposition~\ref{prop:replaceangle}]
By Lemma~\ref{lem:hFsmall}, it remains to show that 
\begin{align}\label{Delta:goal}
\Delta_\vep\defeq \sup_{y:0<|y|\leq M}|\mathcal L\oF_{\vep,|y|}(q)-\mathcal LG_{\vep,|y|}(q)|\leq C(\|\varphi\|_\infty,M_\varphi,M,q,\lambda)\lv.
\end{align}
Recall \eqref{dec:F} and the estimates in Lemmas~\ref{lem:rad1}--\ref{lem:rad5}. Since $M\geq M_\varphi$, it follows from \eqref{Lap:oF}, \eqref{Lap:G}, and \eqref{asymp:gauss} that  all $\vep\in (0,\vep(\|\varphi\|_\infty, M_\varphi,M,q,\lambda)\wedge \ov)$,
\begin{align}
\Delta_\vep&\leq \sup_{y:<|y|\leq M}\left(\int_{\T}\u (\d \theta)  \int_{\R^2}\d z\varphi(z)\lv\int_0^\infty \d t \e^{-qt}P_t(\vep |y|\e^{\i \theta},\vep z)\right)\Delta_\vep+C(\|\varphi\|_\infty,M_\varphi,M,q,\lambda)\lv^2\notag\\
&\leq \left(1-\frac{C'(\|\varphi\|_\infty,M_\varphi,M,q,\lambda)}{\log \vep^{-1}}\right)\Delta_\vep+C(\|\varphi\|_\infty,M_\varphi,M,q,\lambda)\lv^2.\label{Deltavep:bdd}
\end{align}
Note that in the last inequality, we have also used the following estimate:
\[
\left| \frac{\lv}{\pi}\log \vep^{-1}\int \varphi-1\right|\leq \frac{C(\lambda)}{\log \vep^{-1}},\quad \forall\;\vep\in (0,\vep(\lambda)).
\]
The required bound in \eqref{Delta:goal} for $\Delta_\vep$ now follows from \eqref{Deltavep:bdd}. The proof is complete.
\end{proof}

\subsubsection{Cutoffs of potentials}
The last reduction considers a cutoff of potentials, which will be used in the next section due to the use of some  exponential functionals of planar Brownian motion. We continue to work with the reduction from Proposition~\ref{prop:replaceangle}.

\begin{prop}\label{prop:replace1}
Let $\chi_\vep$ be $[0,1]$-valued functions such that  for all $a>0$, $\lim_{\vep\to 0}\chi_\vep(a)=1$.
Then for all $M\in [M_\varphi,\infty)$, $q\in (q(\|\varphi\|_\infty,\lambda),\infty)$ and $\vep\in (0,\vep(\|\varphi\|_\infty,M_\varphi,q,\lambda)\wedge \ov)$,
\begin{align}\label{lim:main:a}
\begin{split}
&\sup_{y:0<|y|\leq M}\lv^2\int_0^\infty \d t\e^{-q t}\E_{\vep |y|}^\rho\left[\exp\left\{\lv \int_0^t\d r\oovarphi_{\vep} (\rho_r)\right\}\oovarphi_{\vep}(\rho_{t})\big(1-\chi_\vep(\vep^{-1}\rho_t)\big)\right]\\
&\eqspace\leq C(\|\varphi\|_\infty,M_\varphi,M,q,\lambda)\left(\frac{\lv\log \vep^{-1}}{\pi}\int_{\R^2}\d z|\oovarphi\uvep(|z|)|\widetilde{\chi}_\vep(|z|)+\lv\right).
\end{split}
\end{align}
\end{prop}
\begin{proof}
Write $\widetilde{\chi}_\vep=1-\chi_\vep$. By an analogue of \eqref{eq:FK1} and a change of variables, we get
\begin{align*}
&\eqspace\lv^2\int_0^\infty \d t\e^{-q t}\E^\rho_{\vep |y|}\left[\exp\left\{\lv \int_0^t\d r\oovarphi_{\vep} (\rho_r)\right\}  \oovarphi_\vep(\rho_t)\widetilde{\chi}_\vep(\vep^{-1}\rho_t)\right]\\
&=\lv^2\int_0^\infty \d t\e^{-q t}\int_{\R^2}\d zP_t(\vep y,\vep z) \oovarphi\uvep(|z|)\widetilde{\chi}_\vep( |z|)+ \int_{\R^2}\d z\left(\lv\int_0^\infty \d t\e^{-qt}P_t(\vep y,\vep z)\right)\oovarphi\uvep(z) \\
&\eqspace\times \lv ^2\int_0^\infty \d t \e^{-q t}\E^\rho_{\vep |z|}\left[\exp\left\{\lv \int_0^t\d r\oovarphi_{\vep} (\rho_r)\right\}  \oovarphi_\vep(\rho_t)\widetilde{\chi}_\vep(\vep^{-1}\rho_t)\right].
\end{align*} 
It follows that 
\[
L_\vep(q)\defeq\sup_{y:0<|y|<\infty}\left| \lv ^2\int_0^\infty \d t \e^{-q t}\E^\rho_{\vep |y|}\left[\exp\left\{\lv \int_0^t\d r\oovarphi_{\vep} (\rho_r)\right\}  \oovarphi_\vep(\rho_t)\widetilde{\chi}_\vep(\vep^{-1}\rho_t)\right]\right|
\]
satisfies the inequality
\begin{align}\label{LAB}
L_\vep(q)\leq A_\vep(q)+B_\vep(q)L_\vep(q),
\end{align}
where
\begin{align*}
A_\vep(q)&\defeq \sup_{y:0<|y|<\infty}\left|\lv^2\int_0^\infty \d t\e^{-q t}\int_{\R^2}\d zP_t(\vep y,\vep z) \oovarphi\uvep(|z|)\widetilde{\chi}_\vep( |z|)\right|,\\
B_\vep(q)&\defeq \int_0^\infty \d t\e^{-qt} \lv \sup_{y\in \R^2}\int_{\R^2}\d z P_t(\vep y,\vep z)|\oovarphi\uvep(|z|)|.
\end{align*}

To bound $A_\vep(q)$, we use \eqref{asymp:gauss} under condition \eqref{ass:radvepy} as in the proof of Lemma~\ref{lem:hFsmall}. Hence, for all $\vep\in (0,\vep(M_\varphi,M,q)\wedge \ov)$,
\begin{align}
\begin{split}\label{Cutoff:A}
A_\vep(q)&\leq \frac{\lv^2\log \vep^{-1}}{\pi}\int_{\R^2}\d z|\oovarphi\uvep(|z|)|\widetilde{\chi}_\vep(|z|)+C(\|\varphi\|_\infty,M_\varphi,M,q,\lambda)\lv^2.
\end{split}
\end{align}
Additionally, since $\int_{\R^2}|\oovarphi{}^\vep(|z|)|\d z
\leq 1/2+\lv C(\|\varphi\|_\infty,M_\varphi)$,
the proof of \eqref{Lap:b-bdd} shows that 
\begin{align}\label{Cutoff:B}
\begin{split}
B_\vep(q)&\leq 1-\frac{C(\|\varphi\|_\infty,q,\lambda)}{\log \vep^{-1}},\quad \forall\; q\in (q(\|\varphi\|_\infty,\lambda),\infty),\;  \forall\; \vep\in (0,\vep(\|\varphi\|_\infty,M_\varphi,q,\lambda)\wedge \ov).
\end{split}
\end{align}
An application of \eqref{Cutoff:A} and \eqref{Cutoff:B} to \eqref{LAB} as in \eqref{Deltavep:bdd}  shows \eqref{lim:main:a}. The proof is complete.
\end{proof}

Let us summarize the reductions obtained so far in Propositions~\ref{prop:main:c}, \ref{prop:replaceangle} and \ref{prop:replace1}. With $\chi_\vep$ as in Proposition~\ref{prop:replace1}, set
\begin{align}\label{main:aa}
\ooS(q,|y|)&\defeq \lv^2\int_0^\infty \e^{-\vep^2qt}\E_{\vep |y|}^\rho[\e^{\ooA_\vep(t)} \oovarphi\uvep(\rho_{t})\chi_\vep(\rho_t)]\d t,\quad\mbox{where } 
\ooA_\vep(t)\defeq \lv \int_0^t \oovarphi\uvep(\rho_r)\d r.
\end{align}
Then for all $M\in [M_\varphi,\infty)$ and $q\in (q(\|\infty\|_\infty,\lambda),\infty)$, it holds that 
\begin{align}\label{SS:approx}
\lim_{\vep\to 0}\sup_{y:0<|y|\leq M}|\S^\beta_\vep(q,y)-\ooS(q,|y|)|=0. 
\end{align}
We specify the choice of $\chi_\vep$ in Section~\ref{sec:radial}.

\subsection{Semigroups for radially symmetric short-range potentials}\label{sec:radial}
In this section, we complete the proof of Theorem~\ref{thm:main1} by proving the asymptotic behavior of $\ooS(q,|y|)$ defined in \eqref{main:aa}. We begin by summarizing some more basic properties of $\BES^2$. Write $(Q_t)$ for the probability semigroup of $\BES^2$. The densities (with respect to the Lebesgue measure) are given by 
\begin{align}\label{def:density}
Q_t(a,b)\defeq \frac{b}{t}\exp\left(-\frac{a^2+b^2}{2t}\right)I_0\left(\frac{ab}{t}\right),\quad \;\forall\; a,b,t\in (0,\infty),
\end{align}
with an extension to $a=0$ by continuity \cite[p.446]{RY}, where $I_0$ is the modified Bessel function of the first kind \eqref{def:I}. 
Recall that we work with $\rho_t=|W_t|$ as a version of $\BES^2$. 
Since $\rho_0\neq 0$ by assumption and $0$ is polar \cite[p.442]{RY}, It\^{o}'s formula implies
\begin{align}\label{BES:SDE}
\d \rho_t=(2\rho_t)^{-1}\d t+\d \overline{W}_t,
\end{align}
where $\overline{W}$ is a one-dimensional standard Brownian motion given by $\d \overline{W}_t=\sum_j W_t^j\rho_t^{-1}\d W^j_t$ provided that $W=(W^1,W^2)$.  We write $\P_a=\P^\rho_a$ and $\E_a=\E^\rho_a$ throughout this subsection.

We fix the normalization of the semimartingale local times $L^b_t=L^b_t(\rho)$ at levels $b\in (0,\infty)$ according to Tanaka's formula. The occupation times formula holds in the following form with probability one:
\begin{align}\label{otf}
\int_0^t g(\rho_s)\d s=\int_\R g(b) L^{b}_t\d b,\quad\forall\;g\in \B_+(\R),\;t\geq 0.
\end{align}
Moreover, since $\rho_t$ has a  probability density for all $t>0$. one can choose a jointly continuous two-parameter modification $\R\times \R_+\times (b,t)\mapsto L^b_t$. Since $|W_t|\in L^p(\P)$ for all $p\in [1,\infty)$, an application of Tanaka's formula and the Burkholder--Davis--Gundy inequality to \eqref{BES:SDE} yields
\begin{align}\label{Lt:Lp}
\sup_{b\in \R}\E_a[(L^b_t)^p]<\infty,\quad \forall\; p\in [1,\infty),\; t>0.
\end{align}
See \cite[Section~VI.1]{RY} for these properties. 
On the other hand, under $\P_b$ for $b\in (0,\infty)$, the Markovian local time at $b$ 
exists since $b$ is instantaneous and recurrent for  $(\rho_t)$
by the skew-product representation \eqref{def:skew}. To use these two notions of local times at the same time, note that  \eqref{otf} and \eqref{Lt:Lp} imply $\E_b[L^b_t]=\int_0^tQ_s(b,b)\d s>0$. Moreover, $t\mapsto L^b_t$ is $\P_b$-a.s. supported in $\{t;\rho_t=b\}$ \cite[Proposition~1.3 on p.222]{RY}. Hence, under $\P_b$ for $b\in(0,\infty)$, the strong Markov property of $(\rho_t)$ implies that $L^b$ is also the Markovian local time \cite[Proposition~IV.5 on p.111]{Bertoin}. The recurrence of $b$ gives $L^b_\infty=\infty$ \cite[p.114]{Bertoin}.

The application of $\BES^2$ in this subsection is conducted under the time scales of the inverse local times at levels $b$:
\[
\tau^b_t=\tau^b_t(\rho)\defeq \inf\{s\geq 0;L^b_s>t\},
\]
under $\P_a$ for $a\in(0,\infty)$.
It holds that $\tau^b_t\nearrow +\infty$ a.s. and $\tau^b_0=T_b$, where $T_b=T_b(\rho)$. Then given an additional point $\partial$ isolated from $C(\R_+,\R_+)$, the excursion process $(e^b_t;t>0)$ away from $b\in (0,\infty)$
for $(\rho_t)$ under $\P_b $  is a $C(\R_+,\R_+)\cup \{\partial\}$-valued point process defined as follows: for all $t>0$,
\[
e^b_t=(e^b_{t,s};s\geq 0)=\big(\rho_{\tau^b_{t-}+s\wedge (\tau^b_t-\tau^b_{t-})};s\geq 0\big),\quad \mbox{if }\tau^b_t-\tau^b_{t-}>0;\quad e^b_t=\partial\quad\mbox{ otherwise}.
\]
Whenever $e_t^b\neq \partial$, $e^b_{t,0}=b$, $e^b_{t,s}\neq b$ for all $s\in (0,\zeta)$, and $e^b_{t,s}=b$ for all $s\geq \zeta$, where $\zeta=\zeta(e^b_t)=\inf\{s\geq 0;e^b_{t,s}=b\}$ is the lifetime of the excursion $e^b_t$. We extend any function $F$ on $C(\R_+,\R_+)$ to $\{\partial\}$ by $F(\partial)=0$, so $\zeta(\partial)=0$. The process $(e^b_t;t>0)$ is a Poisson point process \cite[Theorem~10 on p.118]{Bertoin}. Its characteristic measure, denoted by $\N_b$, is the excursion measure. 
Since $\tau^b_t-\tau^b_{t-}$ is the lifetime of $e^b_t$, the Laplace exponent of the subordinator $(\tau^b_t)$ can be expressed as follows by the exponential formula for Poisson point processes \cite[p.8]{Bertoin}:
\begin{align}\label{Phib:Lap}
-\log\E_b[\e^{-\mu\tau^b_1}]= \Phi_b(\mu)\defeq\N_b(1-\e^{-\mu\zeta}),\quad \forall\; \mu\in (0,\infty).
\end{align}

We are ready to present the starting point to calculate $\ooS(q,|y|)$ defined in \eqref{main:aa}. Write $\E[X;\Gamma]=\E[X\1_\Gamma]$.

\begin{prop}
\label{prop:LT}
\noindent {\rm (1$\cc$)} For all $a\in (0,\infty)$ and nonnegative processes $F(t)$,
\begin{align}\label{eq:Lap_LT}
b\mapsto \E_a\left[\int_0^\infty F(t)\d L^b_t\right] =
\int_0^\infty \E_a[ F(\tau^b_t)]\d t
\end{align}
is the density of the measure $\B((0,\infty))\ni \Gamma\mapsto \int_0^\infty \E_a[F(t);\rho_t\in \Gamma]\d t$.
In particular,
\begin{align}\label{Phib:resolvent}
\int_0^\infty \e^{-\mu t}Q_t(b,b)\d t=\E_b\left[\int_0^\infty \e^{-\mu t}\d L^b_t\right]=\Phi_b(\mu)^{-1},\quad\forall\; \mu,b\in (0,\infty).
\end{align}

\noindent {\rm (2$\cc$)} Given $f\in \C_c(\R_+)$, set $A(t)=\int_0^t f(\rho_{s})\d s$ and
\begin{align}\label{def:Phibf}
\Phi_b(\mu,f)\defeq \N_b(1-\e^{-\mu\zeta +A(\zeta)})
\end{align}
with the extended notation $A(t)=\int_0^t f(\epsilon_r)\d r$ under $\N_b(\d \epsilon)$. 
 Assume that for a fixed interval $[g,d]\subset(0,\infty)$,  we can find $\mu\in (0,\infty)$ such that the following two conditions hold for all $b\in [g,d]$:
\begin{align}\label{cond:q}
\begin{split}
&\sup_{a\in \supp (f)\setminus\{0\}}\E_{a}[ \e^{-\mu T_b+ A(T_b)}]<\infty\quad \&\quad \Phi_b(\mu,f)\in (0,\infty),
\end{split}
\end{align}
where $T_b=T_b(\rho)$ is defined in \eqref{def:TX}. 
Then for all  $a\in(0,\infty)$,
\begin{align}
b\mapsto \frac{\E_a[ \e^{-\mu T_b+ A(T_b)}]}{\Phi_b(\mu,f)}\mbox{ is the density of }\B([g,d])\ni \Gamma\mapsto \int_0^\infty \E_a\big[\e^{-\mu t}\e^{A(t)};\rho_t\in \Gamma\big]\d t.\label{LT:Tb}
\end{align}
\end{prop}
 \begin{proof}
{\rm (1$\cc$)} Since $L^b_\infty=\infty$ $\P_a$-a.s. for $a,b\in (0,\infty)$, the change of variables formula for Stieltjes integrals \cite[Proposition~4.9 on p.8]{RY} immediately shows the equality in
\eqref{eq:Lap_LT}. To see the required density property, note that by the extended occupation times formula \cite[Exercise~1.15 on p.232]{RY}, 
\[
\int_0^\infty F(t)g(\rho_t)\d t=
\int_{0}^\infty  \left(\int_0^\infty F(t)\d L^{b}_t\right)g(b)\d b,\quad\forall\; g\in \B_+(\R_+).
\]
Taking expectations of both sides of the foregoing equality proves the required result. \medskip

\noindent (2$\cc$) For all $a,b\in (0,\infty)$, it holds that $\P_a$-a.s. $\tau^b_t=T_b+\tau^b_t\circ \theta_{T_b}$ for all $t\in (0,\infty)$, where $\theta_t$ is the shift operator. Hence, by the strong Markov property, we can write the second expectation in \eqref{eq:Lap_LT}  with $F(t)=\e^{-\mu t+A(t)}$ as 
\begin{align}
\int_0^\infty \E_a[ \e^{-\mu \tau^b_t+A(\tau^b_t)}]\d t
&= \E_a[ \e^{-\mu T_b+A(T_b)}]\int_0^\infty \E_b[ \e^{-\mu\tau^b_t+ A(\tau^b_t)}]\d t
  \label{LT:2}.
\end{align}
For the last integral, note that since $\int_0^\infty \1_{\{\rho_r=b\}}\d r=0$ a.s.,
\begin{align*}
-\mu \tau^b_t+ A(\tau^b_t)
=\sum_{r:0< r\leq t}\int_0^{\zeta(e^b_r)} [-\mu+f(e^b_{r,s})]\d s=\sum_{r:0< r\leq t}[-\mu \zeta(e^b_r)+A(e^b_r)].
\end{align*}

Now, we work with $\mu$ assumed to satisfy both conditions in \eqref{cond:q}.
The integrability $\N_b(|1-\e^{-\mu\zeta +A(\zeta)}|)<\infty$ for all $b\in [g,d]$ is proven later on in Proposition~\ref{prop:Phibf} by using the first condition in  \eqref{cond:q}. 
Hence, by the last equality and
the exponential formula for Poisson point processes \cite[p.8]{Bertoin}, 
\begin{align}\label{Psi:exp}
\E_b\big[ \e^{-\mu\tau^b_t+A(\tau^b_t)}\big]=\e^{-\Phi_{b}(\mu,f)t},\quad t\geq 0,
\end{align}
where $\Phi_b(\mu,f)$ is defined by \eqref{def:Phibf}. Since $\Phi_b(\mu,f)\in (0,\infty)$ by the second assumed condition in \eqref{cond:q}, the required property in \eqref{LT:Tb} follows from (1$\cc$) and \eqref{Psi:exp}. The proof is complete. 
\end{proof}

The functional defined in \eqref{def:Phibf} can be represented explicitly. 

\begin{prop}\label{prop:Phibf}
Given $f\in \C_c(\R_+)$, define $A(t)$ as in Proposition~\ref{prop:LT} {\rm (2$\cc$)}, and assume only the first condition in \eqref{cond:q}. Then for all $b\in [g,d]$, $\N_b(|1-\e^{-\mu\zeta +A(\zeta)}|)<\infty$ and 
\begin{align}\label{Phiblambda:resolvent}
 \Phi_b(\mu,f)
=\Phi_b(\mu)- \Phi_b(\mu)\E_b\left[\int_0^\infty  \e^{-\mu r} f(\rho_r)\E_{\rho_r}[\e^{-\mu T_b+A(T_b)}]\d r\right].
\end{align}
\end{prop}
\begin{proof}
For $f\geq 0$, we can write
\begin{align*}
\Phi_b(\mu,f)
=\N_b(1-\e^{-\mu \zeta})+
\N_b(\e^{-\mu \zeta}-\e^{-\mu \zeta+A(\zeta)})
=\Phi_b(\mu)-
\N_b\left(\e^{-\mu \zeta}\int_0^\zeta  f(\epsilon_r) \e^{\int_r^\zeta f(\epsilon_s)\d s}\d r\right),
\end{align*}
where the last equality follows from the definition of $\Phi_b$ and \eqref{taylor1}. The last term can be written as 
\begin{align}
 \N_b\left(\e^{-\mu \zeta}\int_0^\zeta  f(\epsilon_r) \e^{\int_r^\zeta f(\epsilon_s)\d s}\d r\right)
&=\N_b\left(\int_0^\zeta \e^{-\mu r}f(\epsilon_r)\E_{\epsilon_r}[\e^{-\mu T_b+A(T_b)}]\d r\right)\label{Psib:0}\\
&=\Phi_b(\mu) \E_b\left[\int_0^\infty  \e^{-\mu r} f(\rho_r)\E_{\rho_r}[\e^{-\mu T_b+A(T_b)}]\d r\right].\label{Psib:1}
\end{align}
Here, \eqref{Psib:0} follows from the simple Markov property at $r$ under the excursion measure \cite[Theorem~III.3.(28) in pp.102--103]{Blumenthal}, and \eqref{Psib:1} follows from the compensation formula in excursion theory \cite[Corollary~11 on p.119]{Bertoin} and the second equality in \eqref{Phib:resolvent}. Combining \eqref{Psib:1} and the display before shows that under the first condition in \eqref{cond:q}, $\N_b(|1-\e^{-\mu\zeta +A(\zeta)}|)<\infty$ and \eqref{Phiblambda:resolvent} holds. These properties extend to a general $f\in \C_c(\R_+)$ by considering $|f|$. The proof is complete.
\end{proof}

We apply Propositions~\ref{prop:LT} (2$\cc$) to $\ooS(q,|y|)$ in \eqref{main:aa} by taking
\begin{align}\label{setting:radial}
f=f_\vep=\lv \oovarphi\uvep\quad \&\quad \mu=\vep^2q,\quad \forall\;\vep\in (0,\ovp),
\end{align}
such that the positivity property \eqref{oovarphi:positive} holds. (Recall the function $\oovarphi\uvep$ defined in \eqref{def:varphi0}.) For $\chi_\vep$ to be specified near the end of this subsection, let $m_\vep\in(0,\infty)$ be such that $\supp(\chi_\vep)\subseteq \{y:|y|\geq m_\vep\}$. Then under the assumption of \eqref{cond:q} for $[g,d]=[m_\vep,M_\varphi]$, we can write
\begin{align}
\ooS(q,|y|)&=\lv^2\int_0^\infty \frac{\E_{|y|}[\e^{-\vep^2q T_b+\ooA_\vep(T_b)}]}{\Phi_b(\vep^2q, f_\vep)}\oovarphi\uvep(b)\chi_\vep(b)\d b,\label{eq2:main}
\end{align}
where $\Phi_b(\vep^2q, f_\vep)$ can be expanded as in \eqref{Phiblambda:resolvent}. Observe that in terms of \eqref{cond:q}, \eqref{Phiblambda:resolvent} and \eqref{eq2:main}, the first-passage functional 
\begin{align}\label{def:k}
\kappa_{\vep}^q(a,b)\defeq \E_{a}[\e^{-\vep^2q T_b+\ooA_\vep(T_b)}]
\end{align}
is the main object. Henceforth, the task to complete the proof of Theorem~\ref{thm:main1} is to verify the two conditions in \eqref{cond:q} for the choice in \eqref{setting:radial} and $[g,d]=[m_\vep,M_\varphi]$. For the following proofs, we set
\begin{align}\label{def:nubeta}
\nu_\vep=\|\oovarphi\uvep\|_\infty\lv\quad\& \quad \beta_\vep=(\log \log \vep^{-1})^{2},
\end{align}
and we turn to the explicit formulas of some exponential functionals of Brownian motions.

\begin{lem}[Verification of the first condition of \eqref{cond:q}]\label{lem:kappato1}
For any $q\in (0,\infty)$ and $\delta\in (0,1)$, 
\begin{align}
\begin{split}
&\frac{K_0(M_\varphi\sqrt{2\vep^2q})}{K_0(M_\varphi\e^{-\delta\frac{\pi}{2\sqrt{2\nu_\vep}}} \sqrt{2\vep^2q})}\leq \kappa^q_\vep(a,b)\leq \frac{1}{\cos(\pi\delta/2)},\\
&\hspace{4cm}\forall\; a,b:M_\varphi\e^{-\delta\frac{\pi}{2\sqrt{2\nu_\vep}}}\leq b\leq a\leq M_\varphi,\;\vep\in (0,\ovp),\label{ineq1:k0}
\end{split}
\end{align}
and we can find $\vep_{\ref{ineq2:k0}}=\vep_{\ref{ineq2:k0}}(M_\varphi,q,\lambda)\in (0,\ovp)$ such that $M_\varphi\sqrt{2\nu_\vep}<\pi/4$ and
\begin{align}
\begin{split}
\frac{1}{I_0(M_\varphi\sqrt{2\vep^2q})}&\leq \kappa^q_\vep(a,b)
\leq \frac{ \int_{-1}^1 (1-t^2)^{-1/2}\d t}
{\int_{-1}^1 (1-t^2)^{-1/2}\cos(M_\varphi\sqrt{2\nu_\vep}t)\d t},\\
&\hspace{1.05cm} \forall\;  a,b:0<a\leq b\leq M_\varphi,\;\vep\in (0,\vep_{\ref{ineq2:k0}}).\label{ineq2:k0}
\end{split}
\end{align}
Hence, under the choice of \eqref{setting:radial}, the first condition in \eqref{cond:q} holds for all $[g,d]$ such that $M_\varphi\e^{-\delta\frac{\pi}{2\sqrt{2\nu_\vep}}}\leq g< d\leq M_\varphi$. 
\end{lem}
\begin{proof}
Since $\kappa^q_\vep(a,b)\geq \E_a[\e^{-\vep^2qT_b}]$, the lower bounds in \eqref{ineq1:k} and \eqref{ineq2:k} follow readily from Theorem~\ref{thm:expmom1} and the standard properties of $K_0$ and $I_0$: $K_0$ is decreasing, $I_0$ is increasing, and $I_0(0)=1$.  

The restrictions on $b$ in \eqref{ineq1:k0} and $\vep$ in \eqref{ineq1:k0} and \eqref{ineq2:k0} are required for the upper bounds. 
Note that by \eqref{ineq:SP:g}, 
\begin{align}
\Bigg|\int_0^{T_{b}(\rho)}\oovarphi\uvep(\rho_v)\d v\Bigg|
\leq \|\oovarphi\uvep\|_\infty\int_0^{T_{\log b}(\beta)}\1_{(-\infty,\log M_\varphi]}(\beta_v)\e^{2\beta_v}\d v,\label{ineq:Tb}
\end{align}
where the last equality uses the assumption $\supp(\oovarphi\uvep)\subset [0,M_\varphi]$. Hence, by Proposition~\ref{prop:expmom2}, 
\begin{align}
\kappa^q_\vep(a,b)&\leq \frac{\cos\big(\log (\tfrac{M_\varphi}{a\wedge M_\varphi})\sqrt{2\nu_\vep}\,\big)}{\cos\big(\log(\tfrac{M_\varphi}{b\wedge M_\varphi})\sqrt{2\nu_\vep}\,\big)},\hspace{1.96cm}\forall\;a,b: M_\varphi\e^{-\frac{\pi}{2\sqrt{2\nu_\vep}}}<b\leq a;\label{ineq1:k}\\
\kappa^q_\vep(a,b)&\leq \frac{ \int_{-1}^1 (1-t^2)^{-1/2}\cos(a\sqrt{2\nu_\vep}t)\d t}
{\int_{-1}^1 (1-t^2)^{-1/2}\cos(b\sqrt{2\nu_\vep}t)\d t},\hspace{.6cm}\forall\;a,b:  0<a\leq b<\frac{j_{0,1}}{\sqrt{2\nu_\vep}}.\label{ineq2:k}
\end{align}

Since $\log(\tfrac{M_\varphi}{b\wedge M_\varphi})\sqrt{2\nu_\vep}\in [0,\pi/2)$ for $b>M_\varphi\e^{-\frac{\pi}{2\sqrt{2\nu_\vep}}}$, we obtain the upper bound in \eqref{ineq1:k0} from  \eqref{ineq1:k} and the monotonicity of $\theta\mapsto \cos(\theta)$  in $ [0,\pi/2)$. To get \eqref{ineq2:k0}, choose $\vep_{\ref{ineq2:k0}}=\vep_{\ref{ineq2:k0}}(\|\varphi\|_\infty,M_\varphi,\lambda)\in (0,\ovp)$ such that $M_\varphi< (\pi/4)/\sqrt{2\nu_{\vep}}$ for all $\vep\in(0, \vep_{\ref{ineq2:k0}})$. Then for all these $\vep$, $t\in (-1,1)$, and $0<b\leq M_\varphi$, we have $b\sqrt{2\nu_\vep} |t|\leq M_\varphi\sqrt{2\nu_\vep} |t|<\pi/4$. By the monotonicity of $\theta\mapsto \cos(\theta)$ in $[0,\pi/4)$ and in $(-\pi/4,0]$,  \eqref{ineq2:k} is enough to get \eqref{ineq2:k0}. The proof is complete. 
\end{proof}

Let us introduce some tools to be used repeatedly in the sequel. First, recall the following asymptotic behavior of the Macdonald function $K_0$ \eqref{def:K}:
\begin{align}\label{ob3:K0}
K_0(r)\sim \log r^{-1},\quad r\to 0;\quad
 K_0(r)\sim\sqrt{\pi/(2r)}\e^{-r},\quad r\to \infty.
\end{align} 
\cite[(5.16.4) and (5.16.5) on p.136]{Lebedev}. We also need the identity: 
\begin{align}
\label{int:harmonic}
\int_{-\pi}^\pi \log |1-r\e^{\i \theta}|\d \theta=0,\quad\forall\;r\in [0,1].
\end{align}
To see \eqref{int:harmonic}, note that it holds for all $r\in [0,1)$ by the mean-value property of $\Bbb C\ni z\mapsto \log |1-z|$ in the centered open disc with radius one. The identity extends to $r=1$ by dominated convergence.

Our next goal is to validate the second condition of \eqref{cond:q} by working with \eqref{Phiblambda:resolvent}. 
We start with an asymptotic expansion of $\Phi_b(\vep^2q)^{-1}$.

\begin{lem}\label{lem:Phibasymp}
Given any $q\in (0,\infty)$, there exist $\vep_{\ref{asymp:Phibminus}}=\vep_{\ref{asymp:Phibminus}}(M_\varphi,q)\in (0,\ovp)$ and a function $\eta_{\ref{asymp:Phibminus_bdd}}(\vep)>0$ depending only on $(M_\varphi,q,\vep)$ and tending to zero as $\vep\to 0$, such that the following properties hold for all $\vep\in (0,\vep_{\ref{asymp:Phibminus}})$ and $b\in (0,M_\varphi]$:
\begin{align}\label{asymp:Phibminus}
\begin{split}
&\Phi_b(\vep^2q)^{-1}
=2b\log \vep^{-1} +2b\left(\frac{\log 2-\log q}{2}-\log b-\EM\right)+{\rm R}_{\ref{asymp:Phibminus}}(q,b,\vep),
\end{split}\\
\label{asymp:Phibminus_bdd}
&|{\rm R}_{\ref{asymp:Phibminus}}(q,b,\vep)|\less
b\cdot \eta_{\ref{asymp:Phibminus_bdd}}(\vep).
\end{align}
\end{lem}
\begin{proof}
By \eqref{Phib:resolvent} and then polar coordinates with a change of variable in $t$, we have
\begin{align*}
\Phi_b(\vep^2q)^{-1}&=\int_0^\infty \e^{-\vep^2q t}Q_t(b,b)\d t=\int_{-\pi}^\pi b\cdot \int_0^\infty \frac{\e^{- t} }{2\pi t}\exp\left(-\frac{\vep^2qb^2|1-\e^{\i \theta}|^2}{2t}\right)\d t\d \theta.
\end{align*}
Hence, by  \eqref{ob3:K0} and \eqref{asymp:gauss}, we get
\begin{align}
\Phi_b(\vep^2q)^{-1}&= \int_{-\pi}^\pi b\cdot  \frac{1}{\pi}\log \frac{1}{\vep q^{1/2}b|1-\e^{\i \theta}|}
 \d\theta+\int_{-\pi}^\pi b\cdot \frac{1}{\pi}\left(\frac{\log 2}{2}-\EM\right)\d \theta+{\rm R}_{\ref{asymp:Phibminus0}}(q,b,\vep)\label{asymp:Phibminus0}\\
 \begin{split}\label{asymp:Phibminus0+1}
 &=-2b\log q^{1/2}-2b\log \vep-2b\log b-\frac{b}{\pi}\int_{-\pi}^\pi \log |1-\e^{\i\theta}| \d \theta\\
 &\eqspace+2b\left(\frac{\log 2}{2}-\EM\right)+{\rm R}_{\ref{asymp:Phibminus0}}(q,b,\vep),
 \end{split}
\end{align}
where
 \begin{align*}
&|{\rm R}_{\ref{asymp:Phibminus0}}(q,b,\vep)|\less \left| \int_{\stackrel{\scriptstyle -\pi\leq \theta\leq \pi}{\vep q^{1/2}b|1-\e^{\i \theta}|\geq 1}}  b\cdot  \frac{1}{\pi}\left(\log \frac{2^{1/2}}{\vep q^{1/2}b|1-\e^{\i \theta}|}-\EM\right)
 \d\theta\right|\\
&\hspace{1.8cm}+ \int_{\stackrel{\scriptstyle -\pi\leq \theta\leq \pi}{\vep q^{1/2}b|1-\e^{\i \theta}|\geq 1}} b\cdot \e^{-\vep q^{1/2}b|1-\e^{\i \theta}|}\d \theta+\int_{\stackrel{\scriptstyle -\pi\leq \theta\leq \pi}{\vep q^{1/2}b|1-\e^{\i \theta}|< 1}} b\cdot \vep^2q b^2|1-\e^{\i \theta}|^2\log \frac{1}{ \vep q^{1/2}b|1-\e^{\i \theta}|}\d \theta.
\end{align*}
To bound the right-hand side of the foregoing $\less$-inequality, notice that $\vep q^{1/2}b|1-\e^{\i \theta}|\geq 1$ for some $|\theta|\leq \pi$  implies $\vep q^{1/2}b|1-\e^{\i \pi}|\geq 1$, and $\int_{0+}\big|\log |1-\e^{\i \theta}|\big|\d \theta<\infty$ holds. Hence,
\begin{align}\label{asymp:Phibminus_bdd++}
|{\rm R}_{\ref{asymp:Phibminus0}}(q,b,\vep)|&\less
 b\1_{\{\vep q^{1/2} b\geq 1/4\}}\left(\left|\log \frac{1}{\vep q^{1/2}b}\right|+1\right)+
 b\cdot \vep^2qb^2\left(\left|\log \frac{1}{\vep q^{1/2}b}\right|+1\right).
\end{align}

Finally, \eqref{asymp:Phibminus} follows from  \eqref{asymp:Phibminus0+1} and \eqref{int:harmonic} if we set ${\rm R}_{\ref{asymp:Phibminus}}(q,b,\vep)={\rm R}_{\ref{asymp:Phibminus0}}(q,b,\vep)$. Then \eqref{asymp:Phibminus_bdd} holds by \eqref{asymp:Phibminus_bdd++} and the monotonicity of $r^2\log r^{-1}$ over $(0,\e^{-1/2}]$. The proof is complete.
\end{proof}

Next, we consider the following Green function that enters $\Phi_b(\vep^2q,f_\vep)$ via \eqref{Phiblambda:resolvent}:
\begin{align}\label{inter:1}
\E_b\left[\int_0^\infty  \e^{-\vep^2q r} f_\vep(\rho_r)\E_{\rho_r}\big[\e^{-\vep^2qT_b+\ooA_\vep(T_b)}\big]\d r\right]=
\E_b\left[\int_0^\infty  \e^{-\vep^2q r} f_\vep(\rho_r)\kappa_{\vep}^q(\rho_r,b)\d r\right],
\end{align}
where we use the notation in \eqref{def:k}. The asymptotic behavior of this term calls for the main technical considerations. We expand $\kappa^q_\vep$ in $\vep$. To this end, note that by the first equality in \eqref{taylor1},
\begin{align}\label{eq0:Qexp}
\e^{-\vep^2qT_b+\ooA_\vep(T_b)}=1+\int_0^{T_b}[-\vep^2q+f_\vep(\rho_s)]\e^{-\vep^2q(T_b-s)+\int_s^{T_b}f_\vep(\rho_r)\d r}\d s,
\end{align}
or equivalently,
\begin{align}\label{eq:Qexp}
\e^{-\vep^2qT_b+\ooA_\vep(T_b)}+\int_0^{T_b}\vep^2q\e^{-\vep^2q(T_b-s)+\int_s^{T_b}f_\vep(\rho_r)\d r}\d s=1+\int_0^{T_b}f_\vep(\rho_s)\e^{-\vep^2q(T_b-s)+\int_s^{T_b}f_\vep(\rho_r)\d r}\d s.
\end{align}
Taking expectations of both sides of \eqref{eq:Qexp} and applying the Markov property of $(\rho_t)$ yield
\begin{align}
\kappa_{\vep}^q(a,b)+\vep^2q\E_a\left[\int_0^{T_b}\kappa_{\vep}^q(\rho_s,b)\d s\right]
&=1+ \E_a\left[\int_0^{T_b}f_\vep (\rho_s)\kappa_{\vep}^q(\rho_s,b)\d s\right].\label{inter:2}
\end{align}
Applying \eqref{inter:2} to the right-hand side of \eqref{inter:1}, we get
\begin{align}
\begin{split}\label{exp:Green}
\eqspace\E_b\left[\int_0^\infty  \e^{-\vep^2q r} f_\vep(\rho_r)\kappa^q_\vep(\rho_r,b)\d r\right]+{\rm III}_{\ref{exp:Green}}(q,b,\vep)
=1-{\rm I}_{\ref{exp:Green}}(q,b,\vep)-{\rm II}_{\ref{exp:Green}}(q,b,\vep),
\end{split}
\end{align}
where
\begin{align*}
{\rm I}_{\ref{exp:Green}}(q,b,\vep)&=1-\E_b\left[\int_0^\infty  \e^{-\vep^2q r} f_\vep(\rho_r)\d r\right],\\
{\rm II}_{\ref{exp:Green}}(q,b,\vep)&=-\E_b\left[\int_0^\infty  \e^{-\vep^2q r} f_\vep(\rho_r)\E_{\rho_r}\left[\int_0^{T_b}f_\vep(\rho_s)\kappa_{\vep}^q(\rho_s,b)\d s\right]\d r\right],\\
{\rm III}_{\ref{exp:Green}}(q,b,\vep)&=\E_b\left[\int_0^\infty  \e^{-\vep^2q r} f_\vep(\rho_r) \vep^2q\E_{\rho_r}\left[\int_0^{T_b}\kappa_{\vep}^q(\rho_s,b)\d s\right]  \d r\right].
\end{align*}

\begin{lem}\label{lem:kappaasymp}
{\rm (1$\cc$)} It holds that 
\begin{align}
&\eqspace\eqspace \sup_{a,b:0< a\leq b\leq M_\varphi}\E_a\left[\int_0^{T_b}\kappa_{\vep}^q(\rho_s,b)\d s\right]\leq C(M_\varphi),\quad \forall\; \vep\in (0,\vep_{\ref{ineq2:k0}}).\label{kappaasymp:1}
\end{align}
{\rm (2$\cc$)} Given any $q\in (0,\infty)$ and $\delta\in (0,1)$, we can find a function 
$\eta_{\ref{kappaasymp:2}}(\vep)>0$ depending only on $(M_\varphi,q,\delta,\lambda,\vep)$ and tending to zero as $\vep\to 0$, 
such that, for all $\vep\in (0,\ovp)$,  
\begin{align} 
\begin{split}
\sup_{a,b:M_\varphi\e^{-\delta\frac{\pi}{2\sqrt{2\nu_\vep\beta_\vep   }}}\leq b< a\leq M_\varphi} \left|\frac{\log \vep^{-1}\cdot \vep^2q}{\log a-\log b}\E_a\left[\int_0^{T_b}\kappa_{\vep}^q(\rho_s,b)\d s\right]-1\right|\leq \eta_{\ref{kappaasymp:2}}(\vep),
\end{split}
\label{kappaasymp:2}
\end{align}
where $\nu_\vep$ and $\beta_\vep$ are defined in \eqref{def:nubeta}.
\end{lem}
\begin{proof}
(1$\cc$) By \eqref{eq:Tr} and \eqref{ineq2:k0}, we have, for all $\vep\in (0,\vep_{\ref{ineq2:k0}})$ and $0<a\leq b\leq M_\varphi$,
\begin{align*}
\E_a\left[\int_0^{T_b}\kappa_{\vep}^q(\rho_s,b)\d s\right]&\less 
2\int_a^b(\log b-\log r) r\d r+2(\log b-\log a) \int_0^a r\d r\\
&=2\log b\int_0^br \d r-2\int_a^b (\log r)r \d r-2\log a\int_0^a r\d r,
\end{align*}
which is enough for \eqref{kappaasymp:1}. \medskip

\noindent (2$\cc$) The proof of \eqref{kappaasymp:2} considers the following identity from \eqref{eq:Tl}: for $0<b\leq a<\infty$,
\begin{align}\label{eq1:Q20}
\begin{split}
(\log \vep^{-1})\cdot \vep^2q\cdot \E_a\left[\int_0^{T_b}\kappa_{\vep}^q(\rho_s,b)\d s\right]&=(\log \vep^{-1}) \cdot 2\vep^2q\cdot (\log a-\log b) \int_a^\infty \kappa_{\vep}^q(r,b) r\d r\\
&\quad +(\log \vep^{-1}) \cdot  2\vep^2q\cdot \int_b^a(\log r-\log b)\kappa_{\vep}^q(r,b) r\d r.
\end{split}
\end{align}
For all $\vep\in (0,\ovp)$ and all $a,b$ such that $M_\varphi\e^{-\delta\frac{\pi}{2\sqrt{2\nu_\vep\beta_\vep   }}}\leq b\leq  a\leq M_\varphi$, the upper bound in \eqref{ineq1:k0} shows that the last term in \eqref{eq1:Q20} satisfies
\begin{align}
(\log \vep^{-1}) \cdot  2\vep^2q\cdot  \int_b^a(\log r-\log b)\kappa_{\vep}^q(r,b) r\d r &\leq (\log \vep^{-1}) \cdot \vep^2q\cdot C(M_\varphi,\delta)(\log a-\log b).
\label{eq1:Q200}
\end{align}
For the first term on the right-hand side of \eqref{eq1:Q20}, it remains to estimate
\begin{align}\label{Kint+++}
(\log \vep^{-1}) \cdot 2\vep^2q\cdot  \int_a^\infty \kappa_{\vep}^q(r,b) r\d r.
\end{align}

Consider two-sided bounds of $\kappa^q_\vep(r,b)$ as follows.
For any $\vep\in (0,\ovp)$, we extend $\beta_\vep$ in \eqref{def:nubeta} to a pair of H\"older conjugates $(\alpha_\vep,\beta_\vep)$ defined by
\begin{align}\label{def:alphabeta}
\alpha_\vep=[1-(\log \log \vep^{-1})^{-2}]^{-1},\quad \beta_\vep=(\log\log \vep^{-1})^2. 
\end{align}
As in the proof of Lemma~\ref{lem:kappato1}, it follows from Theorem~\ref{thm:expmom1} and Proposition~\ref{prop:expmom2}, for all $b$ such that $M_\varphi\e^{-\delta\frac{\pi}{2\sqrt{2\nu_\vep\beta_\vep   }}}\leq b\leq M_\varphi$, $b\leq r<\infty$, and $\vep\in (0,\ovp)$,
\begin{align*}
\frac{K_0(r\sqrt{2\vep^2q})}{K_0(b\sqrt{2 \vep^2q})}\leq \kappa_{\vep}^q(r,b)
&\leq \left(\frac{K_0(r\sqrt{2\alpha_\vep \vep^2q}\,)}{K_0(b\sqrt{2\alpha_\vep \vep^2q}\,)}\right)^{1/\alpha_\vep}
\frac{1}{\cos(\pi\delta /2)^{1/\beta_\vep}}.
\end{align*}
Here, the upper bound follows from H\"older's inequality and the argument in the proof of \eqref{ineq1:k0} (with $\nu_\vep$ replaced by $\nu_\vep \beta_\vep$ in \eqref{ineq1:k}). Moreover, since $K_0$ is strictly decreasing, the foregoing display implies 
\begin{align}\label{ineq2:Q}
\frac{K_0(r\sqrt{2\vep^2q})}{K_0(M_\varphi\e^{-\frac{\delta\pi}{2\sqrt{2\nu_\vep\beta_\vep   }}}\sqrt{2 \vep^2q})}\leq \kappa_{\vep}^q(r,b)
&\leq \left(\frac{K_0(r\sqrt{2\alpha_\vep \vep^2q}\,)}{K_0(M_\varphi\sqrt{2\alpha_\vep \vep^2q}\,)}\right)^{1/\alpha_\vep}
\frac{1}{\cos(\pi\delta /2)^{1/\beta_\vep}}.
\end{align}

Now, consider a lower bound of the term in \eqref{Kint+++}.
By the lower bound in \eqref{ineq2:Q}, for all $a,b$ such that $M_\varphi\e^{-\frac{\delta\pi}{2\sqrt{2\nu_\vep\beta_\vep   }}}\leq b\leq a\leq M_\varphi$ and $\vep\in (0,\ovp)$,
\begin{align}
(\log \vep^{-1}) \cdot 2 \vep^2q\cdot \int_a^\infty \kappa_{\vep}^q(r,b) r\d r
&\geq \log \vep^{-1}\cdot 2\vep^2q\cdot \int_a^\infty \frac{K_0(r\sqrt{2\vep^2q})r\d r}{{K_0(M_\varphi\e^{-\frac{\delta\pi}{2\sqrt{2\nu_\vep\beta_\vep   }}}\sqrt{2\vep^2q})}} \notag\\
&=\log \vep^{-1} \cdot \frac{a\sqrt{2\vep^2q}K_1(a\sqrt{2\vep^2q})}{{K_0(M_\varphi\e^{-\frac{\delta\pi}{2\sqrt{2\nu_\vep\beta_\vep   }}}\sqrt{2\vep^2q})}}\notag\\
&\geq \log \vep^{-1} \cdot \frac{1-\int_0^\infty  \e^{- s}\int_0^{M_\varphi^2(2\vep^2q)/(4s)}\e^{-v}\d v \d s}{{K_0(M_\varphi\e^{-\frac{\delta\pi}{2\sqrt{2\nu_\vep\beta_\vep   }}}\sqrt{2\vep^2q})}},
\label{K0:int}
\end{align}
where the equality and the second inequality follow from 
\begin{align}\label{asymp:K0K1}
\frac{\d}{\d r} [rK_1(r)]=-rK_0(r)
\quad\mbox{and}\quad rK_1(r)=\int_0^\infty  \e^{- s}\int_{r^2/(4s)}^\infty \e^{-v}\d v \d s,
\end{align}
respectively.
See \cite[(5.7.9) on p.110]{Lebedev} and \eqref{ineq:K1bdd} for the identities in \eqref{asymp:K0K1}.  To use the bound in \eqref{K0:int}, note that the asymptotic expansion of $K_0(\sqrt{2}r)$ for $r\to 0$ in \eqref{asymp:gauss} gives, for all $\vep\in (0,\ovp)$, 
\begin{align}\label{K0asymp}
K_0(M_\varphi\e^{-\frac{\delta\pi}{2\sqrt{2\nu_\vep\beta_\vep   }}}\sqrt{2\vep^2q})= \log \big(\e^{-\frac{\delta\pi}{2\sqrt{2\nu_\vep\beta_\vep   }}}\sqrt{\vep^2}\big)^{-1}+{\rm R}_{\ref{K0asymp}}
= \frac{\delta\pi}{2\sqrt{2\nu_\vep \beta_\vep}}+\log \vep^{-1}+{\rm R}_{\ref{K0asymp}},
\end{align}
where $|{\rm R}_{\ref{K0asymp}}|\less C(M_\varphi,q,\delta,\lambda)$.
Hence, the right-hand side of \eqref{K0:int} depends only on $(M_\varphi,q,\delta,\lambda,\vep)$ and tends to $1$ as $\vep\to 0$. 

To get an upper bound of the term in \eqref{Kint+++}, note that the other inequality in \eqref{ineq2:Q} gives
\begin{align}
&\eqspace\log \vep^{-1}\cdot 2\vep^2q\cdot \int_a^\infty \kappa_{\vep}^q(r,b)\d r\notag\\
&\leq \log \vep^{-1}\cdot 2\vep^2q\cdot \int_a^\infty \left(\frac{K_0(r\sqrt{2\alpha_\vep \vep^2q}\,)}{K_0(M_\varphi\sqrt{2\alpha_\vep \vep^2q}\,)}\right)^{1/\alpha_\vep} r\d r\cdot \frac{1}{\cos(\pi\delta/2)^{1/\beta_\vep}}\notag\\
\begin{split}
&= \frac{\log \vep^{-1}}{K_0(M_\varphi\sqrt{2\alpha_\vep \vep^2q})^{1/\alpha_\vep}}\cdot
 \frac{1}{\alpha_\vep}\int_{0}^\infty  K_0(r)^{1/\alpha_\vep}r\d r\cdot \frac{1}{\cos(\pi\delta/2)^{1/\beta_\vep}}.\label{ineq2:Quppbdd}
 \end{split}
\end{align}
For the right-hand side of \eqref{ineq2:Quppbdd}, note that
 $1/\alpha_\vep=1-(\log \log \vep^{-1})^{-2}$ by \eqref{def:alphabeta}. Hence, by the asymptotic behavior of $K_0(r)$ as $r\to 0$ in \eqref{ob3:K0},
\begin{align}
\frac{\log \vep^{-1}}{K_0(M_\varphi\sqrt{2\alpha_\vep \vep^2q})^{1/\alpha_\vep}}=
\frac{(\log \vep^{-1})^{1-1/\alpha_\vep}}{[K_0(M_\varphi\sqrt{2\alpha_\vep \vep^2q})/\log \vep^{-1}]^{1/\alpha_\vep}}
\xrightarrow[\vep\to 0]{}1.
\label{ob1:K0}
\end{align}
Additionally, by \eqref{ob3:K0}, dominated convergence, and \eqref{asymp:K0K1}, it holds that 
\begin{align}\label{ob2:K0}
\frac{1}{\alpha_\vep}\int_{0}^\infty  K_0(r)^{1/\alpha_\vep}r\d r\xrightarrow[\vep\to 0]{}\int_{0}^\infty  K_0(r)r\d r=1.
\end{align}
The right-hand side of \eqref{ineq2:Quppbdd} depends only on $(M_\varphi,q,\delta,\vep)$ and tends to $1$ by \eqref{ob1:K0} and \eqref{ob2:K0}.

Now, we obtain from the above implications of \eqref{K0:int} and \eqref{ineq2:Quppbdd} that 
\begin{align} 
\begin{split}
\sup_{a,b:M_\varphi\e^{-\frac{\delta\pi}{2\sqrt{2\nu_\vep\beta_\vep   }}}\leq b< a\leq M_\varphi} \left|\log \vep^{-1}\cdot 2\vep^2q\cdot \int_a^\infty \kappa_{\vep}^q(r,b) r\d r-1\right|\leq \eta_{\ref{kappaasymp:2000}}(\vep),\quad \forall\; \vep\in (0,\ovp),
\end{split}
\label{kappaasymp:2000}
\end{align}
where $\eta_{\ref{kappaasymp:2000}}(\vep)$ satisfies the same property of $\eta_{\ref{kappaasymp:2}}(\vep)$ in the statement of the present lemma. Combining \eqref{eq1:Q20}, \eqref{eq1:Q200} and \eqref{kappaasymp:2000}, we obtain the required estimate in \eqref{kappaasymp:2}. 
\end{proof}

Recall for the last time that $\nu_\vep$ and $\beta_\vep$ are defined in \eqref{def:nubeta}.

\begin{lem}\label{lem:finalexpansion}
Let $\delta_\vep\searrow 0$ with $\delta_\vep/\sqrt{2\nu_\vep\beta_\vep   }\to \infty$ as $\vep\to 0$ and $\delta_\vep/\sqrt{2\nu_\vep\beta_\vep   }\less C(\|\varphi\|,\lambda)(\log\vep^{-1})^{1/3}$. Then for all $q\in (0,\infty)$, $\vep\in (0,\vep(M_\varphi,q)\wedge \ovp)$, and $b\in (0,\infty)$ such that 
$M_\varphi\e^{-\delta_\vep\frac{\pi}{2\sqrt{2\nu_\vep\beta_\vep   }}}\leq b\leq  M_\varphi$, the following asymptotic expansions hold:
\begin{align}
\begin{split}
{\rm I}_{\ref{exp:Green}}(q,b,\vep)&=\frac{-2}{\log \vep^{-1}}\int_{\R^2}\left[\left(\log\frac{2^{1/2}}{q^{1/2} |b-z|}+\lambda-\gamma_{\sf EM}\right)\ovarphi(|z|)+2\pi \oevarphi(|z|)\right]\d z\\
&\eqspace +{\rm R}_{\ref{Ir:1}}(q,b,\vep),\label{Ir:1}
\end{split}\\
{\rm II}_{\ref{exp:Green}}(q,b,\vep)&=\frac{-4\pi }{\log \vep^{-1}}\int_{\R^2}\ovarphi(|z|)\E_{|z|}\left[\int_0^{T_b}\ovarphi (\rho_s)\d s\right]\d z+{\rm R}_{\ref{Ir:2}}(q,b,\vep),\label{Ir:2}\\
{\rm III}_{\ref{exp:Green}}(q,b,\vep)&= \frac{2}{\log \vep^{-1}}\int_{\R^2}\ovarphi(|z|)\log^+ (|z|/ b)\d z+{\rm R}_{\ref{Ir:3}}(q,b,\vep).\label{Ir:3}
\end{align}
(In \eqref{Ir:1}, $b$ and $z$ are identified as complex numbers.)
Here, for $\eta_{\ref{R:Ir2}}(\vep)$ and $\eta_{\ref{R:Ir3}}(\vep)>0$ depending only on $(\|\varphi\|_\infty,M_\varphi,q,\lambda,\delta_\vep,\vep)$ and tending to zero as $\vep\to 0$, the remainders in \eqref{Ir:1}--\eqref{Ir:3} satisfy 
\begin{align}
|{\rm R}_{\ref{Ir:1}}(q,b,\vep)|&\leq  \frac{C(\|\varphi\|_\infty,M_\varphi,q,\lambda)}{(\log \vep^{-1})^{2}},\label{R:Ir1}\\ 
|{\rm R}_{\ref{Ir:2}}(q,b,\vep)|&\less \frac{1}{\log \vep^{-1}} \left(\int_{\R^2}\ovarphi(|z|)\E_{|z|}\left[\int_0^{T_b}\ovarphi(\rho_s)\d s\right]\d z\right)\eta_{\ref{R:Ir2}}(\vep)+\frac{C(\|\varphi\|_\infty,M_\varphi,q,\lambda)}{(\log \vep^{-1})^{5/3}},\label{R:Ir2}\\
|{\rm R}_{\ref{Ir:3}}(q,b,\vep)|&\less \frac{1}{\log \vep^{-1}}\left(\int_{\R^2}\ovarphi(|z|)\log^+ (|z|/ b)\d z\right)\eta_{\ref{R:Ir3}}(\vep)+ \frac{C(\|\varphi\|_\infty,M_\varphi,q,\lambda)}{(\log \vep^{-1})^{5/3}}.\label{R:Ir3}
\end{align}
\end{lem}

\begin{proof}
We start with some general asymptotic expansions. As $\vep q^{1/2}|b-z|\to 0$, \eqref{asymp:gauss} shows
\begin{align}
&\eqspace\int_0^\infty \frac{\e^{-\vep^2qt}}{2\pi t}\exp\left(-\frac{|b-z|^2}{2t}\right)\d t\notag\\
&=\frac{\log \vep^{-1} }{\pi}+\frac{1}{\pi}\Bigg(\log \frac{2^{1/2}}{q^{1/2}|b-z|}-\gamma_{\sf EM}\Bigg)+\mathcal O\left(\vep^2q|b-z|^2\log \frac{1}{\vep q^{1/2}|b-z|}\right).\label{asympG:1}
\end{align}
Recall the definition of $\oovarphi\uvep$ in \eqref{def:varphi0} and the definition of $f_\vep$ in \eqref{setting:radial}. 
By \eqref{ob3:K0} and \eqref{asympG:1}, for $\vep\in (0,\ovp)$ and bounded $h_{b,\vep}$ to be chosen,
\begin{align}
&\eqspace\E_b\left[\int_0^\infty  \e^{-\vep^2q r} f_\vep(\rho_r)h_{b,\vep}(\rho_r)\d r\right]\notag\\
\begin{split}
&=\frac{(\log \vep^{-1}) \lv}{\pi}\int_{\R^2}\ovarphi(|z|)h_{b,\vep}(|z|)\d z+\frac{(\log \vep^{-1} )\lv^2}{\pi}\int_{\R^2}\oevarphi(|z|)h_{b,\vep}(|z|)\d z\\
&\eqspace +\frac{\lv }{\pi}\int_{\R^2}\left(\log\frac{2^{1/2}}{q^{1/2} |b-z|}-\gamma_{\sf EM}\right)\oovarphi\uvep(|z|)h_{b,\vep}(|z|)\d z+ {\rm R}_{\ref{remainder:rad0}}(h_{b,\vep},q, b,\vep)\label{remainder:rad0}
\end{split}\\
\begin{split}\label{remainder:rad}
&=\left(2+\frac{2\lambda}{\log \vep^{-1}}\right)\int_{\R^2}\ovarphi(|z|)h_{b,\vep}(|z|)\d z\\
&\eqspace +\left(2+\frac{2\lambda}{\log \vep^{-1}}\right)\left(\frac{2\pi}{\log \vep^{-1}}+\frac{2\pi \lambda}{(\log \vep^{-1})^2}\right)\int_{\R^2}\oevarphi(|z|)h_{b,\vep}(|z|)\d z\\
&\eqspace+\left(\frac{2}{\log \vep^{-1}}+\frac{2\lambda}{(\log \vep^{-1})^2}\right)\int_{\R^2}\left(\log\frac{2^{1/2}}{q^{1/2} |b-z|}-\gamma_{\sf EM}\right)\oovarphi\uvep(|z|)h_{b,\vep}(|z|)\d z\\
&\eqspace + {\rm R}_{\ref{remainder:rad0}}(h_{b,\vep},q,b,\vep).
\end{split}
\end{align}

Let us explain the term ${\rm R}_{\ref{remainder:rad0}}(h_{b,\vep},q,b,\vep)$. It considers the $\mathcal O$-term in \eqref{asympG:1} after integrating against $\oovarphi\uvep(|z|)h_{b,\vep}(|z|)$ and takes a form similar to ${\rm R}_{\ref{asymp:Phibminus0}}(q,b,\vep)$ from using \eqref{ob3:K0}: 
\begin{align}
|{\rm R}_{\ref{remainder:rad0}}(h_{b,\vep},q, b,\vep)|&\less 
\left|\frac{\lv }{\pi}\int_{\vep q^{1/2}|b-z|\geq \tfrac{1}{2}}\left(\log\frac{2^{1/2}}{\vep q^{1/2} |b-z|}-\gamma_{\sf EM}\right)\oovarphi\uvep(|z|)|h_{b,\vep}(|z|)|\d z\right|\notag\\
&\eqspace+ \lv \int_{\vep q^{1/2}|b-z|\geq \tfrac{1}{2}}\oovarphi\uvep(|z|)|h_{b,\vep}(|z|)|\e^{-\vep q^{1/2}|b-z|}\d z\notag\\
&\eqspace+\lv \int_{\vep q^{1/2}|b-z|< \tfrac{1}{2}}\oovarphi\uvep(|z|)|h_{b,\vep}(|z|)|\cdot \vep^2q|b-z|^2\log \frac{1}{\vep q^{1/2}|b-z|}\d z.\notag
\end{align}
Since $\ovarphi$ has a support in $[0,M_\varphi]$, for all $b\in (0,M_\varphi]$ and $\vep\in (0,\vep_{\ref{asympG:2-1}}(M_\varphi,q)\wedge \ovp)$, the first two terms on the right-hand side of the foregoing inequality are zero. We get
\begin{align}
|{\rm R}_{\ref{remainder:rad0}}(h_{b,\vep},q,b,\vep)|&\leq C(M_\varphi,q)\lv \cdot  \vep^2\log \vep^{-1}\cdot \int_{\R^2}\oovarphi\uvep(|z|)|h_{b,\vep}(|z|)|\d z.
\label{asympG:2-1}
\end{align}
From now on to the end of this proof, we only consider $b\in (0,\infty)$, $\vep\in (0, \vep_{\ref{asympG:2-1}})$, and $z$ such that  $M_\varphi\e^{-\delta_\vep\frac{\pi}{2\sqrt{2\nu_\vep\beta_\vep   }}}\leq b\leq  M_\varphi$. A supremum taken over $b,z$ is subject to the same conditions. 

To proceed, we simplify \eqref{remainder:rad} to an expansion in $1/\log \vep^{-1}$ up to the first order:
\begin{align}
\begin{split}\label{remainder:rad++}
&\E_b\left[\int_0^\infty  \e^{-\vep^2q r} f_\vep(\rho_r)h_{b,\vep}(\rho_r)\d r\right]=2\int_{\R^2}\ovarphi(|z|)h_{b,\vep}(|z|)\d z\\
&\eqspace+\frac{2}{\log \vep^{-1}}\int_{\R^2}\left[\left(\log\frac{2^{1/2}}{q^{1/2} |b-z|}+\lambda-\gamma_{\sf EM}\right)\ovarphi(|z|)+2\pi \oevarphi(|z|)\right]h_{b,\vep}(|z|)\d z\\
&\eqspace + {\rm R}_{\ref{remainder:rad++}}(h_{b,\vep},q,b,\vep),
\end{split}
\end{align}
where the last two terms in \eqref{remainder:rad++} can be bounded as follows:
\begin{align}\label{remainder:rad+++}
\begin{split}
&\sup_{b}\left|\frac{2}{\log \vep^{-1}}\int_{\R^2}\left[\left(\log\frac{2^{1/2}}{q^{1/2} |b-z|}+\lambda-\gamma_{\sf EM}\right)\ovarphi(|z|)+2\pi \oevarphi(|z|)\right]h_{b,\vep}(|z|)\d z\right|\\
&\eqspace\leq  \frac{C(\|\varphi\|_\infty,M_\varphi,q,\lambda)}{\log \vep^{-1}}\sup_{b,z}|h_{b,\vep}(|z|)|,
\end{split}\\
&\sup_{b}|{\rm R}_{\ref{remainder:rad++}}(h_{b,\vep},q, b,\vep)|\leq \frac{C(\|\varphi\|_\infty,M_\varphi,q,\lambda)}{(\log \vep^{-1})^2}\sup_{b,z}|h_{b,\vep}(|z|)|,\label{remainder:rad++++}
\end{align}
where the last inequality uses \eqref{asympG:2-1}.
We work with \eqref{remainder:rad++}--\eqref{remainder:rad++++} in the following three steps.
 \medskip

\noindent {\bf Step 1.} For ${\rm I}_{\ref{exp:Green}}(q,b,\vep)$, take $h_{b,\vep}\equiv -1$. Since $\int \ovarphi=1/2$, \eqref{remainder:rad++} gives
\begin{align}
&\eqspace {\rm I}_{\ref{exp:Green}}(q,b,\vep)\notag\\
&=-\frac{2}{\log \vep^{-1}}\int_{\R^2}\left[\left(\log\frac{2^{1/2}}{q^{1/2} |b-z|}+\lambda-\gamma_{\sf EM}\right)\ovarphi(|z|)+2\pi \oevarphi(|z|)\right]\d z +{\rm R}_{\ref{remainder:rad++}}(-1,q,b,\vep),\label{asymp:I1-1}
\end{align}
which proves \eqref{Ir:1} by using \eqref{remainder:rad++++}. \medskip

\noindent {\bf Step~\hypertarget{ReplaceFinal}{2}.} For ${\rm II}_{\ref{exp:Green}}(q,b,\vep)$, we take
\begin{align}\label{hchoice:2}
h_{b,\vep}(|z|)=-\lv\E_{|z|}\left[\int_0^{T_b}\oovarphi \uvep (\rho_s)\kappa_{\vep}^q(\rho_s,b)\d s\right].
\end{align}
In this case, we only consider the leading order term on the right-hand side of \eqref{remainder:rad++}. So we need to bound $\sup_{b,z}|h_{b,\vep}(|z|)|$ for the right-hand sides of \eqref{remainder:rad+++} and \eqref{remainder:rad++++}. These two terms are given by
\begin{align}\label{Ir:2:obj}
 -2\int_{\R^2}\ovarphi(|z|)\cdot \lv\E_{|z|}\left[\int_0^{T_b}\ovarphi\uvep (\rho_s)\kappa_{\vep}^q(\rho_s,b)\d s\right] \d z\quad \&\quad 
\sup_{b,z}\lv\E_{|z|}\left[\int_0^{T_b}\oovarphi\uvep (\rho_s)\kappa_{\vep}^q(\rho_s,b)\d s\right].
\end{align}

On one hand, by \eqref{eq:Tl}, for $b,z$ in the range specified above with $b\leq |z|$,  
\begin{align}
\E_{|z|}\left[\int_0^{T_b}\ovarphi(\rho_s)\d s\right] &\leq C(\|\varphi\|_\infty,M_\varphi)\left(\log M_\varphi-\log M_\varphi\e^{-\delta_\vep\frac{\pi}{2\sqrt{2\nu_\vep\beta_\vep   }}}\right)\notag\\
&=C(\|\varphi\|_\infty,M_\varphi)
 \delta_\vep\frac{\pi}{2\sqrt{\nu_\vep\beta_\vep}}.\label{h2:estimate}
 \end{align}
 Recall that $\delta_\vep$ is chosen in the statement of the present lemma. 
 The bound in \eqref{h2:estimate} also applies in the case of $|z|\leq b$ by \eqref{kappaasymp:1}. We get
 \begin{align}
\sup_{b,z}\lv\E_{|z|}\left[\int_0^{T_b}\oovarphi\uvep (\rho_s)\kappa_{\vep}^q(\rho_s,b)\d s\right]&\leq C(\|\varphi\|_\infty,M_\varphi) \sup_{b,z}\lv\E_{|z|}\left[\int_0^{T_b}\ovarphi (\rho_s)\kappa_{\vep}^q(\rho_s,b)\d s\right]\notag\\
&\leq \frac{C(\|\varphi\|_\infty,M_\varphi,q,\lambda)}{(\log \vep^{-1})^{2/3}}\label{R:lr2++}
\end{align}
by the choice of $\delta_\vep$.
On the other hand, for the first term in \eqref{Ir:2:obj}, we have the replacement:
\begin{align}
&\eqspace\left|\lv \E_{|z|}\left[\int_0^{T_b}\oovarphi\uvep (\rho_s)\kappa_{\vep}^q(\rho_s,b)\d s\right]-\frac{2\pi}{\log \vep^{-1}} \E_{|z|}\left[\int_0^{T_b}\ovarphi(\rho_s)\d s\right]\right|\notag\\
&\less\frac{1}{\log \vep^{-1}}  \E_{|z|}\left[\int_0^{T_b}\ovarphi(\rho_s)|\kappa_{\vep}^q(\rho_s,b)-1|\d s\right] +
C(\|\varphi\|_\infty,M_\varphi)\lv^2\E_{|z|}\left[\int_0^{T_b}\ovarphi(\rho_s)\d s\right]\notag\\
&\leq \frac{1}{\log \vep^{-1}}  \E_{|z|}\left[\int_0^{T_b}\ovarphi(\rho_s)\d s\right] \eta_{\ref{def:hrad}}(\vep),\label{def:hrad}
\end{align}
where $\eta_{\ref{def:hrad}}(\vep)>0$ depends only on $(M_\varphi,q,\lambda,\delta_\vep,\vep)$ and tends to zero as $\vep\to 0$.  To justify these properties of $\eta_{\ref{def:hrad}}(\vep)$, we use Lemma~\ref{lem:kappato1} and \eqref{R:lr2++}. In using this lemma, recall that $\ovarphi$ has a support in $[0,M_\varphi]$ and $0$ is polar under $\BES^2$. We also need the asymptotic behavior of $K_0$ as in the proof of Lemma~\ref{lem:kappaasymp} (cf. \eqref{K0asymp}), the continuity of $I_0$ with $I_0(0)=1$, and the assumption $\delta_\vep\to 0$.

In summary, according to \eqref{def:hrad}, we replace the first term in \eqref{Ir:2:obj} by the first term on the right-hand side of \eqref{Ir:2}. Then take $\eta_{\ref{R:Ir2}}(\vep)=\eta_{\ref{def:hrad}}(\vep)$. The right-hand side of \eqref{def:hrad} leads to the first term on the right-hand side of \eqref{R:Ir2}. The last term in \eqref{R:Ir2} follows upon applying \eqref{R:lr2++} to \eqref{remainder:rad+++} and \eqref{remainder:rad++++}. \medskip

\noindent {\bf Step 3.} The proof of the required asymptotic expansion of ${\rm III}_{\ref{exp:Green}}(q,b,\vep)$ is similar. In this case, take
\[
h_{b,\vep}(|z|)= \vep^2q\E_{|z|}\left[\int_0^{T_b}\kappa^q_\vep(\rho_s,b)\d s\right].
\]
In contrast to the two terms in \eqref{Ir:2:obj} for Step~\hyperlink{ReplaceFinal}{2}, now consider
\begin{align}\label{Ir:3:obj}
2\int_{\R^2}\ovarphi(|z|)\cdot  \vep^2q\E_{|z|}\left[\int_0^{T_b}\kappa^q_\vep(\rho_s,b)\d s\right]\d z\quad  \&\quad \sup_{b,z} \vep^2q\E_{|z|}\left[\int_0^{T_b}\kappa^q_\vep(\rho_s,b)\d s\right]
\end{align}
for the application of \eqref{remainder:rad++}--\eqref{remainder:rad++++}.

By Lemma~\ref{lem:kappaasymp},  we have
\begin{align}\label{Ir:3:est}
\begin{split}
&\left|  \vep^2q\E_{|z|}\left[\int_0^{T_b}\kappa^q_\vep(\rho_s,b)\d s\right]-\frac{\log^+ (|z|/b)}{\log \vep^{-1}}\right|\leq\frac{\log^+ (|z|/b)}{\log \vep^{-1}}  \eta_{\ref{kappaasymp:2}}(\vep)+ \vep^2 C(M_\varphi,q).
\end{split}
\end{align}
Hence, as in \eqref{h2:estimate}, we have
\begin{align}\label{Ir:3:est1}
\sup_{b,z}\vep^2q\E_{|z|}\left[\int_0^{T_b}\kappa^q_\vep(\rho_s,b)\d s\right]\leq \frac{C(\|\varphi\|_\infty,M_\varphi,q,\lambda)}{(\log \vep^{-1})^{2/3}}.
\end{align}
To complete the proof of \eqref{Ir:3}, the remaining details follow similarly as in the summary of Step~\hyperlink{ReplaceFinal}{2}. The proof is complete.
\end{proof}

We calculate the sum of  the first terms on the right-hand sides of \eqref{Ir:1}--\eqref{Ir:3}.
For all $b\in(0,\infty)$, 
\begin{align}\label{def:finalI:rad}
\begin{split}
\I_{\ref{def:finalI:rad}}(q)
&\defeq-2\int_{\R^2}\left[\left(\log\frac{2^{1/2}}{q^{1/2} |b-z|}+\lambda-\gamma_{\sf EM}\right)\ovarphi(|z|)+2\pi \oevarphi(|z|)\right]\d z\\
&\eqspace -4\pi \int_{\R^2}\ovarphi (|z|)\E_{|z|}\left[\int_0^{T_b}\ovarphi (\rho_s)\d s\right]\d z+ 2\int_{\R^2}\ovarphi(|z|) \log^+ (|z|/b)\d z.
\end{split}
\end{align}
We stress that $\I_{\ref{def:finalI:rad}}(q)$ does not depend on $b$. This property is a consequence of the following lemma.

\begin{lem}\label{lem:exponent}
For all $q\in (0,\infty)$ and $b\in (0,\infty)$,
\begin{align}
\I_{\ref{def:finalI:rad}}(q)&= \int_{\R^2}\int_{\R^2}(\log |z-z'|)\phi(z)\phi(z')\d z\d z'-\left(\log 2+\lambda-\gamma_{\sf EM}\right)-\frac{\log q}{2}.\label{eq:finalI:rad}
\end{align}
Hence, $\I_{\ref{def:finalI:rad}}(q)=(1/2)\log (q/\beta)$, where $\beta$ is defined by \eqref{logbeta}.
\end{lem}
\begin{proof}
We start with the introduction of the central identity for this proof. Let $S_b$ denote the circle in $\Bbb C$ centered at the origin and with radius $b$, and let $g_A(z,z')$ and $h_A(z,\d z')$ be defined by
\[
\E_z^W\left[\int_0^{H_A}f(W_t)\d t\right]=\int_{\R^2}f(z')g_A(z,z')\d z',\quad h_A(z,\d z')=\P^W_z(H_A<\infty, W_{H_A}\in \d z')
\] 
for $H_A=\inf\{t>0;W_t\in A\}$. Then the fundamental identity for logarithmic potentials stated in \cite[p.71]{PS:BM} shows
\begin{align}\label{eq:filp}
\frac{1}{\pi}\log \frac{1}{|z-z'|}=g_{S_b}(z,z')+\int h_{S_b}(z,\d z'')\frac{1}{\pi}\log \frac{1}{|z''-z'|}-\frac{1}{\pi}\log^+(|z|/b),\quad z,z'\in \R^2.
\end{align}
Here, we specialize the identity to $S_b$ and state the identity by incorporating the explicit formulas in \cite[Eq. (1) on p. 70 and Proposition~4.9 on p.75]{PS:BM}.

Let us calculate some integrals in $\I_{\ref{def:finalI:rad}}(q)$. Write the sum of the last two terms in \eqref{def:finalI:rad} as
\begin{align*}
&\eqspace -4\pi \int_{\R^2}\ovarphi (|z|)\E_{|z|}\left[\int_0^{T_b}\ovarphi (\rho_s)\d s\right]\d z +2\int_{\R^2}\ovarphi(|z|) \log^+ (|z|/b)\d z\\
&=-4\pi \int_{\R^2}\int_{\R^2}\ovarphi(|z|)\ovarphi(|z'|)g_{S_b}(z,z') \d z'\d z +4\int_{\R^2}\int_{\R^2}\ovarphi(|z|)\ovarphi(|z'|)\log^+ (|z|/b)\d z'\d z\\
&=-4\int_{\R^2}\int_{\R^2}\ovarphi(|z|)\ovarphi(|z'|) \log \frac{1}{|z-z'|}\d z'\d z+4\int_{\R^2}\int_{\R^2}\ovarphi(|z|)\ovarphi(|z'|)\int h_{S_b}(z,\d z'')\log \frac{1}{|z''-z'|}\d z'\d z,
\end{align*}
where the first equality uses $\int\varphi=1/2$, and the second equality follows from \eqref{eq:filp}. 
By the definition of $h_{S_b}$, the last integral can be written as
\begin{align*}
&\eqspace\int_{\R^2}\int_{\R^2}\ovarphi(|z|)\ovarphi(|z'|)\int h_{S_b}(z,\d z'')\log \frac{1}{|z''-z'|}\d z'\d z\\
&=\int_{\R^2}\int_{\R^2}\ovarphi(|z|)\ovarphi(|z'|)\E_{z}^W\Big[\log \frac{1}{|W_{H_{S_b}}-z'|}\Big]\d z'\d z\\
&=\int_{\R^2}\int_{\R^2}\ovarphi(|z|)\ovarphi(|z'|) \E_{|z|}\Big[\log \frac{1}{|\rho_{H_{S_b}}-z'|}\Big]\d z'\d z
=\frac{1}{2}\int_{\R^2}\ovarphi(|z'|)\log \frac{1}{|b-z'|}\d z',
\end{align*}
where the second equality uses $\ovarphi(|z'|)$ to remove the angular part of $W_{H_{S_b}}$. From the last two displays, 
\begin{align}
&\eqspace -4\pi \int_{\R^2}\ovarphi (|z|)\E_{|z|}\left[\int_0^{T_b}\ovarphi (\rho_s)\d s\right]\d z +2\int_{\R^2}\ovarphi(|z|) \log^+ (|z|/b)\d z\notag\\
&=4\int_{\R^2}\int_{\R^2}\ovarphi(|z|)\ovarphi(|z'|) \log |z-z'|\d z'\d z+2\int_{\R^2}\d z'\ovarphi(|z'|)\log \frac{1}{|b-z'|}.\label{finalid:1}
\end{align}
Additionally, for $\mathcal E(\hvarphi)$ defined by \eqref{def:Ef}, 
we have
\begin{align}\label{finalid:2}
-4\pi \int_{\R^2}\oevarphi(|z|)\d z=-4\pi \int_{\R^2}\mathcal E(\widehat{\varphi})(z)\d z=4\int_{\R^2}\hvarphi(z)\int_{\R^2}\hvarphi(z')\log |z-z'|\d z'\d z.
\end{align}

We are ready to prove the required identity. Applying \eqref{int:harmonic}, \eqref{finalid:1}, and \eqref{finalid:2} to \eqref{def:finalI:rad} yields that
\begin{align*}
\I_{\ref{def:finalI:rad}}(q)&=4 \int_{\R^2}\int_{\R^2}\varphi(z)\varphi(z') \log |z-z'|\d z'\d z-2\int_{\R^2}\left(\frac{\log 2-\log q}{2}+\lambda-\gamma_{\sf EM}\right)\ovarphi(|z|)\d z\\
&= \int_{\R^2}\int_{\R^2}\phi(z)\phi(z')\log |z/\two-z'/\two|\d z'\d z-\left(\frac{\log 2-\log q}{2}+\lambda-\gamma_{\sf EM}\right)\\
&= \int_{\R^2}\int_{\R^2}\phi(z)\phi(z') \log |z-z'|\d z'\d z-\left(\frac{\log 4-\log q}{2}+\lambda-\gamma_{\sf EM}\right),
\end{align*}
as required in \eqref{eq:finalI:rad}.
\end{proof}

For the proof of the following lemma, we can choose $q_{\ref{def:finalI:rad}}=q_{\ref{def:finalI:rad}}(\|\varphi\|_\infty,M_\varphi,\lambda)\in (0,\infty)$ such that the first term on the right-hand side of \eqref{def:finalI:rad} is positive and $\I_{\ref{def:finalI:rad}}(q_{\ref{def:finalI:rad}})>0$. Note that $\I_{\ref{def:finalI:rad}}(q)$ is strictly increasing in $q$.

\begin{lem}[Verification of the second condition of \eqref{cond:q}] \label{lem:rad:final}
Choose $\delta_\vep$ as in the statement of Lemma~\ref{lem:finalexpansion}.
Then for all $q\in( q_{\ref{def:finalI:rad}},\infty)$, there exists $\vep_{\ref{Ir:4}}=\vep_{\ref{Ir:4}}(\|\varphi\|_\infty,M_\varphi,q,\lambda)\in (0,\ovp)$ such that
\begin{align}\label{Ir:4}
\Phi_b(\vep^2q,f_\vep)\in (0,\infty)\quad \&\quad  \frac{\lv^2}{\Phi_b(\vep^2q,f_\vep)}=\frac{8\pi^2b}{\I_{\ref{def:finalI:rad}}(q)}+{\rm R}_{\ref{Ir:4}}(q,b,\vep)
\end{align}
for all $\vep\in (0,\vep_{\ref{Ir:4}})$ and  all $b$ satisfying the following condition:
\begin{align}\label{b:final}
M_\varphi\big(\e^{-\delta_\vep\frac{\pi}{2\sqrt{2\nu_\vep\beta_\vep   }}}\vee 
\e^{-\eta_{\ref{R:Ir2}}(\vep)^{-1/2}}\vee \e^{-\eta_{\ref{R:Ir3}}(\vep)^{-1/2}}\big)\leq b\leq  M_\varphi.
\end{align}
Here, the supremum of $|{\rm R}_{\ref{Ir:4}}(q,b,\vep)|$ over $b$ of the range in \eqref{b:final} is bounded and tends to zero as $\vep\to 0$.
Hence, under the choice of \eqref{setting:radial}, 
the second condition in \eqref{cond:q}  holds for all $[g,d]\subseteq [M_\varphi\e^{-\delta_\vep\frac{\pi}{2\sqrt{2\nu_\vep\beta_\vep   }}},  M_\varphi]$ and all $q\in (q_{\ref{def:finalI:rad}},\infty)$.
\end{lem}
\begin{proof}
Apply \eqref{Ir:1}--\eqref{Ir:3} to  \eqref{exp:Green}. For all $q\in (q_{\ref{def:finalI:rad}},\infty)$, $b$ and $\vep\in (0,\vep_{\ref{asympG:2-1}})$ such that 
$M_\varphi\e^{-\delta_\vep\frac{\pi}{2\sqrt{2\nu_\vep\beta_\vep   }}}\leq b\leq  M_\varphi$, we can write
\begin{align}
&\eqspace 1-\E_b\left[\int_0^\infty  \e^{-\vep^2q r} f_\vep(\rho_r)\kappa^q_\vep(\rho_r,b)\d r\right]\notag\\
&=\frac{\I_{\ref{def:finalI:rad}}(q)}{\log \vep^{-1}}+{\rm R}_{\ref{Ir:1}}(q,b,\vep)+{\rm R}_{\ref{Ir:2}}(q,b,\vep)+{\rm R}_{\ref{Ir:3}}(q,b,\vep).\label{Ir:4:+++1}
\end{align}
Here, for every fixed $q$, it follows from \eqref{R:Ir1}--\eqref{R:Ir3} that 
\begin{align}\label{Ir:4:+++2}
\log \vep^{-1}\cdot \sup_b \left|{\rm R}_{\ref{Ir:1}}(q,b,\vep)+{\rm R}_{\ref{Ir:2}}(q,b,\vep)+{\rm R}_{\ref{Ir:3}}(q,b,\vep)\right|\xrightarrow[\vep\to 0]{}0,
\end{align}
where $b$ ranges over \eqref{b:final}. In particular, the terms  $\e^{-\eta_{\ref{R:Ir2}}(\vep)^{-1/2}}$ and $\e^{-\eta_{\ref{R:Ir3}}(\vep)^{-1/2}}$ in \eqref{b:final} are chosen to satisfy \eqref{Ir:4:+++2} as we bound the $b$'s in the first terms in the bounds of ${\rm R}_{\ref{Ir:2}}(q,b,\vep)$ and ${\rm R}_{\ref{Ir:3}}(q,b,\vep)$. See also \eqref{eq:Tr} and \eqref{eq:Tl} for bounding that term of ${\rm R}_{\ref{Ir:2}}(q,b,\vep)$. By the last two displays, the choice of $q_{\ref{def:finalI:rad}}$, and Proposition~\ref{prop:Phibf}, we obtain the strict positivity of $\Phi_b(\vep^2q,f_\vep)\in(0,\infty) $ for all $q\in (q_{\ref{def:finalI:rad}},\infty)$ and $\vep\in (0,\vep(\|\varphi\|_\infty,M_\varphi,q,\lambda)\wedge \ovp)$. 

To obtain the expansion in \eqref{Ir:4}, note that applying  Lemma~\ref{lem:Phibasymp} and \eqref{Ir:4:+++1} to \eqref{Phiblambda:resolvent} gives
\[
\frac{\lv^2}{\Phi_b(\vep^2q,f_\vep)}\sim \frac{4\pi^2}{(\log \vep^{-1})^2}\cdot 2b\log \vep^{-1}\cdot \frac{\log \vep^{-1}}{\I_{\ref{def:finalI:rad}}(q)}=\frac{8\pi^2b}{\I_{\ref{def:finalI:rad}}(q)}.
\]
We still need to quantify the convergence rate for the above approximation and prove the asserted properties of ${\rm R}_{\ref{Ir:4}}(q,b,\vep)$. To this end, we use the decomposition in \eqref{Ir:4:+++1} and the rate of convergence in  \eqref{Ir:4:+++2} in the elementary form of the resolvent equation $A^{-1}-(A-\vep)^{-1}=-\vep [A(A-\vep)]^{-1}$ for the left-hand side of \eqref{Ir:4:+++1}. For this term and the term $\Phi_b(\vep^2 q)^{-1}$ by \eqref{asymp:Phibminus_bdd}, the growth of $|\log b|$ can be controlled again by restricting $b$ to the range in \eqref{b:final}. The proof is complete.
\end{proof}

\begin{prop}\label{prop:Sconv}
For all $q\in (q(\varphi,\lambda),\infty)$,
\[
\lim_{\vep\to 0}\sup_{y:0<|y|\leq M_\varphi}\left|\S^\beta_\vep(q,y)-\frac{4\pi }{\log(q/\beta)}\right|=0.
\]
\end{prop}
\begin{proof}
By \eqref{SS:approx}, it remains to show that for an appropriate choice of the cutoff functions $\chi_\vep$, we have
\begin{align}
\lim_{\vep\to 0}\sup_{y:0<|y|\leq M}\left|\ooS(q,|y|)-\frac{4\pi }{\log(q/\beta)}\right|=0. 
\end{align}
Here, recall that $\ooS(q,|y|)$ is defined in \eqref{main:aa}. We also recall that $m_\vep$ is chosen above \eqref{eq2:main}.

We choose the cutoff functions $\chi_\vep$ such that
\[
m_\vep=M_\varphi\big(\e^{-\delta_\vep\frac{\pi}{2\sqrt{2\nu_\vep\beta_\vep   }}}\vee 
\e^{-\eta_{\ref{R:Ir2}}(\vep)^{-1/2}}\vee \e^{-\eta_{\ref{R:Ir3}}(\vep)^{-1/2}}\big).
\]
Then by Lemma~\ref{lem:rad:final},
\begin{align*}
\lim_{\vep\to 0}\sup_{y:0<|y|\leq M}\left|\ooS(q,|y|)-\int_0^\infty \frac{8\pi^2b}{ 
\I_{\ref{def:finalI:rad}} (q) }\ovarphi(b)\d b\right|=0. 
\end{align*}
The integral in the foregoing display can be evaluated as follows: by Lemma~\ref{lem:exponent}, 
\[
\int_0^\infty \frac{8\pi^2b}{ 
\I_{\ref{def:finalI:rad}} (q) }\ovarphi(b)\d b=\frac{16\pi^2}{\log (q/\beta)}\int_0^\infty \ovarphi(b)b\d b=\frac{4\pi}{\log (q/\beta)}
\]
by polar coordinates and the identity $\int \varphi(|z|)=1/2$. The proof is complete.
\end{proof}

Recall that the required limiting semigroup is characterized by \eqref{def:Rlambda}, and the approximate semigroups under consideration are defined by $P^\beta_{\vep;t}f(x)$ in \eqref{def:FK}.

\begin{prop}
For $q\in (q(\varphi,\lambda),\infty)$ and $x\neq 0$, 
\begin{align}
\lim_{\vep\to 0}\int_0^\infty \e^{-qt}P^\beta_{\vep;t}f( x)\d t=\lim_{\vep\to 0}\int_0^\infty \e^{-qt}P^\beta_{t}f(x)\d t.\label{eq:main11} 
\end{align}
\end{prop}
\begin{proof}
By \eqref{Lap:int}, it remains to calculate 
\begin{align*}
\lim_{\vep\to 0}\int_{\R^2}\d y\varphi (y) \left(\int_0^\infty \d t\e^{-qt}P_t(x/\two,\vep y)\right)\times \S^\beta_\vep(q,y)\times \left(\int_0^\infty \d t\e^{-qt}P_tF(0)\right),
\end{align*}
where $F(z)=f(\two z)$. By Proposition~\ref{prop:Sconv} and dominated convergence, the limit is equal to 
\begin{align}
&\eqspace 
\int_{\R^2}\d y\varphi(y)\int_0^\infty \d t\e^{-qt}P_t(x/\two)\times\left(\frac{4\pi}{\log (q/\beta)}\right)\times \left(\int_0^\infty \d t\e^{-qt}P_tF(0)\right)\notag\\
&=G_q( x)\times\frac{4\pi}{\log (q/\beta)}\times \left(\int_0^\infty \d t\e^{-qt}P_tF(0)\right)\label{lim:integral}
\end{align}
since
\[
\left(\int_{\R^2}\d y\varphi(y)\right)
\int_0^\infty \d t\e^{-qt}P_t(x/\two)=\int_0^\infty\d t \e^{-qt}\frac{1}{4\pi t}\exp\left(-\frac{|x|^2}{4t}\right).
\]
The last term in \eqref{lim:integral} can be written as
\[
\int_0^\infty \d t\e^{-qt}P_tF(0)=\int_0^\infty \d t\e^{-qt}\E^W_0[f(\two W_t)]=\int_0^\infty \d t\e^{-qt}\E_0^W[f(W_{2t})].
\]
Hence, the density implied by \eqref{lim:integral} from ranging over $f$ is 
\[
G_q( x)\frac{4\pi}{\log (q/\beta)}G_q(z).
\]
We have proved that the limit in \eqref{lim:integral} coincides with the last term in \eqref{def:Rlambda}. 
The proof is complete. 
\end{proof}

\begin{proof}[End of the proof of Theorem~\ref{thm:main1}]
It remains to translate the convergence of Laplace transforms in \eqref{eq:main11} to the required pointwise convergence. First, note that if $f,g$ are nonnegative and defined on $[0,\infty)$ and $f$ is increasing, then $f\star g$ is increasing. Hence, for the simple case $f\equiv \1$, the last term in \eqref{eq:FK2} of $P^{\beta}_{\vep;t}f(x)$ is increasing by the monotonicity of the $\d \tau$-integral there. Hence, the required pointwise convergence holds by \eqref{eq:main11}, an extension of L\'evy's continuity theorem and Helly's selection principle  (cf. \cite[the proofs of Theorem~4.1 and Corollary~4.2]{CCC}). 

For the case of general nonnegative $f\in \B_b(\R^2)$, we need to consider the equicontinuity of $P^{\beta}_{\vep;t}f(x)$ in $t$. By \eqref{eq:FK1}, it is enough to show the equicontinuity of 
\begin{align}\label{equicontinuity}
t\mapsto \lv \int_0^t\d s\int_{\R^2} \d yP_{t-s}(x, y) \varphi_\vep( y)\E_y[\e^{A^o_\vep(s)}f(W_{s})].
\end{align}
To see this property, we notice that by Leibniz's rule, the derivative of this function is
\begin{align}\label{final:derivative}
&\lv \varphi_\vep(x)\E_x[\e^{A^o_\vep(t)}f(W_{t})]+\lv \int_0^t\d s\int_{\R^2} \d y\frac{\partial}{\partial t}P_{t-s}(x, y) \varphi_\vep( y)\E_y[\e^{A^o_\vep(s)}f(W_{s})],
\end{align}
where
\[
\frac{\partial}{\partial t}P_{t-s}(x,y)=\left(-\frac{1}{2\pi (t-s)^2}+\frac{1}{2\pi (t-s)}\cdot\frac{|x-y|^2}{2(t-s)^2}\right)\exp\left(-\frac{|x-y|^2}{2(t-s)}\right). 
\]
Since $x\neq 0$, the first term in \eqref{final:derivative} is zero for all small $\vep$ independent of $t$. 
In this case, the derivative of the function in $\eqref{equicontinuity}$ can be bounded by the increasing function
\begin{align}\label{equicontinuity1}
t\mapsto \lv \int_0^t\d s\int_{\R^2} \d y\left|\frac{\partial}{\partial t}P_{t-s}(x, y)\right| \varphi_\vep( y)\E_y[\e^{A^o_\vep(s)}].
\end{align}
The monotonicity of \eqref{equicontinuity1} ensures that  this function can be bounded by the Laplace transform on compacts; see \eqref{Lap:bdd}. On the other hand, the Laplace transform is convergent, if we recall \eqref{Lap:int} and notice that $s\mapsto \partial P_{t-s}(x,0)/\partial t$ is integrable at $t-$. The convergence holds by \eqref{def:Fvep:Lap} and Proposition~\ref{prop:Sconv}.  We have proved the equicontinuity of the function in \eqref{equicontinuity}. This is enough to ensure the continuity of the limiting function. The proof is complete. Note that for a sequence of increasing functions that converges to a continuous function, the convergence is uniform on compacts, and so the convergence mode in Theorem~\ref{thm:main1} can be improved this way. 
\end{proof}

\section{Application to the multi-particle delta-Bose gas}
Throughout this section, we fix an integer $N\geq 3$ and use the following notation. Write states of $N$ many particles in the plane as $x=(x^1,x^2,\cdots,x^N)$, with $x^j\in \R^{2}$ for all $1\leq j\leq N$. Components of other vector states are indexed by superscripts in the same way. Accordingly, $\R^{2m}$ is understood as $m$ products of $\R^2$. For the pairwise interactions, $\ms I_N$ denotes the set of $\bi=(i\prime,i)\in \{1,\cdots,N\}^2$ such that $i\prime>i$, with
\[
E\defeq |\ms I_N|.
\]
We often regard $\bi$ as a set and do not distinguish between sets or points for the use. For example, for $\bi,\bj\in \ms I_N$ such that $\bi\neq \bj$ and $\bi\cap\bj\neq \varnothing$, $x^{\bi\cap \bj}=x^\ell$ for $\ell\in \bi\cap \bj$. In addition, for any $\bi\in \ms I_N$, define a unitary linear transformation between $(x^{i\prime},x^{i})\in \R^4$ and $(x^{\bi\prime},x^{\bi})\in \R^4$ by 
\begin{align}\label{unitary}
x^{\bi\prime}=\frac{x^{i\prime }+x^i}{\two},\quad 
x^{\bi}=\frac{x^{i\prime }-x^i}{\two}\Longleftrightarrow x^{i\prime}=\frac{x^{\bi\prime }+x^{\bi}}{\two},\quad 
x^{i}=\frac{x^{\bi\prime }-x^{\bi}}{\two}.
\end{align}

Let $B=(B^1,\cdots,B^N)$ be a $2N$-dimensional standard Brownian motion. Our goal in this section is to obtain the limit as $\vep\to 0$ of the following Feynman--Kac semigroups as approximate solutions to the $N$-body delta-Bose gas in two dimensions:
\begin{align}\label{def:Qvep}
P_{\vep,t}^{\lambda,N}f(x_0)\defeq
\E^B_{x_0}\Bigg[\exp\Bigg\{\lv \sum_{\bi\in \ms I_N}\int_0^t \varphi_\vep(B^{\bi}_r)\d r\Bigg\}f(B_t)\Bigg],\quad f\in \B_b(\R^{2N}),
\end{align}
where $B^\bi$ is a two-dimensional standard Brownian motion defined by the notation in \eqref{unitary}.

\subsection{Poissonian representations of exponential functionals}\label{sec:first}
For fixed $q\in (0,\infty)$, introduce a mean-field interacting system $\{(\Lambda_t)_{t\geq 0},(\xi_n)_{n\geq 0}\}$ such that $(\Lambda_t)$ is a Poisson process with rate $q$ and jump times $0<S_1<S_2<\cdots$, and $\xi=(\xi_n)$ is an independent Markov chain on $\ms I_N$ such that 
$\P(\xi_1=\bj|\xi_0=\bi)=(M-1)^{-1}$ for all $ \bj\neq \bi\in \ms I_N$.
The uniform distribution $\mu_E$ on $\ms I_N$ is the stationary distribution of $\xi=(\xi_n)$.

The general result in the following proposition is the starting point to reduce \eqref{def:Qvep} to the semigroups for two particles under attractive interactions. Here and in what follows, we use the convention that a product over an empty set is equal to $1$.

\begin{prop}\label{prop:poisson}
For all Borel measurable functions $\varphi(\bi ,r):\ms I_N\times \R_+\to \R_+$, it holds that
\begin{align}\label{exp:id}
\exp\Bigg\{\int_0^t\sum_{\bi\in \ms I_N} \varphi(\bi,r)\d r\Bigg\}=\e^{qt}\E^\xi_{\mu_E}\left[\prod_{n=1}^{\Lambda_t}\frac{\varphi(\xi_n,S_n)}{q/E_n}\exp\Bigg\{\int_{S_{n}\wedge t}^{S_{n+1}\wedge t}\varphi(\xi_n,r)\d r\Bigg\}\right],
\end{align}
where $E_1=E$ and $E_n=E-1$ for all integers $n\geq 2$.
\end{prop}
\begin{proof}
By dominated convergence and monotone convergence, it is enough to assume that
$\varphi(\bi,r)$ is bounded away from zero and from above. In this proof, we work with a Poisson random measure $\mathcal M$ on $\ms I_N\times \R_+$ with intensity measure
\[
\Bigg(\frac{1}{E}\sum_{\bi\in \ms I_N} \delta_\bi(\d \bj)\Bigg)\otimes  (q\d r)=\frac{q}{E} \sum_{\bi\in \ms I_N}\delta_\bi(\d \bj)\otimes \d r,
\]
where $\delta_\bi$ denotes the delta measure at $\bi$ defined on $\ms I_N$. 

The proof is to consider the following application of the exponential formula for Poisson random measures \cite[Proposition~XII.1.12 on p.476]{RY}:
\begin{align}
\exp\Bigg\{-q t+\int_0^t\sum_{\bi\in \ms I_N} \varphi(\bi,r)\d r\Bigg\}&=\exp\left\{\sum_{\bi\in \ms I_N}\frac{q}{E}\int_0^t \big(\e^{\log \frac{\varphi(\bi,r)}{q/E}}-1\big) \d r\right\}\notag\\
&=\E\left[\exp\left\{\int_{\ms I_N\times (0,t]} \log \frac{\varphi(\bj,r)}{q/E}\mathcal M(\d \bj,\d r) \right\}\right].\label{poisson}
\end{align}
Write  $\mathcal E(t)\defeq \mathcal E_0(t)$ for the right-hand side of \eqref{poisson}, where
\[
\mathcal E_{s'}(s)\defeq\E\left[\exp\left\{\int_{\ms I_N\times (0,s]} \log \frac{\varphi(\bj,s'+r)}{q/E}\mathcal M(\d \bj,\d r) \right\}\right].
\]
The rest of the proof is to show that $\mathcal E(t)$ can be written as the expectation in \eqref{exp:id} by iterating a first jump recurrence equation. The first step is  the following identity: 
\begin{align}
\mathcal E(t)
&=\e^{-qt}+\E_{\mu_E}^\xi\left[\frac{\varphi(\xi_1,S_1)}{q/E_1}\mathcal E_{S_1}(t-S_1)
;S_1\leq t\right],\label{Et}
\end{align}
where $E_1=E$.
 To see \eqref{Et}, note that $(\xi_1,S_1)$ under $\P^\xi_{\mu_E}$ has the same law as the atom $(\xi',S')$ of $\mathcal M$ at  the first jump time of the Poisson process $t\mapsto \mathcal M(\ms I_N\times (0,t])$ \cite[Proposition~2 on p.7]{Bertoin}, since the law of $\xi_1$ under $\P^\xi_{\mu_E}$ is the uniform distribution by stationarity and $S_1$ is distributed as the exponential random variable with mean $q^{-1}$. Conditioning on the first jump time of $\mathcal M$ proves \eqref{Et}.

The next step is to put $\mathcal E_{s'}(s)$ in a form that allows for further iterations of \eqref{Et}. Roughly speaking, the goal is to condition appropriately such that only a sequence of non-consecutive $\ms I_N$-components of atoms of $\ms I_N$ is present under the expectation.  Given any $\bi\in \ms I_N$, conditioning on the first hitting time $\tau_\bi$  of $\{\bi\}^\complement$ of the Poisson point process associated with $\mathcal M$ leads to
\begin{align}
\begin{split}\label{Es's:dec}
\mathcal E_{s'}(s)&=\E\left[\exp\left\{\int_{\{\bi\}\times (0,s]} \log \frac{\varphi(\bi,s'+r)}{q/E}\mathcal M(\d \bj,\d r) \right\};\tau_\bi> s\right]\\
&\eqspace+\E\left[\exp\left\{\int_{\ms I_N\times (0,\tau_\bi]} \log \frac{\varphi(\bj,s'+r)}{q/E}\mathcal M(\d \bj,\d r)\right\}\mathcal E_{s'+\tau_\bi}(s-\tau_\bi) ;\tau_\bi\leq s\right].
\end{split}
\end{align}
The two terms on the right-hand side of \eqref{Es's:dec} can be calculated as follows. For the first term,
note that Poisson random measures of disjoint sets are independent, and $\tau_\bi$ depends only on $\mathcal M\rest (\{\bi \}^\complement\times \R_+)$ and is distributed as an exponential random variable with mean $[q(E-1)/E]^{-1}$. Hence,
\begin{align}
&\eqspace\E\left[\exp\left\{\int_{\{\bi\}\times (0,s]} \log \frac{\varphi(\bi,s'+r)}{q/E}\mathcal M(\d \bj,\d r) \right\};\tau_\bi> s\right]\notag\\
&=\exp\left\{\int_0^s \left(\varphi(\bi,s'+r)-\frac{q}{E}\right)\d r\right\}\exp\left\{-\frac{q(E-1)}{E}s\right\}\notag\\
&=\E\left[\exp\left\{\int_0^{s} \varphi(\bi,s'+r)\d r\right\};S_1>s\right],\label{Es's:dec1}
\end{align}
where the first equality follows from independence and the exponential formula. For the other term in \eqref{Es's:dec}, apply \cite[Proposition~2 on p.7]{Bertoin} again, and so if $\xi'$ is such that $(\xi',\tau_\bi)$ is an atom of $\mathcal M$, then $\xi'$ and $\tau_\bi$ are independent and $\xi'$ is uniformly distributed on $\ms I_N\setminus\{\bi\}$. We get
\begin{align}
&\eqspace\E\left[\exp\Bigg\{\int_{\ms I_N\times (0,\tau_\bi]} \log \frac{\varphi(\bj,s'+r)}{q/E}\mathcal M(\d \bj,\d r)\Bigg\}\mathcal E_{s'+\tau_\bi}(s-\tau_\bi) ;\tau_\bi\leq s\right]\notag\\
&=\E\left[\exp\Bigg\{\int_{\{\bi\}\times (0,\tau_\bi)} \log \frac{\varphi(\bi,s'+r)}{q/E}\mathcal M(\d \bj,\d r)\Bigg\}\frac{\varphi(\xi',s'+\tau_\bi)}{q/E}\mathcal E_{s'+\tau_\bi}(s-\tau_\bi) ;\tau_\bi\leq s\right]\notag\\
&=\int_0^s \frac{q(E-1)}{E}\e^{-q(E-1)r'/E}\exp\Bigg\{\int_0^{r'} \left(\varphi(\bi,s'+r)-\frac{q}{E}\right)\d r\Bigg\}\frac{1}{E-1}\sum_{\bj:\bj\neq \bi}\frac{\varphi(\bj,s'+r')}{q/E}\mathcal E_{s'+r'}(s-r')\d r'\notag\\
&=\int_0^s q\e^{-qr'}\exp\Bigg\{\int_0^{r'} \varphi(\bi,s'+r)\d r\Bigg\}\frac{1}{E-1}\sum_{\bj:\bj\neq \bi}\frac{\varphi(\bj,s'+r')}{q/(E-1)}\mathcal E_{s'+r'}(s-r')\d r'\notag\\
&=\E^\xi_\bi\left[\exp\Bigg\{\int_0^{S_1} \varphi(\bi,s'+r)\d r\Bigg\}\frac{\varphi(\xi_1,s'+S_1)}{q/E_2}\mathcal E_{s'+S_1}(s-S_1);S_1\leq s\right],\label{Es's:dec2}
\end{align}
where $E_2=E-1$. Applying \eqref{Es's:dec1} and \eqref{Es's:dec2} to \eqref{Es's:dec} leads to the following identity for all $\bi,s,s'$:
\begin{align}
\begin{split}\label{def:Eis}
\mathcal E_{s'}(s)&= \E\left[\exp\left\{\int_0^{ s} \varphi(\bi,s'+r)\d r\right\};S_1>s\right]\\
&\eqspace +\E_\bi^\xi\left[\exp\left\{\int_0^{S_1} \varphi(\bi,s'+r)\d r\right\}\frac{\varphi(\xi_1,s'+S_1)}{q/E_2}\mathcal E_{s'+S_1}(s-S_1);S_1\leq s\right].
\end{split}
\end{align}
Note that although $\mathcal E_{s'}(s)$ does not depend on $\bi$, the decomposition on the right-hand side does. 

We are ready to conclude the proof. Applying \eqref{def:Eis} to \eqref{Et} gives
\begin{align}
\begin{split}
\mathcal E(t)&=\e^{-qt}+\E_{\mu_E}^\xi\left[\frac{\varphi(\xi_1,S_1)}{q/E_1}\exp\left\{\int_{S_1\wedge t}^{S_2\wedge t} \varphi(\xi_1,r)\d r\right\}
;S_1\leq t<S_2\right]\\
&\eqspace+\E_{\mu_E}^\xi\Bigg[\frac{\varphi(\xi_1,S_1)}{q/E_1}\exp\left\{\int_{S_1\wedge t}^{S_2\wedge t} \varphi(\xi_1,r)\d r\right\}\frac{\varphi(\xi_2,S_2)}{q/E_2}\mathcal E_{S_2}(\xi_2,t-S_2);S_2\leq t\Bigg].\label{Et:iterate}
\end{split}
\end{align}
Since $\P(S_n\leq t)=\P(\Lambda_t\geq n)\leq \e^{qt(\e^c-1)-cn}$ for all $c\in (0,\infty)$ by Markov's inequality,  \eqref{exp:id} follows by iterating with \eqref{def:Eis} indefinitely as in \eqref{Et:iterate}.  The proof is complete.
\end{proof}

Proposition~\ref{prop:poisson} allows for a series expansion of $Q^{\lambda,N}_{\vep,t}f(x_0)$ from \eqref{def:Qvep}. To state the series, write
\begin{align}\label{def:E+F}
\mathcal E^\bi_{\vep;s,t}&= \exp\left\{\lv\int_s^t \varphi_\vep(B^{\bi}_r) \d r \right\}\quad\&\quad \mathcal L^\bi_{\vep;s,t}=\int_{s}^{t}\lv^2 \mathcal E_{\vep;s,\tau}^{\bi} \varphi_\vep(B^{\bi}_{\tau})\d \tau.
\end{align}
Then recall that, conditioned on the number $m$ of jumps of a Poisson process by time $t$, the jump times within $[0,t]$ are jointly distributed as the order statistics of $m$ independent uniform random variables over $[0,t]$ \cite[Theorem~2.4.6 on p.79]{Norris}. Hence, by Proposition~\ref{prop:poisson} with $\varphi(\bi,r)=\lv \varphi_\vep(B_r^\bi)$ and by conditioning on the number of jumps of $\Lambda_t$, we get 
\begin{align}
\begin{split}
P_{\vep,t}^{\beta,N}f(x_0)&=
\E^B_{x_0}[f(B_t)]+\sum_{m=1}^\infty \left(\prod_{n=1}^mE_n\right)\int_{0< s_1< s_2< \cdots< s_m< t}\E^\xi_{\mu_E}\big[P^{\xi_1,\cdots,\xi_m}_{\vep;s_1,\cdots,s_m,t}f(x_0)\big]\d \bs s_m ,\label{P:series}
\end{split}
\end{align}
where  $\d \bs s_m=\d s_1\cdots \d s_m$, and for all $\bi_1\neq  \cdots\neq \bi_m$, $0< s_1<\cdots< s_m< t$, and $f\in \B_b(\R^{2N})$,
\begin{align}
\begin{split}
P^{\bi_1,\cdots,\bi_m}_{\vep;s_1,\cdots,s_m,t}f(x_0)\,&\!\defeq\! \E^B_{x_0}\big[ \varphi_\vep(B^{\bi_1}_{s_1})\cdot \lv\mathcal E^{\bi_1}_{\vep;s_1,s_2} \cdot \varphi_\vep(B^{\bi_{2}}_{s_{2}})\cdot \lv\mathcal E^{\bi_2}_{\vep;s_2,s_3}\cdot \varphi_\vep(B^{\bi_{3}}_{s_{3}})\\
&\eqspace\eqspace \cdots 
\lv\mathcal E^{\bi_{m-1}}_{\vep;s_{m-1},s_m} \cdot \varphi_\vep(B^{\bi_{m}}_{s_{m}})
\cdot \lv\mathcal E^{\bi_m}_{\vep;s_m,t}f(B_t)\big]\notag
\end{split}\\
\begin{split}
&=\E^B_{x_0}\big[ \varphi_\vep(B^{\bi_1}_{s_1})\cdot (\lv+\mathcal L^{\bi_1}_{\vep;s_1,s_2}) \cdot \varphi_\vep(B^{\bi_{2}}_{s_{2}})\cdot (\lv+\mathcal L^{\bi_2}_{\vep;s_2,s_3}) \cdot \varphi_\vep(B^{\bi_{3}}_{s_{3}})\\
&\eqspace\eqspace \cdots 
(\lv+\mathcal L^{\bi_{m-1}}_{\vep;s_{m-1},s_m})  \cdot \varphi_\vep(B^{\bi_{m}}_{s_{m}})
\cdot (\lv+\mathcal L^{\bi_m}_{\vep;s_{m},t})f(B_t)\big].
\label{sgp}
\end{split}
\end{align}
In more detail, the $m$-th term on the right-hand side of  \eqref{P:series} follows from conditioning on $\{\Lambda_t=m\}$, 
and \eqref{sgp} follows from the forward expansion in \eqref{taylor1}. An expansion of the product in \eqref{sgp} gives 
\begin{align}\label{Pell:sum}
P^{\bi_1,\cdots,\bi_m}_{\vep;s_1,\cdots,s_m,t}f(x_0)=\sum_{\spin \in \{0,1\}^m} P^{\spin;\bi_1,\cdots,\bi_m}_{\vep;s_1,\cdots,s_m,t}f(x_0),
\end{align}
where, for all $\spin=(\sigma_1,\cdots,\sigma_m)\in \{0,1\}^m$, $\bi_1\neq \cdots\neq \bi_m$, $0< s_1<\cdots< s_m< t$, and $f\in \B_b(\R^{2N})$, 
\begin{align}
\begin{split}
&\eqspace \;P^{\spin;\bi_1,\cdots,\bi_m}_{\vep;s_1,\cdots,s_m,t}f(x_0)\\
&\defeq \E_{x_0}^B\big[ \varphi_\vep(B^{\bi_1}_{s_1})\cdot ((1-\sigma_1)\lv+\sigma_1\mathcal L^{\bi_1}_{\vep;s_1,s_2}) \cdot \varphi_\vep(B^{\bi_{2}}_{s_{2}})\cdot ((1-\sigma_2)\lv+\sigma_2\mathcal L^{\bi_2}_{\vep;s_2,s_3}) \cdot \varphi_\vep(B^{\bi_{3}}_{s_{3}})\\
&\eqspace\cdots 
((1-\sigma_{m-1})\lv+\sigma_{m-1}\mathcal L^{\bi_{m-1}}_{\vep;s_{m-1},s_m})  \cdot \varphi_\vep(B^{\bi_{m}}_{s_{m}})
\cdot ((1-\sigma_{m})\lv+\sigma_m\mathcal L^{\bi_m}_{\vep;s_{m},t})f(B_t)\big].
\label{sgp:sigma}
\end{split}
\end{align}

To apply \eqref{sgp:sigma} to \eqref{P:series}, write $\d \tau_\ell$ for the integrator of $\mathcal L^{\bi_\ell}_{\vep;s_\ell,s_{\ell+1}}$. By the Chapman--Kolmogorov equation, $P^{\spin;\bi_1,\cdots,\bi_m}_{\vep;s_1,\cdots,s_m,t}f(x_0)$ is an integral of a function of $s_1-\tau_0,\tau_1-s_1,\cdots,\tau_m-s_m,t-\tau_m$ with respect to the measure $\prod_{\ell:\sigma_\ell=0}\delta(\tau_\ell-s_\ell) \d \tau_1\cdots \d \tau_m $, where $\delta(\tau_\ell-s_\ell)$ is the Dirac delta function. Set
\begin{align}\label{variables}
\begin{split}
&\tau_0=0, \;t=s_{m+1},\; \bs u_{m+1}=(u_1,u_2,\cdots,u_{m+1})\\
&\mbox{and }\bs v_m=(v_1,v_2,\cdots,v_m) \mbox{ such that } 
u_\ell=s_\ell-\tau_{\ell-1}\mbox{ and } v_\ell=\tau_\ell-s_\ell.
\end{split}
\end{align}
Then the following integral representation holds: for some $I_{\vep;f,x_0}^{\spin;\bi_1,\bi_2,\cdots,\bi_m}(\bs u_{m+1},\bs v_m)$,
\begin{align}\label{rep:Pvep}
\begin{split}
\int_{0<s_1<s_2<\cdots<s_m<t} P^{\spin;\bi_1,\cdots,\bi_m}_{\vep;s_1,\cdots,s_m,t}f(x_0)\d \bs s_m
=\int_{\Delta_m(t)} I_{\vep;f,x_0}^{\spin;\bi_1,\cdots,\bi_m}(\bs u_{m+1},\bs v_m)\prod_{\ell:\sigma_\ell=0}\delta(v_\ell)\d \bs u_{m+1}\d \bs v_m,
\end{split}
\end{align}
where $\Delta_{m}(t)$ is the set of $u_\ell\in (0,t)$ and $v_\ell\in [0,t)$ such that $\sum_{\ell=1}^m (u_\ell+v_\ell)+u_{m+1}<t$.

\subsection{Coalescing-branching-type path ensembles}
Our goal in this subsection is to present an iterated-integral representation of $I_{\vep;f,x_0}^{\spin;\bi_1,\cdots,\bi_m}(\bs u_{m+1},\bs v_m)$ in \eqref{rep:Pvep} that will be used to pass the limit of the series \eqref{P:series} term by term as $\vep\to 0$. 

\subsubsection{Integral representations of transitions}
Let us begin with the following lemma for explicit representations of the essential transition mechanisms between $s_\ell$ and $s_{\ell+1}$ in \eqref{sgp:sigma}. Here and in what follows, we set
\begin{align}\label{def:svep}
\s_\vep^\beta(t;x,y)\defeq \lv^2\E_{\vep x/\two}^{W}\left[ \exp\left\{\lv\int_0^t \varphi_\vep(W_r) \d r \right\}; W_\tau=\vep y\right],
\end{align}
in the sense of densities in $y$, where $W$ is a planar Brownian motion, and we use the following notation for the usual multiplications of numbers:
\begin{align}\label{def:column}
\begin{bmatrix}
a_1\\
\vdots\\
a_n
\end{bmatrix}_\times
\defeq\,a_1 \cdot \cdots a_n.
\end{align}

\begin{lem}\label{lem:transition}
Fix $(x^{i\prime},x^i)\in \R^{4}$, $0<s<\infty$, $\bi=(i\prime,i)\in \ms I_N$, 
and a nonnegative Borel measurable function $F:\R^4\to \R$. Write $A^\bi_\vep(s)=\lv \int_0^s \varphi_\vep(B^\bi_r)\d r$. Then 
\begin{align}
\begin{split}
& \E_{(x^{i\prime},x^i)}^{B^{i\prime},B^i)}[\varphi_\vep(B^\bi_s)F(B^{i\prime}_s,B^i_s)]\\
&\eqspace=\int_{\R^4}\d (x_1^{i\prime},x_1^i)
\begin{bmatrix}
P_s(x^{i\prime},x_1^i+\vep (x_1^{i\prime}-x_1^i))\\
P_s(x^{i},x_1^i)
\end{bmatrix}_\times
\phi(x_1^{i\prime}-x_1^i)F\big(x^{i}_1+\vep(x^{i\prime}_1-x^i_1),x_1^i\big),
\end{split}\label{Int:1}\\
\begin{split}
&\E^{(B^{i\prime},B^i)}_{(x^{i}+\vep(x^{i\prime}-x^i),x^i)}[\lv^2\e^{A^\bi_\vep(s)}\varphi_\vep(B^\bi_{s})F(B^{i\prime}_{s},B^i_{s})]\\
&\eqspace=\int_{\R^4}\d (z',z) 
\begin{bmatrix}
P_{s}(\two x^{i}+\vep x^{\bi\prime}, z')\\
\vspace{-.3cm}\\
\s^\beta_\vep(s;x^{i\prime}-x^i,z)
\end{bmatrix}_\times  \varphi(z) F\left(\frac{z'}{\two}+\vep\frac{z}{\two},\frac{z'}{\two}-\vep\frac{z}{\two}\right),\label{Int:2}
\end{split}
\end{align}
where $x^{\bi\prime}$ is the transformation of $x^{i\prime}$ and $x^i$ defined in \eqref{unitary}. 
\end{lem}
\begin{proof}
To see \eqref{Int:1}, recall $\varphi(z)=\phi(\two z)$, and then apply the general formula 
\begin{align}\label{cov:vep}
\int_{\R^4} \phi_\vep(x-y)g(x,y)\d x \d y =\int_{\R^4}\phi(x-y')g(x,x+\vep(y'-x))\d x \d y' ,
\end{align}
which follows from changing variables: $y'=x+\vep^{-1}(y-x)$. The other identity \eqref{Int:2} applies the unitary transformation \eqref{unitary} to $(B^{i\prime},B^i)$ such that $(z',\widetilde{z})$ denotes the state of $(B^{\bi\prime},B^\bi)$. Then change variables from $\widetilde{z}$ to $z$ to remove the scaling in $\varphi_\vep$.  
\end{proof}

The rule in \eqref{Int:1} for changing variables  is to think of the $(i\prime)$-th particle being attracted to the $i$-th particle. Thus the state $x^i_1$ of $B^i_s$ is the major contribution under $\varphi_\vep(B^\bi_s)$. In \eqref{Int:2}, the second row in the $[\cdot]_\times$-column tracks  the attractive interactions of the two particles in terms of  the exponential local time of $B^{\bi}$. For the remaining terms of the integrand of \eqref{Int:2}, set $\vep$ to zero in the Gaussian kernel and the function $F$. Then the first row of the $[\cdot]_\times$-column shows the transition of $B^{\bi\prime}$ from $\two x^{i}=(x^{i}+x^{i})/\two$ to $z'$, and under $F$, the $(i\prime)$-th and $i$-th particles have the same state $z'/\two$. 

\subsubsection{Iterated-integral representations}\label{sec:integral}
In the following steps, we give the definition of the required iterated-integral representation of the function $I_{\vep;f,x_0}^{\spin;\bi_1,\bi_2,\cdots,\bi_m}(\bs u_{m+1},\bs v_m)$  in \eqref{rep:Pvep}, which is extended from a method in \cite{CSZ}. See Figure~\ref{fig:2}. \medskip  

\noindent {\bf Step~\hypertarget{Integral:1}{1}.} By \eqref{Int:1}, \eqref{Int:2} and the Chapman--Kolmogorov equation, the function $I_{\vep;f,x_0}^{\spin;\bi_1,\cdots,\bi_m}(\bs u_{m+1},\bs v_m)$ can be written as an integral of a product of  the Gaussian kernel $P_s(x)$ and the exponential local time density $\s_\vep^\beta(s;x,y)$. The integration is with respect to the product of the following measures:
\begin{align}
\begin{split}\label{def:productmeasure}
&\eqspace\phi(x_\ell^{i_\ell\prime}-x^{i_\ell}_\ell)\d_{\R^4} (x_\ell^{i_\ell\prime},  x_\ell^{i_\ell}),\; \d_{\R^2} x_\ell^k,\; 1\leq \ell\leq m,\; k\notin \bi_\ell;\\
&\eqspace \varphi(z_\ell)\d_{\R^4}(z'_\ell, z_\ell),\; \d_{\R^2} z_\ell^k,\; \sigma_\ell=1,\; k\notin \bi_\ell; \\
&\eqspace\d_{\R^2} x_{m+1}^k,\;1\leq k\leq N,
\end{split}
\end{align}
where $\phi(x_\ell^{i_\ell\prime}-x^{i_\ell}_\ell)$ arises from the mollifier at time $s_\ell$ as in \eqref{Int:1}, and $\varphi(z_\ell)$ arises from the mollifier at time $\tau_\ell$ as in \eqref{Int:2}. Note that $\bs x_{m+1}= \{x_\ell^k;1\leq \ell\leq m,\;1\leq k\leq N\}$ consists of $\R^2$-states at time $s_\ell$; $\bs z_{\spin}=\{z_\ell',z_\ell\}\cup \{z_\ell^{k};k\notin \mathbf i_\ell\}$ is the set of $\R^2$-states at time $\tau_\ell$. The subscripts $\R^4$ in \eqref{def:productmeasure}, for example, stress the domain of integration. \medskip

\noindent {\bf Step 2.}  In this step, we define a weighted graph 
$\mathcal G^{\spin;\bi_1,\cdots,\bi_m}_{\vep;x_0,\bs x_{m+1},\bs z_{\spin},\bs u_{m+1},\bs v_{m}}$
for fixed $x_0$, $\bs x_{m+1}$, $\bs z_{\spin}$, and $0<s_1\leq \tau_1<s_2\leq \tau_2<\cdots<s_m\leq \tau_m<t$, such that the product of the edge weights is the one of which $I_{\vep;f,x_0}^{\spin;\bi_1,\cdots,\bi_m}(\bs u_{m+1},\bs v_m)$ is an integral with respect to the grand product measure mentioned in Step~\hyperlink{Integral:1}{1} from \eqref{def:productmeasure}. We  also require that $s_\ell=\tau_\ell$ if and only if $\sigma_\ell=0$. Here, the change of variables from \eqref{variables} is in force, and the $\R^2$-components in $x_0$, $\bs x_{m+1}$, and $\bs z_{\spin}$ are all distinct. In view of \eqref{def:productmeasure}, we further think of $x_\ell^{i_\ell\prime}$ and $x_\ell^{i_\ell}$ as being glued to each other. The same rule applies to $z_\ell'$ and $z_\ell$.

The definition of this weighted graph is done all at once by specifying the vertices and the weighted edges according to \eqref{Int:1}, \eqref{Int:2}, and the following way of drawing edges. 
 
\begin{defi}[Drawing of edges]
By ``$w:v_0\edge v_1$'', we mean an edge of weight $w$ drawn as a straight line segment between two vertices $v_0$ and $v_1$. For $v_0\cedge v_1$, a coiled line segment is drawn.  \hfill $\blacksquare$
\end{defi}

For any $1\leq k\leq N$ and $\bi=(i\prime,i)\in \ms I_N$, set $k/\bi=i$ if $k= i\prime$ and $k/\bi=k$ otherwise. 
Now, between time $0$ and time $s_1$ and between time $s_m$ and time $t$, define the following weighted edges along with the vertices:
\begin{align}
\begin{split}\label{weight1}
\forall\;1\leq k\leq N,\quad P_{s_1}\big(x_0^{k},x_1^{k/\bi_1}+\1_{\{k=i_1\prime\}}\vep(x^{i_1\prime}_1-x^{i_1}_1)\big):&\; (x_0^{k},0)\edge (x^{k}_1,s_1);\\
\forall\;1\leq k\leq N,\quad P_{t-s_m}\big(x_m^{k/ \bi_m}+\1_{\{k=i_m\prime\}}\vep(x^{i_m\prime}_m-x^{i_m}_m),x_{m+1}^{k}\big):&\; (x_m^{k},s_m)\edge (x_{m+1}^{k},t).
\end{split}
\end{align}
For $\sigma_\ell=0$, the weighted edges and vertices are defined between $s_\ell$ and $s_{\ell+1}$ by
\begin{align}
\begin{split}\label{weight2}
&\forall\;k\in \bi_\ell,\quad \lv^{1/2} P_{s_{\ell+1}-s_\ell}\big(x_\ell^{k/\bi_\ell}+\1_{\{k=i_\ell\prime\}}\vep(x^{i_\ell\prime}_\ell-x^{i_\ell}_\ell),x_{\ell+1}^{k/\bi_{\ell+1}}+\1_{\{k=i_{\ell+1}\prime\}}\vep(x^{i_{\ell+1}\prime}_{\ell+1}-x^{i_{\ell+1}}_{\ell+1})\big)\\
&\hspace{9.87cm}:\; (x_\ell^{k},s_\ell)\edge (x_{\ell+1}^{k},s_{\ell+1});\\
&\forall\; k\notin  \bi_\ell,\quad P_{s_{\ell+1}-s_\ell}\big(x_\ell^{k}
,x_{\ell+1}^{k/\bi_{\ell+1}}+\1_{\{k=i_{\ell+1}\prime\}}\vep(x^{i_{\ell+1}\prime}_{\ell+1}-x^{i_{\ell+1}}_{\ell+1})\big):
 (x_\ell^{k},s_\ell)\edge (x_{\ell+1}^{k},s_{\ell+1}).
\end{split}
\end{align}
For $\sigma_\ell=1$, we consider the time interval between $s_\ell$ and $\tau_\ell$ and the time period between $\tau_\ell$ and $s_{\ell+1}$, separately. In the first case,  with the transformation of $(x_\ell^{i_\ell\prime},x_\ell^{i_\ell})$ to $x_\ell^{\bi_\ell\prime}$ defined by \eqref{unitary}, 
\begin{align}
\begin{split}\label{weight3}
\begin{bmatrix}
P_{\tau_{\ell}-s_\ell}\big(\two x_\ell^{i_\ell\prime}+\vep x_\ell^{\bi_\ell\prime},z'_\ell\big)\\
\vspace{-.3cm}\\
\s_\vep^\beta \big(\tau_\ell-s_\ell;x_\ell^{i_{\ell}\prime}-x^{i_\ell}_\ell,z_\ell\big)
\end{bmatrix}_\times
& :\;(x_\ell^{i_\ell\prime}, x_\ell^{i_\ell},s_\ell)\cedge (z_\ell',z_{\ell},\tau_{\ell});\\
\forall\;k\notin \bi_\ell,\quad \;
P_{\tau_\ell-s_\ell}(x^k_\ell,z_{\ell}^k)&:\;  (x^k_\ell,s_\ell)\edge (z^k_{\ell},\tau_{\ell}).
\end{split}
\end{align}
Over the other time interval, only edges as straight line segments and Gaussian weights are defined:
\begin{align}
\begin{split}\label{weight4}
P_{s_{\ell+1}-\tau_\ell}\Big(\frac{z'_\ell}{\two}+\vep \frac{z}{\two},x_{\ell+1}^{i_\ell\prime}+\1_{\{i_\ell\prime=i_{\ell+1}\prime\}}\vep(x^{i_{\ell+1}\prime}_{\ell+1}-x^{i_{\ell+1}}_{\ell+1})\Big)&:\; (z'_\ell,\tau_\ell)\edge (x_{\ell+1}^{i_\ell\prime},s_{\ell+1});\\
P_{s_{\ell+1}-\tau_\ell}\Big(\frac{z'_\ell}{\two}-\vep \frac{z}{\two},x_{\ell+1}^{i_\ell}+\1_{\{i_\ell=i_{\ell+1}\prime\}}\vep(x^{i_{\ell+1}\prime}_{\ell+1}-x^{i_{\ell+1}}_{\ell+1})\Big)&:\; (z_\ell,\tau_\ell)\edge (x_{\ell+1}^{i_\ell},s_{\ell+1});\\
\forall\;k\notin \bi_\ell,\quad P_{s_{\ell+1}-\tau_\ell}\big(z_\ell^{k},x_{\ell+1}^{k/\bi_{\ell+1}}+\1_{\{k=i_{\ell+1}\prime\}}\vep(x^{i_{\ell+1}\prime}_{\ell+1}-x^{i_{\ell+1}}_{\ell+1})\big)&:\; (z_\ell^{k},\tau_\ell)\edge (x_{\ell+1}^{k},s_{\ell+1}).
\end{split}
\end{align}

This weighted graph $\mathcal G^{\spin;\bi_1,\cdots,\bi_m}_{\vep;x_0,\bs x_{m+1},\bs z_{\spin},\bs u_{m+1},\bs v_{m}}$, defined by \eqref{weight1}--\eqref{weight4}, depends on $\vep$ only through the weights. (The $\vep$-multiples in the Gaussian kernels arise from removing the scalings in mollifiers.) The unweighted graph is denoted by  $\mathcal G=\mathcal G^{\spin;\bi_1,\cdots,\bi_m}_{x_0,\bs x_{m+1},\bs z_{\spin},\bs u_{m+1},\bs v_{m}}$. 
In the next two steps, we define subgraphs of $\mathcal G$ by which $I_{\vep;f,x_0}^{\spin;\bi_1,\cdots,\bi_m}(\bs u_{m+1},\bs v_m)$ is an iterated integral. Each fold integrates over the space components of the vertices of one subgraph with respect to the corresponding measures in \eqref{def:productmeasure}. \medskip

\noindent {\bf Step \hypertarget{Integral:3}{3}.} In this step, we define the following unweighted subgraphs:
\begin{align}\label{def:Gell}
\mathcal G_\ell=\mathcal G^{\spin;\bi_1,\cdots,\bi_m}_{\ell;x_0,\bs x_{m+1},\bs z_{\spin}}(\bs u_{m+1},\bs v_{m}),\quad \ell=0,1,\cdots,m.
\end{align}
such that the edge sets are disjoint, and so are the vertex sets. 
The vertex set of $\mathcal G_\ell$ consists of the vertices in $\mathcal G$ attached to all of the edges of $\mathcal G_\ell$, but we exclude the leftmost vertices, namely, the vertices $v$'s which do not have any edges in $\mathcal G_\ell$ on the left-hand sides of $v$'s. 

\begin{defi}\label{def:path}
(1$\cc$) A {\bf path} in $\mathcal G$ is a sequence of non-repeating adjacent edges such that a vertex can be shared by at most two of these edges. It is {\bf entanglement-free} if no edges is drawn as a coiled line segment. An entanglement-free path is {\bf left-maximal} if it cannot be extended to an entanglement-free path by adding more edges drawn as straight line segments over smaller times. A {\bf usual path} begins at time $0$ and ends at time $t$ such that no edges share vertices with edges drawn as coiled line segments.\smallskip

\noindent (2$\cc$) The {\bf duration} of an edge with vertices of time components $\tilde{s}$ and $\hat{s}$
 is $|\tilde{s}-\hat{s}|$. The sum of the durations of the edges of a path $\mathcal P$ is the duration of this path and is denoted by $|\mathcal P|$.
 \hfill $\blacksquare$
\end{defi} 

Now, the subgraph $\mathcal G_0$ has edges of the two entanglement-free paths whose rightmost endpoints intersect at time $s_1$.
For $\ell=1,\cdots,m-1$, $\mathcal G_\ell$ has the following edges on both sides of $s_\ell$: (1) the edge drawn as the coiled line segment between $s_{\ell}$ and $\tau_{\ell}$ if $\sigma_\ell=1$, and (2) the edges of the unique pair of left-maximal entanglement-free paths, denoted by $\mathcal P'_\ell$ and $\mathcal P_{\ell}$, such that their rightmost endpoints are given by the vertex $(x^{i_{\ell+1}\prime},x^{i_{\ell+1}},s_{\ell+1})$. Plainly, $|\mathcal P'_\ell|$ and $|\mathcal P_\ell|$ are bounded below by $s_{\ell+1}-\tau_\ell$. These two entanglement-free paths can be concatenated for an entanglement-free path, called the {\bf spine} of $\mathcal G_\ell$ and denote by $\mathcal P_\ell'\oplus \mathcal  P_\ell$. Finally, $\mathcal G_m$ has the following edges: (1) the edge drawn as the unique coiled line segment between $s_m$ and $\tau_m$ if $\sigma_m=1$; (2) the edges of the two unique left-maximal entanglement-free paths between $\tau_{m-1}$ and $t$ such that their leftmost endpoints at time $\tau_{m-1}$ coincide; (3) the edges of the remaining left-maximal entanglement-free paths ending at time $t$. The union of the edge sets of these subgraphs is different from the edge set of $\mathcal G$ when usual paths are present. 

The key observation we need is the following immediate consequence of  the assumption $\bi_{\ell}\neq \bi_{\ell+1}$.

\begin{figure}[t]
\begin{center}
 \begin{tikzpicture}[scale=1.5]
    \draw[gray, thin] (0,0) -- (7.6,0);
    \foreach \i in {0.0, 0.5, 2.0, 2.5, 3.2, 3.7, 4.7, 6.4, 7.6} {\draw [gray] (\i,-.05) -- (\i,.05);}
    \draw (0.0,-0.07) node[below]{$\tau_0$};
    \draw (0.5,-0.07) node[below]{$s_1$};
    \draw (2.0,-0.07) node[below]{$s_2$};
    \draw (2.5,-0.07) node[below]{$\tau_2$};
    \draw (3.2,-0.07) node[below]{$s_3$};
    \draw (3.7,-0.07) node[below]{$\tau_3$};
    \draw (4.7,-0.07) node[below]{$s_4$};
    \draw (6.4,-0.07) node[below]{$s_5$};
    \draw (7.6,-0.07) node[below]{$s_6$};
    \draw [line width=2.5pt, color=red!80!white] (0.0,2.30) -- (0.5,2.20); %red:0-s1
    \draw [line width=2.5pt, color=red!80!white] (0.5,2.20) -- (2.0,2.20); %red:s1-tau1
    \draw [line width=2.5pt, color=red!80!white] (2.0,2.20) -- (2.5,1.80); %red:s2-tau2     
    \draw [line width=2.5pt, color=red!80!white] (2.5,1.80)--(3.2,2.10); %red:tau2-s3      
    \draw [line width=2.5pt, color=red!80!white] (2.5,1.25)--(3.2,1.35); %red:tau2-s3   
    \draw [line width=2.5pt, color=red!80!white] (3.2,1.35)--(3.7,1.05); %red:s3-tau3      
    \draw [line width=2.5pt, color=red!80!white] (3.2,2.10)--(3.7,2.02); %red:s3-tau3       
    \draw [line width=2.5pt, color=red!80!white] (3.7,1.05)--(4.7,2.25); %red:tau3-s4      
    \draw [line width=2.5pt, color=red!80!white] (3.7,2.02)--(4.7,2.35); %red:tau3-s4    
    \draw [line width=2.5pt, color=green!80!black] (0.0,0.40) -- (0.5,0.55); %gray:0-s1
    \draw [line width=2.5pt, color=green!80!black] (0.0,1.20) -- (0.5,0.65); %gray:0-s1
    \draw [line width=2.5pt, color=blue!70!white] (0.0,1.80) -- (0.5,1.15); %green    
    \draw [line width=2.5pt, color=blue!70!white] (0.5,0.65) -- (2.0,1.4); %green
    \draw [line width=2.5pt, color=blue!70!white] (0.5,1.15) -- (2.0,1.5); %green    
    \draw [line width=2.5pt, color=purple!60!white] (0.5,0.55) -- (2.0,0.15); %blue
    \draw [line width=2.5pt, color=purple!60!white] (2.0,0.15) -- (2.5,0.45); %blue
    \draw [line width=2.5pt, color=purple!60!white] (2.0,1.46)--(2.5,1.19); %blue
    \draw [line width=2.5pt, color=purple!60!white] (2.5,0.45)--(3.2,0.8); %blue
    \draw [line width=2.5pt, color=purple!60!white] (2.5,1.15)--(3.2,0.9); %blue
    \draw [line width=2.5pt, color=orange] (3.7,0.45)--(4.7,1.25); %orange
    \draw [line width=2.5pt, color=orange] (3.7,0.55)--(4.7,1.75); %orange
    \draw [line width=2.5pt, color=orange] (4.7,1.25)--(6.4,0.3); %orange
    \draw [line width=2.5pt, color=orange] (4.7,1.75)--(6.4,0.4); %orange
    \draw [line width=2.5pt, color=red!80!white] (3.2,0.87)--(3.7,0.49); %red:s3-tau3
    \draw [line width=2.5pt, color=yellow!90!black] (4.7,2.35)--(6.4,2.40);  %gray:s4-tau4
    \draw [line width=2.5pt, color=yellow!90!black] (4.7,2.25)--(6.4,1.15);  %gray:s4-tau4
    \draw [line width=2.5pt, color=yellow!90!black] (6.4,0.3)--(7.6,0.45); %gray:s5-tau5
    \draw [line width=2.5pt, color=yellow!90!black] (6.4,0.4)--(7.6,0.85); %gray:s5-tau5
    \draw [line width=2.5pt, color=yellow!90!black] (6.4,1.15)--(7.6,1.55); %gray:s5-tau5
    \draw [line width=2.5pt, color=yellow!90!black] (6.4,2.40)--(7.6,2.15); %gray:s5-tau5
    \node at (0.0,0.40) {$\bullet$};
    \node at (0.0,1.20) {$\bullet$};
    \node at (0.0,1.80) {$\bullet$};
    \node at (0.0,2.30) {$\bullet$};
    \draw (0.0,0.40) node [left] {$x_0^1$};
    \draw (0.0,1.20) node [left] {$x_0^2$};
    \draw (0.0,1.80) node [left] {$x_0^3$};
    \draw (0.0,2.30) node [left] {$x_0^4$};
%0-s1%    
    \draw [thick, color=black] (0.0,0.40) -- (0.5,0.55);
    \draw [thick, color=black] (0.0,1.20) -- (0.5,0.65);
    \draw [thick, color=black] (0.0,1.80) -- (0.5,1.15);    
    \draw [thick, color=black] (0.0,2.30) -- (0.5,2.20); 
%s1%   
    \node at (0.5,0.55) {$\bullet$};
    \node at (0.5,0.65) {$\bullet$};
    \node at (0.5,1.15) {$\bullet$};
    \node at (0.5,2.20) {$\bullet$};
%s1-tau1%
    \draw [thick, color=black] (0.5,0.65) -- (2.0,1.4);
    \draw [thick, color=black] (0.5,1.15) -- (2.0,1.5);    
    \draw [thick, color=black] (0.5,2.20) -- (2.0,2.20);    
    \draw [thick, color=black] (0.5,0.55) -- (2.0,0.15);    
%s2%
    \node at (2.0,0.15) {$\bullet$};
    \node at (2.0,1.4) {$\bullet$};
        \node at (2.0,1.5) {$\bullet$};
    \node at (2.0,2.20) {$\bullet$};
%s2-tau2%
    \draw [thick, color=black] (2.0,0.15) -- (2.5,0.45);
    \draw [thick, color=black, snake=coil, segment length=4pt] (2.0,1.44) -- (2.5,1.20);    
    \draw [thick, color=black] (2.0,2.20) -- (2.5,1.80);    
%tau2%
    \node at (2.5,0.45) {$\bullet$};
    \node at (2.5,1.25) {$\bullet$};
        \node at (2.5,1.15) {$\bullet$};
    \node at (2.5,1.80) {$\bullet$};
%tau2-s3%
    \draw [thick, color=black] (2.5,0.45)--(3.2,0.8);  
    \draw [thick, color=black] (2.5,1.15)--(3.2,0.9);    
    \draw [thick, color=black] (2.5,1.25)--(3.2,1.35);  
    \draw [thick, color=black] (2.5,1.80)--(3.2,2.10);        
%s3%
    \node at (3.2,0.8) {$\bullet$};
        \node at (3.2,0.9) {$\bullet$};
    \node at (3.2,1.35) {$\bullet$};
    \node at (3.2,2.10) {$\bullet$};
%s3-tau3%
    \draw [thick, color=black, snake=coil, segment length=4pt] (3.2,0.85)--(3.7,0.50) ;
    \draw [thick, color=black] (3.2,1.35)--(3.7,1.05);    
    \draw [thick, color=black] (3.2,2.10)--(3.7,2.02);  
%tau3%
    \node at (3.7,0.45) {$\bullet$};
        \node at (3.7,0.55) {$\bullet$};
    \node at (3.7,1.05) {$\bullet$};
    \node at (3.7,2.02) {$\bullet$};
%tau3-s4%
    \draw [thick, color=black] (3.7,0.45)--(4.7,1.25);
    \draw [thick, color=black] (3.7,0.55)--(4.7,1.75);    
    \draw [thick, color=black] (3.7,1.05)--(4.7,2.25);    
    \draw [thick, color=black] (3.7,2.02)--(4.7,2.35);    
%s4%
    \node at (4.7,1.25) {$\bullet$};
    \node at (4.7,1.75) {$\bullet$};
    \node at (4.7,2.25) {$\bullet$};
        \node at (4.7,2.35) {$\bullet$};
%s4-tau4%
    \draw [thick, color=black] (4.7,1.25)--(6.4,0.3);
    \draw [thick, color=black] (4.7,1.75)--(6.4,0.4);   
    \draw [thick, color=black] (4.7,2.25)--(6.4,1.15); 
    \draw [thick, color=black] (4.7,2.35)--(6.4,2.40); 
%s5%
    \node at (6.4,0.3) {$\bullet$};
        \node at (6.4,0.4) {$\bullet$};
    \node at (6.4,1.15) {$\bullet$};
    \node at (6.4,2.40) {$\bullet$};
%s5-tau5%
    \draw [thick, color=black] (6.4,0.3)--(7.6,0.45);
    \draw [thick, color=black] (6.4,0.4)--(7.6,0.85);
    \draw [thick, color=black] (6.4,1.15)--(7.6,1.55);    
    \draw [thick, color=black] (6.4,2.40)--(7.6,2.15);
%t%
    \node at (7.6,0.45) {$\bullet$};
    \node at (7.6,0.85) {$\bullet$};
    \node at (7.6,1.55) {$\bullet$};
    \node at (7.6,2.15) {$\bullet$};
\end{tikzpicture}
\end{center}
\caption{The figure illustrates the paths of $B^1,B^2,B^3,B^4$ undergoing pairwise attractions behind $P^{\spin;\bi_1,\bi_2,\bi_3,\bi_4}_{\vep;s_1,s_2,s_2,s_3,t}f(x_0)$, where $\bi_1=(2,1)$, $\bi_2=(3,2)$, $\bi_3=(2,1)$, $\bi_4=(4,3)$, $\bi_5=(2,1)$, and $\spin=(0,1,1,0,0)$.}
\label{fig:2}
\end{figure}

\begin{lem}\label{lem:path}
For all $\ell=1,2,\cdots,m-1$, at least one of the paths $\mathcal P_\ell'$ and $\mathcal P_\ell$ starts at time $\tau_{\ell-1}$ or a smaller time. Hence, $|\mathcal P_\ell'\oplus \mathcal P_\ell|$  is bounded below by $2(s_{\ell+1}-\tau_{\ell})+(s_\ell-\tau_{\ell-1})$. 
\end{lem}

\noindent {\bf Step 4.}
To state the required iterated-integral representation of $I_{\vep;f,x_0}^{\spin;\bi_1,\cdots,\bi_m}(\bs u_{m+1},\bs v_m)$, write
\begin{align}\label{def:Pell}
\Pi^{\spin}_{\vep;\ell}(\bs u_{m+1},\bs v_{m})=\Pi^{\spin;\bi_1,\cdots,\bi_m}_{\vep;\ell;x_0,\bs x_{m+1},\bs z_{\spin}}(\bs u_{m+1},\bs v_{m})
\end{align}
for the product of the edge weights of $\mathcal G_\ell$ from the definition of $\mathcal G^{\spin;\bi_1,\cdots,\bi_m}_{\vep;x_0,\bs x_{m+1},\bs z_{\spin},\bs u_{m+1},\bs v_{m}}$. For any function $\widetilde{F}$ of the spatial states in the vertices of a subgraph $\mathcal G_\ell$, $\int_{\mathcal G_\ell}\widetilde{F}$ denotes the integral with respect to the product of the measures in \eqref{def:productmeasure} associated with these spatial variables. Recall that $\Delta_m(t)$ is defined below \eqref{rep:Pvep}, and $\d \bs s_m=\d s_1\cdots \d s_m$. The next result follows from \eqref{rep:Pvep} and the above definitions by integrating out the usual paths with the Chapman--Kolmogorov equation. 

\begin{prop}
 It holds that
\begin{align}
\begin{split}\label{graphical}
&\eqspace \int_{0<s_1<\cdots<s_m<t}P^{\spin;\bi_1,\cdots,\bi_m}_{\vep;s_1,\cdots,s_m,t}f(x_0)\d \bs s_m\\
&=\int_{\Delta_m(t)} \left( \int_{\mathcal G_0}\Pi^\spin_{\vep;0}(\bs u_{m+1},\bs v_{m}) \cdots \int_{\mathcal G_m}\Pi^\spin_{\vep;m}(\bs u_{m+1},\bs v_{m}) f\right) \prod_{\ell:\sigma_\ell=0}\delta(v_\ell) \d \bs u_{m+1}\d \bs v_m,
\end{split}
\end{align}
where the shorthand notation in \eqref{def:Gell} and \eqref{def:Pell} is in force. 
\end{prop}

\subsection{Convergences to the limit series}
In this subsection, we prove the limit of $P^{\beta;N}_{\vep;t}f(x_0)$ as $\vep\to 0$. By \eqref{weight1}--\eqref{weight4}, the natural candidate for the limit of $P^{\mathbf 1_m;\bi_1,\cdots,\bi_m}_{\vep;s_1,\cdots,s_m,t}f(x_0)$ is by setting $\vep=0$ in the Gaussian weights and replacing $\s^\beta_\vep(\tau;x,y)$ with $\s^\beta(\tau)$; $P^{\spin;\bi_1,\cdots,\bi_m}_{\vep;s_1,\cdots,s_m,t}f(x_0)$ for $\spin\neq \mathbf 1_m$ extends to zero by setting $\lv=0$. By \eqref{Pell:sum}, this description defines $P^{\bi_1,\cdots,\bi_m}_{s_1,\cdots,s_m,t}f(x_0)$ as an \emph{extension} of $P^{\mathbf 1_m;\bi_1,\cdots,\bi_m}_{\vep;s_1,\cdots,s_m,t}f(x_0)$ to $\vep=0$. (We still need to prove that $P^{\bi_1,\cdots,\bi_m}_{s_1,\cdots,s_m,t}f(x_0)$ is also the \emph{limit} as $\vep\to 0$.) The series solution analogue to \eqref{P:series} is
\begin{align}
\begin{split}
P^{\beta;N}_tf(x_0)&\defeq \E^B_{x_0}[f(B_t)]+\sum_{m=1}^\infty \left(\prod_{n=1}^mE_n\right)\int_{0< s_1< s_2< \cdots< s_m< t}\E^\xi_{\mu_E}\big[P^{\xi_1,\cdots,\xi_m}_{s_1,\cdots,s_m,t}f(x_0)\big]\d \bs s_m.\label{Q:serieslim}
\end{split}
\end{align}

To state an iterated-integral expansion as in \eqref{graphical}, set
\begin{align}\label{def:Pell-lim}
\Pi_{\ell}(\bs u_{m+1},\bs v_{m})=\Pi^{\bi_1,\cdots,\bi_m}_{\ell;x_0,\bs x_{m+1},\bs z_{\spin}}(\bs u_{m+1},\bs v_{m})
\end{align}
as the product of the edge weights in $\Pi^\spin_{\vep;\ell}(\bs u_{m+1},\bs v_{m})$ \eqref{def:Pell} extended to $\vep=0$ as before \eqref{Q:serieslim}. With the shorthand notation in \eqref{def:Gell} and \eqref{def:Pell-lim}, we have 
\begin{align}
\begin{split}\label{graphical:lim}
&\eqspace \int_{0<s_1<\cdots<s_m<t}P^{\bi_1,\cdots,\bi_m}_{s_1,\cdots,s_m,t}f(x_0)\d \bs s_m\\
&=\int_{\Delta_m(t)} \left( \int_{\mathcal G_0}\Pi_{0}(\bs u_{m+1},\bs v_{m}) \cdots \int_{\mathcal G_m}\Pi_{m}(\bs u_{m+1},\bs v_{m}) f\right)  \d \bs u_{m+1}\d \bs v_m.
\end{split}
\end{align}
For this iterated-integral representation, the list of measures in \eqref{def:productmeasure} can be made simpler by replacing $\phi(x_\ell^{i_\ell\prime}-x^{i_\ell}_\ell)\d_{\R^4} (x_\ell^{i_\ell\prime},  x_\ell^{i_\ell})$ with  $\d_{\R^2}x_\ell^{i_\ell}$, and $\varphi(z_\ell)\d_{\R^4}(z_\ell',z_\ell)$ with $(1/2)\d_{\R^2} z'_\ell$, since setting $\vep=0$ removes the dependence on $x_\ell^{i_\ell\prime}$ in the Gaussian weights and $\s^\beta(\tau)$ does not depend on $z_\ell$.

\begin{thm}\label{thm:main2}
For all $t\in (0,\infty)$, $f\in \B_b(\R^{2N})$, and $x_0\in \R^{2N}$ such that $x_0^\ell\neq x_0^{\ell'}$ for all $\ell\neq \ell'$, it holds that $P^{\beta;N}_{\vep;t}f(x_0)\to P^{\beta;N}_tf(x_0)$. 
\end{thm}

The mode of convergence in this theorem is the second main result of this work. 
The crucial tool is the following lemma, which restates some methods from \cite[pp.407--409]{CSZ}. 

\begin{lem}\label{lem:CSZ}
{\rm (1$\cc$)} For all integers $m\geq 1$, it holds that 
\[
\int_{u_1,\cdots,u_{m+1}\in (0,1)}\prod_{\ell=1}^{m-1}\frac{1}{\sqrt{u_{\ell+1}(u_{\ell+1}+u_{\ell})}}\d \bs u_{m+1}\leq 2\e \cdot 32^{m-1}.
\]

\noindent {\rm (2$\cc$)} Fix $\spin\in  \{0,1\}^m$ for $m\geq 1$ and nonnegative functions $S_\ell(v_\ell)$ and $T_\ell(u_\ell)$. Then for all $q\in [0,\infty)$, 
\begin{align*}
&\eqspace \int_{\Delta_m(t)}\Bigg(\prod_{\ell:\sigma_\ell=1}S_\ell(v_\ell)\Bigg)\Bigg(\prod_{\ell=1}^{m-1} \frac{T_\ell(u_\ell) }{2u_{\ell+1}+u_\ell}\Bigg) \Bigg(\prod_{\ell=m}^{m+1}T_\ell(u_\ell)\Bigg)\prod_{\ell:\sigma_\ell=0}\delta(v_\ell)\d \bs u_{m+1}\d \bs v_{m}\\
&\leq t^2 \e^{qt}\Bigg(\prod_{\ell:\sigma_\ell=1}\int_0^t \e^{-qv}S_\ell(v)\d v\Bigg)\int_{u_1,\cdots,u_{m+1}\in (0,1)}
\Bigg(\prod_{\ell=1}^{m-1}\frac{T_\ell(u_\ell/t)}{\sqrt{u_{\ell+1}(u_{\ell+1}+u_{\ell})}}\Bigg)\Bigg(\prod_{\ell=m}^{m+1}T_\ell(u_\ell/t)\Bigg)\d \bs u_{m+1}.
\end{align*}
\end{lem}
\begin{proof}
(1$\cc$) By \cite[Lemma~5.4]{CSZ}, the sequence of decreasing functions $\{\phi_k\}_{k\in \Bbb Z_+}$ on $(0,1)$ defined  inductively by
\begin{equation} \label{eq:phik}
	\phi_0(v) = 1 \, \quad \text{and} \quad
	\phi_k(v) = \int_0^1 \frac{ 1 }{\sqrt{s(s + v)}}
	\phi_{k-1}(s)   \d s ,
	\quad \forall\; k \in \Bbb N, 
\end{equation}
satisfies
\begin{equation} \label{eq:estphiklast}
	\phi_k(v) \le
	\displaystyle 32^k  \e/\sqrt{v} ,
	\quad \forall \;v \in (0,1),\;k\in \Bbb N.
\end{equation} 
Hence,
\[
\int_{u_1,\cdots,u_{m+1}\in (0,1)}\prod_{\ell=1}^{m-1}\frac{1}{\sqrt{u_{\ell+1}(u_{\ell+1}+u_{\ell})}}\d \bs u_{m+1}=\int_0^1\phi_{m-1}(u_1)\d u_1\leq 2\e \cdot 32^{m-1},
\]
as required. \medskip

\noindent (2$\cc$) By the elementary inequality $a+b\geq \sqrt{ab}$ for nonnegative $a,b$, we have 
\[
2u_{\ell+1}+u_\ell\geq \sqrt{u_{\ell+1}(u_{\ell+1}+u_\ell)}. 
\]
Hence, given $q\in [0,\infty)$, it follows from the nonnegativity of $S_\ell$ and $T_\ell$ that
\begin{align}
&\eqspace \Bigg(\prod_{\ell:\sigma_\ell=1}S_\ell(v_\ell)\Bigg)\Bigg(\prod_{\ell=1}^{m-1} \frac{T_\ell(u_\ell) }{2u_{\ell+1}+u_\ell}\Bigg) \Bigg(\prod_{\ell=m}^{m+1}T_\ell(u_\ell)\Bigg)\notag\\
\begin{split}\label{Markov-type}
&\leq \e^{qt}\Bigg(\prod_{\ell:\sigma_\ell=1}\e^{-qv_\ell} S_\ell(v_\ell)\Bigg) \Bigg(\prod_{\ell=1}^{m-1} \frac{T_\ell(u_\ell)}{\sqrt{u_{\ell+1}(u_{\ell+1}+u_\ell)}}\Bigg)\Bigg(\prod_{\ell=m}^{m+1}
T_\ell(u_\ell)\Bigg).
 \end{split}
\end{align}
The required inequality follows from integrating both sides of the last inequality, using the definition of $\Delta_m(t)$, and changing variables by replacing $u_\ell$ with $tu_\ell$. 
\end{proof}

Let us proceed to the first step of the proof of Theorem~\ref{thm:main2} by showing the convergence of the series in \eqref{Q:serieslim}. Recall that $P^{\bi_1,\cdots,\bi_m}_{s_1,\cdots,s_m,t}f(x_0)$ has been defined at the beginning of this subsection.

\begin{prop}\label{prop:series1}
Fix $t\in (0,\infty)$, $\bi_1\neq  \cdots\neq \bi_m$ for $m\geq 1$, and 
$x_0\in \R^{2N}$ with $x_0^{i_1\prime}\neq x_0^{i_1}$. Then 
\begin{align}
\begin{split}
& \int_{0< s_1<  \cdots< s_m< t} P^{\bi_1,\cdots,\bi_m}_{s_1,\cdots,s_m,t}\1(x_0)\d \bs s_{m}\\
&\eqspace \leq \max_{s\in (0,\infty)} P_{2s}(x_0^{i_1\prime}-x_0^{i_1})\cdot 2\e t^2\e^{(m\vee 2)t} \cdot 32^{m-1}\cdot \left(\frac{4\pi}{\log [(m\vee 2)/\beta]}\right)^m.
\label{interseries}
\end{split}
\end{align}
\end{prop}
\begin{proof}
Recall the notation in Step~\hyperlink{Integral:3}{3} of Section~\ref{sec:integral}. By the Chapman--Kolmogorov equation, we get 
\begin{align}
\int_{\mathcal G_0}\Pi_0(\bs u_{m+1},\bs v_m)&\leq \max_{s\in (0,\infty)}P_{2s}(x_0^{i_1\prime}-x_0^{i_1}),\label{Pibdd1}\\
\begin{split}
\int_{\mathcal G_\ell}\Pi_\ell(\bs u_{m+1},\bs v_m)
&\leq \s^\beta(v_\ell)
\max_{y\in \R^2}P_{|\mathcal P'_\ell\oplus \mathcal P_\ell|}(y)
\leq \frac{\s^\beta(v_\ell)}{2u_{\ell+1}+u_\ell},\quad 1\leq \ell\leq m-1,\label{Pibdd2}
\end{split}\\
\int_{\mathcal G_m}\Pi_m(\bs u_{m+1},\bs v_m)&\leq \s_\beta(v_m), \label{Pibdd3}
\end{align}
where the second inequality in \eqref{Pibdd2} uses Lemma~\ref{lem:path}. Apply \eqref{Pibdd1}--\eqref{Pibdd3} to \eqref{graphical:lim}. Then \eqref{interseries} follows from \eqref{Lap:Tbeta} and Lemma~\ref{lem:CSZ} (2$\cc$) with $\spin=\mathbf 1_m$, $S_\ell=\s^\beta$, $T_\ell\equiv 1$, and $q=m\vee 2$.  
\end{proof}

For the next step of proving Theorem~\ref{thm:main2}, we need some properties of the Gaussian kernels. 

\begin{lem}\label{lem:phikbd}
{\rm (1$\cc$)} For all $x',x\in \R^2$ and $s,t,\vep\in (0,\infty)$, 
\begin{align}\label{ineq:gausstime}
\int_{\R^4}
\begin{bmatrix}
P_{t}(x',z+\vep(z'-z))\\
P_s(x,z)
\end{bmatrix}_{\times}
\phi(z'-z)\d (z',z)&=\int_{\R^2}P_{s+t}(x-x'-\vep \bar{z}) \phi(\bar{z})\d \bar{z}. 
\end{align}

\noindent {\rm (2$\cc$)} For all $\vep,\eta_\vep,M\in (0,\infty)$ such that $\vep/\eta_\vep^{1/2}\leq 1$, it holds that
\begin{align}\label{gauss:vep}
\sup_{t\geq \eta}\sup_{x\in \R^2}\sup_{|y|\leq M}|P_t(x+\vep y)-P_t(x)|\leq C(M)(\vep/\eta_\vep^{1/2})P_{2t}(x).
\end{align}
\end{lem}
\begin{proof}
By the change of variables for $z=(w'-\vep w)/\two$ and $z'-z=\two w$, we get
\begin{align*}
\int_{\R^4}
\begin{bmatrix}
P_{t}(x',z+\vep(z'-z))\\
P_s(x,z)
\end{bmatrix}_{\times}
\phi(z'-z)\d (z',z)
&=\int_{\R^4}P_{t}\left(x',\frac{w'+\vep w}{\two }\right) P_s\left(x,\frac{w'-\vep w}{\two}\right)\varphi(w)\d (w', w)\\
&=2\int_{\R^2}P_{s+t}\left(x'-\frac{\vep w}{\two} ,x+\frac{\vep w}{\two}\right) \varphi(w)\d w,
\end{align*}
where the second equality uses the Chapman--Kolmogorov equation. The last equality implies \eqref{ineq:gausstime} upon changing variables with $\bar{z}=\two w$.\medskip 

\noindent (2$\cc$) Write
\begin{align*}
P_t(x+\vep y)-P_t(x)&=\frac{1}{2\pi t}\exp\left\{-\frac{1}{2}\left|\frac{x}{t^{1/2}}\right|^2\right\}\left[\exp\left\{\frac{x}{t^{1/2}}\cdot\frac{\vep y}{t^{1/2}}-\frac{1}{2}\left|\frac{\vep y}{t^{1/2}}\right|^2\right\}-1\right].
\end{align*}
By the assumption on $y$ and $t$, $\vep |y|/t^{1/2}\leq (\vep/\eta_\vep^{1/2}) M$. Hence, to bound the right-hand side of the preceding equality, it is enough to consider, for $a,b\in \R^2$ such that $|b|\leq (\vep/\eta_\vep^{1/2}) M$,
 \begin{align*}
&\eqspace \left|\frac{1}{2\pi t}\exp\left\{-\frac{|a|^2}{2}\right\}\left[\exp\left\{a\cdot b-\frac{1}{2}|b|^2\right\}-1\right]\right|\\
&\leq \left|\frac{1}{2\pi t}\exp\left\{-\frac{|a|^2}{2}+a\cdot b\right\}\left[\exp\left\{-\frac{1}{2}|b|^2\right\}-1\right]\right|+\left|\frac{1}{2\pi t}\exp\left\{-\frac{|a|^2}{2}\right\}\left[\exp\left\{a\cdot b\right\}-1\right]\right|\\
&\leq \frac{1}{2\pi t}\exp\left\{-\frac{|a|^2}{2}+|a|(\vep/\eta_\vep^{1/2})  M\right\}\cdot \frac{(\vep^2/\eta_\vep) M^2}{2}+\frac{1}{2\pi t}\exp\left\{-\frac{a^2}{2}+|a|(\vep/\eta_\vep^{1/2}) M\right\}\cdot |a|(\vep/\eta_\vep^{1/2}) M,
\end{align*}
since $1-\e^{-x}\leq x$ and $\e^x-1\leq x\e^x$ for all $x\geq 0$.
The required bound in \eqref{gauss:vep} now follows from the last two displays.
\end{proof}

The next proposition is the central part of the proof of Theorem~\ref{thm:main2}. Set 
\[
\widehat{\S}^\beta_\vep(q)\defeq\sup_{x\in \supp(\varphi)}\lv^2\int_0^\infty \e^{-q\tau}\E_{\vep x}^{W}[\e^{A^o_\vep(\tau)}\varphi_\vep(W_\tau)]\d \tau.
\] 
Recall that $\widehat{\S}^\beta_\vep(q)$ converges to $4\pi /\log (q/\beta)$ for all large $q$ as $\vep\to 0$.

\begin{prop}\label{prop:series2}
Fix $t\in (0,\infty)$, a Borel measurable function $f$ satisfying $0\leq f\leq 1$, and $x_0\in \R^{2N}$ such that $x_0^{\ell'}\neq  x_0^\ell$ for all $1\leq \ell'\neq \ell\leq N$. Then
\begin{align}
\begin{split}
\Delta_{\ref{interseries}}(m,t;\vep)\defeq \sup_{\bi_1\neq \cdots\neq \bi_m}\Bigg|\int_{0< s_1< \cdots< s_m< t}\big(P^{\bi_1,\cdots,\bi_m}_{\vep;s_1,\cdots,s_m,t}f(x_0)- P^{\bi_1,\cdots,\bi_m}_{s_1,\cdots,s_m,t}f(x_0)\big)\d \bs s_m  \Bigg|  \label{interseries}
\end{split}
\end{align}
satisfies the following conditions:
\begin{align}
\begin{split}
\Delta_{\ref{interseries}}(m,t;\vep)&\leq   \max_{s\in (0,\infty)} P_{2s}(x_0^{i_1\prime}-x_0^{i_1})\times 2\e t^2\e^{qt}\times 32^{m-1}\\
&\eqspace \times \left[\big(\widehat{\S}^\beta_\vep(q) +\lv\big)^m+\left(\frac{4\pi}{\log (q/\beta)}\right)^m\right],\quad \forall\;q\in (\beta,\infty);\label{def:Delta1}
\end{split}\\
 \Delta_{\ref{interseries}}(m,t;\vep)&\to 0\quad \mbox{as $\vep\to 0$.}\label{def:Delta2}
\end{align}
\end{prop}
\begin{proof}
We prove the required bound by considering
\begin{align}
&\eqspace \Big|\int_{0< s_1< \cdots< s_m< t}\big( P^{\bi_1,\cdots,\bi_m}_{\vep;s_1,\cdots,s_m,t}f(x_0) -\ P^{\bi_1,\cdots,\bi_m}_{s_1,\cdots,s_m,t}f(x_0)\big)\d \bs s_m \Big|\notag\\
\begin{split}\label{series:main}
&\leq \sum_{\spin:\spin\neq \mathbf 1_m}\int_{0<s_1< \cdots< s_m< t}P^{\spin;\bi_1,\cdots,\bi_m}_{\vep;s_1,\cdots,s_m,t}\1(x_0)\d \bs s_m  \\
&\eqspace+\Big|\int_{0< s_1<  \cdots< s_m< t}\big(P^{\mathbf 1_m;\bi_1,\cdots,\bi_m}_{\vep;s_1,\cdots,s_m,t}f(x_0)- P^{\bi_1,\cdots,\bi_m}_{s_1,\cdots,s_m,t}f(x_0)\big)\d\bs s_m\Big|.
\end{split}
\end{align}

\noindent {\bf Step \hypertarget{PP:1}{1}.} In this step, we prove that for all $\spin\in \{0,1\}^m$ (not just $\spin\neq \bs 1_M$),
\begin{align}\label{seriesbdd:1}
\begin{split}
&\eqspace\int_{0< s_1<\cdots< s_m< t} P^{\spin;\bi_1,\cdots,\bi_m}_{\vep;s_1,\cdots,s_m,t}\1(x_0) \d \bs s_m\\
&\leq \max_{s\in (0,\infty)} P_{2s}(x_0^{i_1\prime}-x_0^{i_1})\cdot 2\e t^2\e^{qt}\cdot 32^{m-1}\cdot \lv^{m-\|\spin\|}\widehat{\S}^\beta_\vep(q)^{\| \spin\|},\quad\mbox{where $\|\spin\|\defeq\sum_{\ell=1}^m\sigma_\ell$.}
\end{split}
\end{align}
Then given the bound in \eqref{seriesbdd:1} and the general identity $(a+b)^m-a^m=\sum_{\spin:\spin\neq \mathbf 1_m}a^{m-\|\spin\|}b^{\|\spin\|}$,
\begin{align}\label{seriesbdd:1-1}
\begin{split}
&\eqspace\sum_{\spin:\spin \neq \mathbf 1_m}\int_{0< s_1< \cdots< s_m< t}P^{\spin;\bi_1,\cdots,\bi_m}_{\vep;s_1,\cdots,s_m,t}\1(x_0)\d\bs s_m \\
&\leq \max_{s\in (0,\infty)} P_{2s}(x_0^{i_1\prime}-x_0^{i_1})\cdot 2\e t^2\e^{qt}\cdot 32^{m-1}\cdot \big[\big(\widehat{\S}^\beta_\vep(q) +\lv\big)^m-\widehat{\S}^\beta_\vep(q)^m
\big],
\end{split}
\end{align}
which gives the bound we need for the first term in \eqref{series:main}. Moreover, by the property of $\widehat{\S}^\beta_\vep(q)$ recalled before this proposition, the right-hand side tends to zero.

The differences between the proof of \eqref{seriesbdd:1} and the proof of Proposition~\ref{prop:series1} arise from the possibility of dealing with $\spin\neq \mathbf 1_m$. The bounds in \eqref{Pibdd2} and \eqref{Pibdd3} need to be generalized. For $\ell=1,\cdots,m-1$, the spine of $\mathcal G_\ell$ has only one vertex adjacent to an edge drawn as a coiled line segment such that the vertex is not an endpoint of the spine. Hence, by Lemma~\ref{lem:phikbd} (1$\cc$), we can still use the Chapman--Kolmogorov equation and Lemma~\ref{lem:path} and get the following bounds: for all $\ell=1,\cdots,m-1$,
\begin{align}\label{Pibdd1+}
\int_{\mathcal G_\ell}\Pi^\spin_{\vep;\ell}(\bs u_{m+1},\bs v_m)\leq \left\{
\begin{array}{ll}
\displaystyle 
\lv^2\E^W_{\vep (x_\ell^{i_\ell\prime}-x^{i_\ell}_\ell)}[\e^{A^o_\vep(v_\ell)}\varphi_\vep(W_{v_\ell})]\cdot \frac{1}{2u_{\ell+1}+u_\ell},\quad \mbox{ if }\sigma_\ell=1;\\
 \vspace{-.2cm}\\
\displaystyle  \frac{\lv}{2u_{\ell+1}+u_\ell} ,\quad \mbox{ if }\sigma_\ell=0.
 \end{array}
 \right.
\end{align}
Similarly, 
\begin{align}\label{Pibdd2+}
\int_{\mathcal G_m}\Pi^\spin_{\vep;m}(\bs u_{m+1},\bs v_m)\leq \left\{
\begin{array}{ll}
\displaystyle 
\lv^2\E^W_{\vep (x_m^{i_m\prime}-x^{i_m}_m)}[\e^{A^o_\vep(v_m)}\varphi_\vep(W_{v_m})],\quad \mbox{ if }\sigma_m=1;\\
 \vspace{-.2cm}\\
\displaystyle  \lv ,\quad \mbox{ if }\sigma_m=0.
 \end{array}
 \right.
\end{align}

To involve $\widehat{\S}^\beta_\vep(q)$, we also need a slight generalization of the proof of Lemma~\ref{lem:phikbd} (2$\cc$) in using a Markov-type inequality for \eqref{Markov-type}. After the application of $\int_0^t\d v_\ell\e^{-qv_\ell}$ to the corresponding bound in \eqref{Pibdd1+} or \eqref{Pibdd2+} for $\sigma_\ell=1$, take the supremum over $x_\ell^{i_\ell\prime}-x_\ell^{i_\ell}\in \supp(\varphi)$. We get a bound as $\widehat{\S}^\beta_\vep(q)$. This argument can be done inductively in the order of $\ell=m,\cdots,1$. We have proved \eqref{seriesbdd:1}. \medskip

\begin{figure}[t]
\begin{center}
 \begin{tikzpicture}[scale=1.5]
    \draw[gray, thin] (0,0) -- (7.6,0);
    \foreach \i in {0.0, 0.5, 1.2, 2.0, 2.5, 3.2, 3.7, 4.7, 5.5, 6.4, 7.0, 7.6} {\draw [gray] (\i,-.05) -- (\i,.05);}
    \draw (0.0,-0.07) node[below]{$\tau_0$};
    \draw (0.5,-0.07) node[below]{$s_1$};
    \draw (1.2,-0.07) node[below]{$\tau_1$};
    \draw (2.0,-0.07) node[below]{$s_2$};
    \draw (2.5,-0.07) node[below]{$\tau_2$};
    \draw (3.2,-0.07) node[below]{$s_3$};
    \draw (3.7,-0.07) node[below]{$\tau_3$};
    \draw (4.7,-0.07) node[below]{$s_4$};
    \draw (5.5,-0.07) node[below]{$\tau_4$};
    \draw (6.4,-0.07) node[below]{$s_5$};
    \draw (7.0,-0.07) node[below]{$\tau_5$};
    \draw (7.6,-0.07) node[below]{$s_6$};
    \draw [line width=2.5pt, color=red!80!white] (0.0,2.30) -- (0.5,2.20); %red:0-s1
    \draw [line width=2.5pt, color=red!80!white] (0.5,2.20) -- (1.2,2.00); %red:s1-tau1
    \draw [line width=2.5pt, color=red!80!white] (1.2,2.00) -- (2.0,2.20); %red:tau1-s2     
    \draw [line width=2.5pt, color=red!80!white] (2.0,2.20) -- (2.5,1.80); %red:s2-tau2     
    \draw [line width=2.5pt, color=red!80!white] (2.5,1.80)--(3.2,2.10); %red:tau2-s3      
    \draw [line width=2.5pt, color=red!80!white] (2.5,1.25)--(3.2,1.35); %red:tau2-s3   
    \draw [line width=2.5pt, color=red!80!white] (3.2,1.35)--(3.7,1.05); %red:s3-tau3      
    \draw [line width=2.5pt, color=red!80!white] (3.2,2.10)--(3.7,2.02); %red:s3-tau3       
    \draw [line width=2.5pt, color=red!80!white] (3.7,1.05)--(4.7,2.25); %red:tau3-s4      
    \draw [line width=2.5pt, color=red!80!white] (3.7,2.02)--(4.7,2.35); %red:tau3-s4    
    \draw [line width=2.5pt, color=green!80!black] (0.0,0.40) -- (0.5,0.55); %gray:0-s1
    \draw [line width=2.5pt, color=green!80!black] (0.0,1.20) -- (0.5,0.65); %gray:0-s1
    \draw [line width=2.5pt, color=blue!70!white] (0.0,1.80) -- (0.5,1.15); %green    
    \draw [line width=2.5pt, color=blue!70!white] (0.5,0.60) -- (1.2,0.80); %green
    \draw [line width=2.5pt, color=blue!70!white] (0.5,1.15) -- (1.2,1.40); %green    
    \draw [line width=2.5pt, color=blue!70!white] (1.2,0.85) -- (2.0,1.4); %green
    \draw [line width=2.5pt, color=blue!70!white] (1.2,1.40) -- (2.0,1.5); %green
    \draw [line width=2.5pt, color=purple!60!white] (1.2,0.75) -- (2.0,0.15); %blue
    \draw [line width=2.5pt, color=purple!60!white] (2.0,0.15) -- (2.5,0.45); %blue
    \draw [line width=2.5pt, color=purple!60!white] (2.0,1.46)--(2.5,1.18); %blue
    \draw [line width=2.5pt, color=purple!60!white] (2.5,0.45)--(3.2,0.8); %blue
    \draw [line width=2.5pt, color=purple!60!white] (2.5,1.15)--(3.2,0.9); %blue
    \draw [line width=2.5pt, color=orange] (3.7,0.45)--(4.7,1.25); %orange
    \draw [line width=2.5pt, color=orange] (3.7,0.55)--(4.7,1.75); %orange
    \draw [line width=2.5pt, color=orange] (4.7,1.25)--(5.5,0.45); %orange
    \draw [line width=2.5pt, color=orange] (4.7,1.75)--(5.5,1.20); %orange
    \draw [line width=2.5pt, color=orange] (5.5,0.45)--(6.4,0.3); %orange
    \draw [line width=2.5pt, color=orange] (5.5,1.20)--(6.4,0.4); %orange
    \draw [line width=2.5pt, color=red!80!white] (3.22,0.85)--(3.7,0.49); %gray:s3-tau3
    \draw [line width=2.5pt, color=orange] (4.7,2.32)--(5.5,1.79);  %orange:s4-tau4
    \draw [line width=2.5pt, color=yellow!90!black] (5.5,1.75)--(6.4,1.15); %gray:tau4-s5
    \draw [line width=2.5pt, color=yellow!90!black] (5.5,1.85)--(6.4,2.40); %gray:tau4-s5
    \draw [line width=2.5pt, color=yellow!90!black] (6.4,0.34)--(7.0,0.99); %gray:s5-tau5
    \draw [line width=2.5pt, color=yellow!90!black] (6.4,1.15)--(7.0,1.85); %gray:s5-tau5
    \draw [line width=2.5pt, color=yellow!90!black] (6.4,2.40)--(7.0,2.30); %gray:s5-tau5
    \draw [line width=2.5pt, color=yellow!90!black] (7.0,0.9)--(7.6,0.45); %gray:tau5-t
    \draw [line width=2.5pt, color=yellow!90!black] (7.0,1.0)--(7.6,0.85); %gray:tau5-t
    \draw [line width=2.5pt, color=yellow!90!black] (7.0,1.85)--(7.6,1.55); %gray:tau5-t
    \draw [line width=2.5pt, color=yellow!90!black] (7.0,2.30)--(7.6,2.15); %gray:tau5-t
%%%%%%%%%%%%%%%%%%%%%%%
%0%
    \node at (0.0,0.40) {$\bullet$};
    \node at (0.0,1.20) {$\bullet$};
    \node at (0.0,1.80) {$\bullet$};
    \node at (0.0,2.30) {$\bullet$};
    \draw (0.0,0.40) node [left] {$x_0^1$};
    \draw (0.0,1.20) node [left] {$x_0^2$};
    \draw (0.0,1.80) node [left] {$x_0^3$};
    \draw (0.0,2.30) node [left] {$x_0^4$};
%0-s1%    
    \draw [thick, color=black] (0.0,0.40) -- (0.5,0.55);
    \draw [thick, color=black] (0.0,1.20) -- (0.5,0.65);
    \draw [thick, color=black] (0.0,1.80) -- (0.5,1.15);    
    \draw [thick, color=black] (0.0,2.30) -- (0.5,2.20); 
%s1%   
    \node at (0.5,0.55) {$\bullet$};
    \node at (0.5,0.65) {$\bullet$};
    \node at (0.5,1.15) {$\bullet$};
    \node at (0.5,2.20) {$\bullet$};
%s1-tau1%
    \draw [thick, color=black, snake=coil, segment length=4pt] (0.5,0.60) -- (1.2,0.80);
    \draw [thick, color=black] (0.5,1.15) -- (1.2,1.40);    
    \draw [thick, color=black] (0.5,2.20) -- (1.2,2.00);    
%tau1%
    \node at (1.2,0.75) {$\bullet$};
    \node at (1.2,0.85) {$\bullet$};
    \node at (1.2,1.40) {$\bullet$};
    \node at (1.2,2.00) {$\bullet$};
%tau1-s2%
    \draw [thick, color=black] (1.2,0.75) -- (2.0,0.15);
    \draw [thick, color=black] (1.2,0.85) -- (2.0,1.4);
    \draw [thick, color=black] (1.2,1.40) -- (2.0,1.5);    
    \draw [thick, color=black] (1.2,2.00) -- (2.0,2.20);  
%s2%
    \node at (2.0,0.15) {$\bullet$};
    \node at (2.0,1.4) {$\bullet$};
    \node at (2.0,1.5) {$\bullet$};
    \node at (2.0,2.20) {$\bullet$};
%s2-tau2%
    \draw [thick, color=black] (2.0,0.15) -- (2.5,0.45);
    \draw [thick, color=black, snake=coil, segment length=4pt] (2.0,1.44) -- (2.52,1.185);    
    \draw [thick, color=black] (2.0,2.20) -- (2.5,1.80);    
%tau2%
    \node at (2.5,0.45) {$\bullet$};
    \node at (2.5,1.25) {$\bullet$};
    \node at (2.5,1.15) {$\bullet$};
    \node at (2.5,1.80) {$\bullet$};
%tau2-s3%
    \draw [thick, color=black] (2.5,0.45)--(3.2,0.8);  
    \draw [thick, color=black] (2.5,1.15)--(3.2,0.9);    
    \draw [thick, color=black] (2.5,1.25)--(3.2,1.35);  
    \draw [thick, color=black] (2.5,1.80)--(3.2,2.10);        
%s3%
    \node at (3.2,0.80) {$\bullet$};
    \node at (3.2,0.90) {$\bullet$};
    \node at (3.2,1.35) {$\bullet$};
    \node at (3.2,2.10) {$\bullet$};
%s3-tau3%
    \draw [thick, color=black, snake=coil, segment length=4pt] (3.2,0.85)--(3.734,0.45) ;
    \draw [thick, color=black] (3.2,1.35)--(3.7,1.05);    
    \draw [thick, color=black] (3.2,2.10)--(3.7,2.02);  
%tau3%
    \node at (3.7,0.45) {$\bullet$};
        \node at (3.7,0.55) {$\bullet$};
    \node at (3.7,1.05) {$\bullet$};
    \node at (3.7,2.02) {$\bullet$};
%tau3-s4%
    \draw [thick, color=black] (3.7,0.45)--(4.7,1.25);
    \draw [thick, color=black] (3.7,0.55)--(4.7,1.75);    
    \draw [thick, color=black] (3.7,1.05)--(4.7,2.25);    
    \draw [thick, color=black] (3.7,2.02)--(4.7,2.35);    
%s4%
    \node at (4.7,1.25) {$\bullet$};
    \node at (4.7,1.75) {$\bullet$};
    \node at (4.7,2.25) {$\bullet$};
        \node at (4.7,2.35) {$\bullet$};
%s4-tau4%
    \draw [thick, color=black] (4.7,1.25)--(5.5,0.45);
    \draw [thick, color=black] (4.7,1.75)--(5.5,1.20);   
    \draw [thick, color=black, snake=coil, segment length=4pt] (4.7,2.30)--(5.5,1.80); 
%tau4%
    \node at (5.5,0.45) {$\bullet$};
    \node at (5.5,1.20) {$\bullet$};
    \node at (5.5,1.75) {$\bullet$};
        \node at (5.5,1.85) {$\bullet$};
%tau4-s5%
    \draw [thick, color=black] (5.5,0.45)--(6.4,0.3);
    \draw [thick, color=black] (5.5,1.20)--(6.4,0.4); 
    \draw [thick, color=black] (5.5,1.75)--(6.4,1.15);    
    \draw [thick, color=black] (5.5,1.85)--(6.4,2.40); 
%s5%
    \node at (6.4,0.3) {$\bullet$};
        \node at (6.4,0.4) {$\bullet$};
    \node at (6.4,1.15) {$\bullet$};
    \node at (6.4,2.40) {$\bullet$};
%s5-tau5%
    \draw [thick, color=black, snake=coil, segment length=4pt] (6.4,0.35)--(7.0,0.95);
    \draw [thick, color=black] (6.4,1.15)--(7.0,1.85);    
    \draw [thick, color=black] (6.4,2.40)--(7.0,2.30);
%tau5%
    \node at (7.0,0.9) {$\bullet$};
    \node at (7.0,1.0) {$\bullet$};
    \node at (7.0,1.85) {$\bullet$};
    \node at (7.0,2.30) {$\bullet$};
%tau5-t%
    \draw [thick, color=black] (7.0,0.9)--(7.6,0.45);
    \draw [thick, color=black] (7.0,1.0)--(7.6,0.85);    
    \draw [thick, color=black] (7.0,1.85)--(7.6,1.55); 
    \draw [thick, color=black] (7.0,2.30)--(7.6,2.15); 
%t%
    \node at (7.6,0.45) {$\bullet$};
    \node at (7.6,0.85) {$\bullet$};
    \node at (7.6,1.55) {$\bullet$};
    \node at (7.6,2.15) {$\bullet$};
\end{tikzpicture}
\end{center}
\caption{The figure illustrates the paths of $B^1,B^2,B^3,B^4$ undergoing pairwise attractions behind $P^{\mathbf 1_m;\bi_1,\bi_2,\bi_3,\bi_4}_{\vep;s_1,s_2,s_2,s_3,t}f(x_0)$, where $\bi_1=(2,1)$, $\bi_2=(3,2)$, $\bi_3=(2,1)$, $\bi_4=(4,3)$, and $\bi_5=(2,1)$.}
\label{fig:1}
\end{figure}

\noindent {\bf Step~\hypertarget{PP:2}{2}.} Condition \eqref{def:Delta1} holds by Proposition~\ref{prop:series1} and \eqref{seriesbdd:1-1}. To prove \eqref{def:Delta2}, we already have the property mentioned below \eqref{seriesbdd:1-1}, and so it remains to show that the supremum of the last term in \eqref{series:main} over all $\bi_1\neq \cdots\bi_m$ tends to zero.  We divide the proof into two further steps. In Step~\hyperlink{PP:21}{2.1}, we remove the $\vep$-correction terms in the Gaussian weights. In Step~\hyperlink{PP:22}{2.2}, we use the convergence of $\s_\vep^\beta$ to $\s^\beta$ to complete the proof. \medskip

\noindent {\bf Step~\hypertarget{PP:21}{2-1}.}
Choose $\eta_\vep=\vep^2 \log\vep^{-1} $. 
For all $\bs a\times \bs b=(a_1,\cdots,a_{m+1},b_1,\cdots,b_m)\in  \{0,1\}^{m+1}\times \{0,1\}^{m}$, let $\Delta_{\bs a\times \bs b}(t)$ denote the subset of $\Delta(t)$ such that $u_\ell\in [\eta_\vep t, t),v_{\ell'} \in [\eta_\vep t, t)$ for $\ell,\ell'$ such that $a_\ell=1,b_{\ell'}=1$, and $u_\ell\in (0,\eta_\vep t),v_{\ell'}\in [0,\eta_\vep t)$ otherwise.  Hence, $\Delta(t)$ is a disjoint union of the sets $\Delta_{\bs a\times \bs b}(t)$ for $\bs a\times \bs b$ ranging over $\in \{0,1\}^{m+1}\times \{0,1\}^m$. We set
\begin{align}
\begin{split}\label{graphical:vep}
&\eqspace \int_{0< s_1<\cdots< s_m< t}  P^{\mathbf 1_m;\bi_1,\cdots,\bi_m}_{\vep;s_1,\cdots,s_m,t;\bs a\times \bs b}f(x_0)\d \bs s_m\\
&=\int_{\Delta_{\bs a\times \bs b}(t)} \d \bs u_{m+1}\d \bs v_m  \int_{\mathcal G_0}\Pi^{\mathbf 1_m}_{\vep;0}(\bs u_{m+1},\bs v_{m}) \int_{\mathcal G_1}\Pi^{\mathbf 1_m}_{\vep;1}(\bs u_{m+1},\bs v_{m})\cdots \int_{\mathcal G_m}\Pi^{\mathbf 1_m}_{\vep;m}(\bs u_{m+1},\bs v_{m}) f.
\end{split}
\end{align}
By incorporating the indicator functions for $[\eta_\vep t,t)$, $(0,\eta_\vep t)$ and $[0,\eta_\vep t)$ for $u_\ell$ and $v_\ell$ when appropriate, it follows from  Lemma~\ref{lem:path}  and a generalization of the proof of  Lemma~\ref{lem:CSZ} (2$\cc$) as in Step~\hyperlink{PP:1}{1} that the foregoing iterated integral can be bounded by 
\begin{align}
\begin{split}\label{ensemble-est}
& \max_{s\in (0,\infty)}P_{2s}(x_0^{i_1\prime}-x_0^{i_1})\times  t^2 \e^{qt}\Bigg(\prod_{\stackrel{\scriptstyle \ell:\sigma_\ell=1}{b_\ell=1}}\sup_{x\in \supp(\varphi)}\int_0^{\eta_\vep t} \lv^2\E^W_{\vep x}[\e^{A^o_\vep(v)}\varphi_\vep(W_v)]\d v\Bigg)\\
&\times \Bigg(\prod_{\stackrel{\scriptstyle \ell:\sigma_\ell=1}{b_\ell=0}}\sup_{x\in \supp(\varphi)}\int_0^{ t}\lv^2\E^W_{\vep x}[\e^{A^o_\vep(v)}\varphi_\vep(W_v)]\d v\Bigg)\\
&\times\int_{u_1,\cdots,u_{m+1}\in (0,1)}\d \bs u_{m+1}\Bigg(\prod_{\ell=1}^{m-1}\frac{T_\ell(u_\ell/t)}{\sqrt{u_{\ell+1}(u_{\ell+1}+u_{\ell})}}\Bigg)\Bigg(\prod_{\ell=m}^{m+1}T_\ell(u_\ell/t)\Bigg),
\end{split}
\end{align}
where $T_\ell(\tilde{u}_\ell)=\1_{[\eta_\vep t,t)}(\tilde{u}_\ell)$ if $a_\ell=1$ and $T_\ell(\tilde{u}_\ell)=\1_{(0,\eta_\vep t)}(\tilde{u}_\ell)$ otherwise.  The term in the foregoing display tends to zero by dominated convergence and the following  implications of Proposition~\ref{prop:bounds}:
\begin{align}
&\sup_{x\in \supp(\varphi)}\lv^2\int_0^{\eta_\vep t} \e^{-qv}\E^W_{\vep x}[\e^{A^o_\vep(v)}\varphi_\vep(W_v)]\d v\to 0,\label{lim:etavep+++}\\
&\sup_{x\in \supp(\varphi)}\lv^2\int_0^{ t} \e^{-qv}\E^W_{\vep x}[\e^{A^o_\vep(v)}\varphi_\vep(W_v)]\d v\leq C(\|\varphi\|_\infty,\lambda)\e^{q(\|\varphi\|_\infty,\lambda)(t\vee 1)}.\notag
\end{align}
Note that \eqref{lim:etavep+++} uses $\eta_\vep$ for the property $\lv\log (\vep^{-2}\eta_\vep t)\to 0$. 
By \eqref{ensemble-est}, we have proved that 
\begin{align}\label{approx_ab:1}
\lim_{\vep\to 0}\int_{0< s_1<\cdots< s_m< t}  P^{\mathbf 1_m;\bi_1,\cdots,\bi_m}_{\vep;s_1,\cdots,s_m,t;\bs a\times \bs b}f(x_0)\d \bs s_m=0,\quad \forall\;\bs a\times \bs b\neq \mathbf 1_{m+1}\times \mathbf 1_m.
\end{align}

Now we focus on the case $\bs a\times \bs b=\mathbf 1_{m+1}\times \mathbf 1_m$. 
Denote by $\overline{\Pi}_{\vep;\ell}(\bs u_{m+1},\bs v_m)$ the product of the Gaussian edge weights in $\Pi^{\mathbf 1_m}_{\vep;\ell}(\bs u_{m+1},\bs v_m)$ with $\vep=0$. We stress that the edge weights defined by \eqref{weight3} remain in use in $\overline{\Pi}_{\vep;\ell}(\bs u_{m+1},\bs v_m)$. Then the difference between the iterated integral in \eqref{graphical:vep} for $\bs a\times \bs b=\mathbf 1_{m+1}\times \mathbf 1_m$
and the analogous integral from replacing $\Pi^{\mathbf 1_m}_{\vep;m}$ by $\overline{\Pi}_{\vep;m}$, given by
\[
\int_{\Delta_{\bs 1_{m+1}\times \bs 1_{m}}(t)} \d \bs u_{m+1}\d \bs v_m  \int_{\mathcal G_0}\Pi^{\mathbf 1_m}_{\vep;0}(\bs u_{m+1},\bs v_{m}) \int_{\mathcal G_1}\Pi^{\mathbf 1_m}_{\vep;1}(\bs u_{m+1},\bs v_{m})\cdots \int_{\mathcal G_m}\overline{\Pi}_{\vep;m}(\bs u_{m+1},\bs v_{m}) f,
\]
tends to zero.
To see the validity of this replacement, note that by Lemma~\ref{lem:phikbd},
\begin{align*}
&\left|\int_{\mathcal G_m}\Pi^{\mathbf 1_m}_{\vep;m}(\bs u_{m+1},\bs v_{m}) f-\int_{\mathcal G_m}\overline{\Pi}_{\vep;m}(\bs u_{m+1},\bs v_{m}) f\right| \leq \lv^2\E^W_{\vep (x_m^{i_m\prime}-x^{i_m}_m)}[\e^{A^o_\vep(v_m)}\varphi_\vep(W_{v_m})]\cdot E(\vep/\eta_\vep^{1/2})M_\phi,
\end{align*}
where $M_\phi$ is such that $\phi$ is supported in $|y|\leq M_\phi$. 
Here, $(\vep/\eta_\vep^{1/2})M_\phi$ arises whenever we remove $\vep$ from the Gaussian weight where the left endpoint of the edge is where two left-maximal entanglement-free paths are adjacent to an edge drawn as a coiled line segment. There are less than $E$ many such pairs. Note that we have integrated out the right endpoints of the Gaussian weights so that none of them appear in the above bound. Hence, the argument in Step~\hyperlink{PP:1}{1} shows the required replacement.

The differences between the above replacement and the remaining replacements are very minor. For $\ell=1,\cdots,m-1$,
removing $\vep$ from the Gaussian weight where $\mathcal P_\ell'$ and $\mathcal P_\ell$ join shows that
\begin{align*}
&\left|\int_{\mathcal G_\ell}\Pi^{\mathbf 1_m}_{\vep;\ell}(\bs u_{m+1},\bs v_{m}) -\int_{\mathcal G_\ell}\overline{\Pi}_{\vep;\ell}(\bs u_{m+1},\bs v_{m}) \right| \\
&\eqspace \leq \lv^2\E^W_{\vep (x_\ell^{i_\ell\prime}-x^{i_\ell}_\ell)}[\e^{A^o_\vep(v_\ell)}\varphi_\vep(W_{v_\ell})]\cdot (\vep/ \eta_\vep^{1/2})M_\phi \max_{y\in \R^2}P_{2|\mathcal P_{\ell}'\oplus \mathcal P_\ell|}(y),
\end{align*}
again by  Lemma~\ref{lem:phikbd}, where the maximum of the Gaussian kernel can be bounded by using Lemma~\ref{lem:path}. (More precisely,  we use Lemma~\ref{lem:phikbd} (2$\cc$) at the rightmost vertex of the spine.)
 The case $\ell=0$ only requires the bound of $(\vep/\eta_\vep^{1/2})M_\phi$. We have proved that 
\begin{align}\label{finalreplace:1}
&\lim_{\vep\to 0}\int_{0< s_1<\cdots< s_m< t}  \Big(P^{\mathbf 1_m;\bi_1,\cdots,\bi_m}_{\vep;s_1,\cdots,s_m,t;\mathbf 1_m\times \mathbf 1_{m+1}}f(x_0)-\widetilde{P}^{\bi_1,\cdots,\bi_m}_{\vep;s_1,\cdots,s_m,t}f(x_0)\Big)\d \bs s_m=0,
\end{align}
where $\widetilde{P}^{\bi_1,\cdots,\bi_m}_{\vep;s_1,\cdots,s_m,t}f(x_0)$ is defined by \eqref{graphical:vep} with $\bs a\times \bs b=\mathbf 1_{m+1}\times \mathbf 1_m$ and with $\Pi^{\mathbf 1_m}_{\vep;1}(\bs u_{m+1},\bs v_{m})$ replaced by $\overline{\Pi}_{\vep;\ell}(\bs u_{m+1},\bs v_{m})$ for all $\ell=0,1,\cdots,m$. Moreover, a slight modification of the proof of \eqref{approx_ab:1} shows that for $\overline{P}^{\bi_1,\cdots,\bi_m}_{\vep;s_1,\cdots,s_m,t}f(x_0)$ defined by replacing $\Delta_{\mathbf 1_{m+1}\times \mathbf 1_m}(t)$ with $\Delta(t)$ in $\widetilde{P}^{\bi_1,\cdots,\bi_m}_{\vep;s_1,\cdots,s_m,t}f(x_0)$, 
\begin{align}\label{finalreplace:2}
&\lim_{\vep\to 0}\int_{0< s_1<\cdots< s_m< t} \Big(\widetilde{P}^{\mathbf 1_m;\bi_1,\cdots,\bi_m}_{\vep;s_1,\cdots,s_m,t;\mathbf 1_m\times \mathbf 1_{m+1}}f(x_0)-\overline{P}^{\bi_1,\cdots,\bi_m}_{\vep;s_1,\cdots,s_m,t}f(x_0)\Big)\d \bs s_m=0.
\end{align}
By \eqref{finalreplace:1} and \eqref{finalreplace:2}, we have removed 
all the $\vep$-correction terms in the Gaussian weights.
 \medskip

\noindent {\bf Step~\hypertarget{PP:22}{2-2}.}
Let us begin by comparing $\overline{P}^{\bi_1,\cdots,\bi_m}_{\vep;s_1,\cdots,s_m,t}f(x_0)$ and $P^{\bi_1,\cdots,\bi_m}_{s_1,\cdots,s_m,t}f(x_0)$. With respect to the simplified list of measures discussed below \eqref{graphical:lim}, $\overline{P}^{\bi_1,\cdots,\bi_m}_{\vep;s_1,\cdots,s_m,t}f(x_0)$ differs from $P^{\bi_1,\cdots,\bi_m}_{\vep;s_1,\cdots,s_m,t}f(x_0)$ only by the use of 
\[
\overline{\s}^\beta_\vep(t)\defeq 2\dint \phi(x)\s_\vep^\beta(t;x,y)\varphi(y)\d x\d y
\]
instead of $\s^\beta_\vep(t)$. To see this simplification, note that since all the Gaussian weights in \eqref{weight1}--\eqref{weight4} do not depend on $x_\ell^{i_\ell\prime}$ and $z_\ell$ after setting $\vep=0$, these simplified measures lead to the new weight
\[
 \dint \phi(x^{i_\ell\prime}_\ell-x^{i_\ell}_\ell)\s_\vep^\beta \big(\tau_\ell-s_\ell;x_\ell^{i_{\ell}\prime}-x^{i_\ell}_\ell,z_\ell\big)\d x^{i_\ell\prime}_\ell \varphi(z_\ell)\d z_\ell=\frac{1}{2}\overline{\s}^\beta_\vep(\tau_\ell-s_\ell)
\] 
for \eqref{weight3}. On the other hand, in $P^{\bi_1,\cdots,\bi_m}_{s_1,\cdots,s_m,t}f(x_0)$, the corresponding weight  under the simplified list of measures is $\tfrac{1}{2}\s_\beta(\tau_\ell-s_\ell)$.
In this step, we work with these characterizations under the simplified list of measures and show that 
\begin{align}\label{final_limit!!}
\lim_{\vep\to 0}\int_{0< s_1<\cdots< s_m< t} \Big(\overline{P}^{\bi_1,\cdots,\bi_m}_{\vep;s_1,\cdots,s_m,t}f(x_0)-P^{\bi_1,\cdots,\bi_m}_{s_1,\cdots,s_m,t}f(x_0)\Big)\d \bs s_m=0.
\end{align}

Recall that for probability measures $\mu_n,\mu,\nu_n,\nu$ on $\R$, $\mu_n\otimes \nu_n\Rightarrow \mu\otimes \nu$ if and only if $\mu_n\Rightarrow \mu$ and $\nu_n\Rightarrow \nu$ \cite[Theorem~2.8 on p.23]{Bill}. This result extends immediately to finite measures with total masses converging in $\R_+$ 
 by normalization.  Recall also the discussion above \eqref{equicontinuity} for convergences of Laplace transforms. 
We have $\int_0^t \overline{\s}_\vep^\beta(v)\d v\to\int_0^t \s^\beta(v)\d v$ for all $t\geq 0$ from the monotonicity of these functions of $t$ and the convergence of the Laplace transform. Hence, by the Portmanteau theorem \cite[Theorem~2.1 on p.16]{Bill}, the following convergence property holds: For all bounded continuous functions $F(\bs u_{m+1},\bs v_m)$ with a compact support,  
\[
\lim_{\vep\to 0}\int_{\R^{m+1}_+\times \R^m_+}  F(\bs u_{m+1},\bs v_m)\prod_{\ell=1}^m\overline{\s}_\vep^\beta(v_\ell)\d \bs u_{m+1}\d \bs v_m= \int_{\R_+^{m+1}\times \R_+^m} F(\bs u_{m+1},\bs v_m) \prod_{\ell=1}^m\s^\beta(v_\ell) \d \bs u_{m+1}\d \bs v_m.
\]
Recall that we write $\Delta(t)$ for the bounded set of $u_1,\cdots,u_{m+1}\in (0,t)$ and $v_1,\cdots,v_m\in [0,t)$ such that $\sum_{\ell=1}^m (u_\ell+v_\ell)+u_{m+1}<t$.
This set is a continuity set of the limiting product measure. Hence,  for all bounded continuous functions $F(\bs u_{m+1},\bs v_{m})$ on $\R_+^{m+1}\times \R_+^m$, we have
\begin{align}\label{conv:port}
\lim_{\vep\to 0}\int_{\Delta(t)} F(\bs u_{m+1},\bs v_m)\prod_{\ell=1}^m\overline{\s}_\vep^\beta(v_\ell) \d \bs u_{m+1}\d \bs v_m= \int_{\Delta(t)} F(\bs u_{m+1},\bs v_m)\prod_{\ell=1}^m\s^\beta(v_\ell) \d \bs u_{m+1}\d \bs v_m.
\end{align}

Write $F_0(\bs u_{m+1},\bs v_m)$ for the iterated integral in \eqref{graphical:lim} over the vertices of $\mathcal G_0,\mathcal G_1,\cdots,\mathcal G_m$ after factoring out $\s_\beta(v_\ell)$ for all $\ell$. (This function $F_0(\bs u_{m+1},\bs v_m)$ is the same as the one when the same factorization is done for the analogous iterated integral for $\overline{P}^{\mathbf 1_m;\bi_1,\cdots,\bi_m}_{\vep;s_1,\cdots,s_m,t}f(x_0)$.) 
We close the proof of this proposition by proving that, with $F_0(\bs u_{m+1},\bs v_{m})$ chosen above, for every $\eta\in(0,1)$, there exists $\vep_0$ such that for all $\vep\in (0,\vep_0)$,
\begin{align}\label{final_claim!!}
\left|\int_{\Delta(t)} F_0(\bs u_{m+1},\bs v_m)\prod_{\ell=1}^m\overline{\s}_\vep^\beta(v_\ell) \d \bs u_{m+1}\d \bs v_m- \int_{\Delta(t)} F_0(\bs u_{m+1},\bs v_m) \prod_{\ell=1}^m\s^\beta(v_\ell)\d \bs u_{m+1}\d \bs v_m\right|<\eta.
\end{align}
This property leads to \eqref{final_limit!!} by the choice of $F_0$. To prove \eqref{final_claim!!},  first, we choose for every $\delta_0>0$ a $[0,1]$-valued continuous function $\chi_{\delta_0}$ defined on $[0,\infty)$ such that $\chi_{\delta_0}(v)=1$ for all $v\geq \delta_0$ and $\chi_{\delta_0}(0)=0$. For all $\bs a\times \bs b\in \{0,1\}^{m+1}\times\{0,1\}^m$, write $\tilde{\chi}=1-\chi$ and
\[
\chi_{\delta_0;\bs a\times \bs b}(\bs u_{m+1},\bs v_m)=\prod_{\ell:a_\ell=1}\chi_{\delta_0}(u_\ell)\prod_{\ell:a_\ell=0}\tilde{\chi}_{\delta_0}(u_\ell)\prod_{\ell:b_{\ell}=1}\chi_{\delta_0}(v_\ell)\prod_{\ell:b_{\ell}=0}\tilde{\chi}_{\delta_0}(v_\ell).
\]
The sum of $\chi_{\delta_0;\bs a\times \bs b}(\bs u_{m+1},\bs v_m)$ over $\bs a\times \bs b$ equals to $1$. 
Then as in Step~\hyperlink{PP:21}{2-1} (see \eqref{ensemble-est}) with $q=0$,
\begin{align}
&\eqspace\int_{\Delta(t)} F_0(\bs u_{m+1},\bs v_m)\chi_{\delta_0;\bs a\times \bs b}(\bs u_{m+1},\bs v_m)\prod_{\ell=1}^m\overline{\s}_\vep^\beta(v_\ell) \d \bs u_{m+1}\d \bs v_m\notag\\
\begin{split}\label{ab:bdd!!!}
&\leq \max_{s\in (0,\infty)}P_{2s}(x_0^{i_1\prime}-x_0^{i_1})\times  t^2\Bigg(\prod_{\ell:b_\ell=1}\int_0^{t} \overline{\s}^\beta_\vep(v)\d v\Bigg)\times \Bigg(\prod_{ \ell:b_\ell=0}\int_0^{\delta_0}\overline{\s}^\beta_\vep(v)\d v\Bigg)\\
&\eqspace\times\int_{u_1,\cdots,u_{m+1}\in (0,1)}\d \bs u_{m+1}\Bigg(\prod_{\ell=1}^{m-1}\frac{\1_{\{a_\ell=1\}}\chi_{\delta_0}(u_\ell/t)+
\1_{\{a_\ell=0\}}\tilde{\chi}_{\delta_0}(u_\ell/t)}{\sqrt{u_{\ell+1}(u_{\ell+1}+u_{\ell})}}\Bigg)\\
&\eqspace\times \prod_{\ell=m}^{m+1}\big\{\1_{\{a_\ell=1\}}\chi_{\delta_0}(u_\ell/t)+
\1_{\{a_\ell=0\}}\tilde{\chi}_{\delta_0}(u_\ell/t)\big\}.
\end{split}
\end{align}

We are ready to prove \eqref{final_claim!!}. Note that  $ \int_0^t\overline{\s}^\vep_\beta(r)\d r$ converges to $\int_0^t\s_\beta(r)\d r$ for all fixed $t\geq 0$, and $\lim_{t\searrow 0}\int_0^t \s_\beta(r)\d r=0$. Whenever there is at least some $\ell$ such that $a_\ell=0$, it follows from Lemma~\ref{lem:CSZ} (1$\cc$) and dominated convergence that 
\begin{align*}
&\lim_{\delta_0\to 0}\int_{u_1,\cdots,u_{m+1}\in (0,1)}\d \bs u_{m+1}\Bigg(\prod_{\ell=1}^{m-1}\frac{\1_{\{a_\ell=1\}}\chi_{\delta_0}(u_\ell/t)+
\1_{\{a_\ell=0\}}\tilde{\chi}_{\delta_0}(u_\ell/t)}{\sqrt{u_{\ell+1}(u_{\ell+1}+u_{\ell})}}\Bigg)\\
&\eqspace\times \prod_{\ell=m}^{m+1}\big\{\1_{\{a_\ell=1\}}\chi_{\delta_0}(u_\ell/t)+
\1_{\{a_\ell=0\}}\tilde{\chi}_{\delta_0}(u_\ell/t)\big\}=0.
\end{align*}
By these observations, we see that for all $\bs a\times \bs b\neq \mathbf 1_{m+1}\times \mathbf 1_m$, the bound in \eqref{ab:bdd!!!} and its analogue with $\overline{\s}^\beta_\vep$ replaced by $\s^\beta$
can be made strictly smaller than $\eta/(3\cdot 2^{2m+1})$ by first choosing $\delta_0$ and then $\vep_0$ according to $\delta_0$. On the other hand, it follows from \eqref{conv:port} that by choosing smaller $\vep_0$ if necessary, 
\begin{align}\label{final_claim!!+++}
\begin{split}
&\Bigg|\int_{\Delta(t)} (F_0\chi_{\delta_0;\mathbf 1_{m+1}\times \mathbf 1_m})(\bs u_{m+1},\bs v_m)\prod_{\ell=1}^m\overline{\s}_\vep^\beta(v_\ell) \d \bs u_{m+1}\d \bs v_m\\
&\hspace{1cm}- \int_{\Delta(t)} (F_0\chi_{\delta_0;\mathbf 1_{m+1}\times \mathbf 1_m})(\bs u_{m+1},\bs v_m)\prod_{\ell=1}^m\s^\beta(v_\ell) \d \bs u_{m+1}\d \bs v_m\Bigg|<\frac{\eta}{3},\quad \forall\;\vep\in (0,\ov).
\end{split}
\end{align}
These bounds are enough to get \eqref{final_claim!!}. The proof is complete.
\end{proof}

\begin{proof}[End of the proof of Theorem~\ref{thm:main2}]
The proof of this theorem follows from Propositions~\ref{prop:series1} and~\ref{prop:series2} if we choose $\vep_0\in (0,\ov)$ such that for appropriate $q\in (\beta,\infty)$ in \eqref{def:Delta1}, 
\begin{align}\label{qvep:bdd}
32^{m}\left[\widehat{\S}^\beta_\vep(q)^m+\left(\frac{4\pi}{\log (q/\beta)}\right)^m\right]\leq  2\cdot\left(\frac{3}{4}\right)^m,\quad \forall\;m\geq 1.
\end{align}
To see this bound, note that we can choose $q$ large enough such that 
\[
\frac{32\cdot 4\pi}{\log (q/\beta)}<\frac{1}{2}. 
\]
Additionally, since $\widehat{\S}^\beta_\vep(q)\to 4\pi/\log (q/\beta)$, we can find $\vep_0\in (0,\ov)$ such that 
\[
\widehat{\S}^\beta_\vep(q)<\frac{4\pi}{\log (q/\beta)}+\frac{1}{4}<\frac{3}{4},\quad \forall\;\vep\in (0,\vep_0).
\] 
We have obtained \eqref{qvep:bdd} by the last two displays. The proof is complete.
\end{proof}

\section{Some special functions and exponential functionals}\label{sec:expmom}
In this section, we collect some properties of the Bessel functions and explicit formulas of exponential functions with positive exponents for one-dimensional Brownian motion and $\BES^2$.

Let us start with properties of the Bessel functions. Recall that given an order $\nu$, $J_\nu$ denotes the Bessel function of the first kind, $I_\nu$ the modified Bessel function of the first kind, and $K_\nu$ the Macdonald function. Then the following integral representations hold:
\begin{align}
K_\nu(z)&=\frac{1}{2}\left(\frac{z}{2}\right)^\nu \int_0^\infty t^{-\nu-1}\e^{-t-(z^2/4t)}\d t,\quad |\arg(z)|<\frac{\pi}{4};\label{def:K}\\
I_\nu(z)&=\frac{(z/2)^\nu}{\sqrt{\pi}\,\Gamma(\nu+1/2)}\int_{-1}^1 (1-t^2)^{\nu-1/2}\cosh(zt)\d t,\quad|\arg (z)|<\pi,\; \Re(\nu)>-1/2\label{def:I}
\end{align} 
\cite[(5.10.25) and (5.10.22) on p.119]{Lebedev}, and $I_0(z)=J_0(\i z)$ for all $z\in \Bbb C$.
We also have the alternative expressions:
\begin{align}\label{def:JIseries}
I_0(z)=\sum_{n=0}^\infty \frac{(z/2)^{2n}}{\Gamma(n+1)^2},\quad 
J_0(z)=\sum_{n=0}^\infty \frac{(-1)^n(z/2)^{2n}}{\Gamma(n+1)^2};
\end{align}
the two series have infinite radii of convergence \cite[(5.3.2) on p.102 and (5.7.1) on p.108]{Lebedev}. 

\begin{lem}\label{lem:asympgauss}
As $a\to 0+$,
\begin{align}
\frac{1}{\pi}K_0(\two a)&=\int_0^\infty \frac{\e^{- t}}{2\pi t}\exp\left\{-\frac{a^2}{t}\right\}\d t=\frac{1}{\pi}\left(\log \frac{1}{a}-\EM\right)+\mathcal O\left(a^2\log \frac{1}{a}\right),\label{asymp:gauss}\\
K_1(a)&=\frac{1}{a}\int_0^\infty  \e^{- s}\int_{a^2/(4s)}^\infty \e^{-v}\d v \d s\sim \frac{1}{a},\label{ineq:K1bdd}
\end{align}
where $\mathcal O(a)$ as $a\to 0$ means $|\mathcal O(a)|\less |a|$ for all $0<|a|\leq 1$.  
\end{lem}
\begin{proof}
We only prove \eqref{ineq:K1bdd}:
\begin{align*}
K_1(a)&=\frac{a}{4}\int_0^\infty \frac{\e^{- t}}{ t^2}\exp\left\{-\frac{a^2}{4t}\right\}\d t
=\frac{a}{4}\int_0^\infty \e^{- s}\int_0^s\frac{1}{ t^2}\exp\left\{-\frac{a^2}{4t}\right\}\d t \d s\\
&=\frac{1}{a}\int_0^\infty  \e^{- s}\int_{a^2/(4s)}^\infty \e^{-v}\d v \d s\sim \frac{1}{a},
\end{align*} 
where the third equality uses the substitution $v=a^2/(4t)$. See also \cite[(5.16.4) on p.136]{Lebedev}. 
\end{proof}

Next, we turn to exponential functionals of Brownian motions.

\begin{thm}\label{thm:expmom1}
For any $\mu\geq 0$ and $M\in (0,\infty)$,
\begin{align}
&\E_a[\e^{-\mu T_b(\rho)}]=\frac{\int_0^\infty t^{-1}\exp\big(-t-\frac{a^2\mu  }{2t}\big)\d t}{\int_0^\infty t^{-1}\exp\big(-t-\frac{b^2\mu }{2t}\big)\d t}=\frac{K_0(a\sqrt{2\mu}\,)}{K_0(b\sqrt{2\mu}\,)},
\hspace{4.3cm}\mbox{$0<b\leq a$};\label{bes:hit1}\\
&\E_a[\e^{-\mu T_b(\rho)}]=\frac{ \int_{-1}^1 (1-t^2)^{-1/2}\cosh(a\sqrt{2\mu}t)\d t}
{\int_{-1}^1 (1-t^2)^{-1/2}\cosh(b\sqrt{2\mu}t)\d t}=\frac{I_0(a\sqrt{2\mu}\,)}{I_0(b\sqrt{2\mu}\,)},
\hspace{3.6cm}\mbox{$0<a\leq b$};\label{bes:hit2}\\
&\E_{\log a}\left[\exp\left\{-\mu\int_0^{T_{\log b}(\beta)}\1_{(-\infty,\log M]}(\beta_v)\d v\right\}\right]=\frac{\cosh\big(\log (\tfrac{M}{a\wedge M})\sqrt{2\mu}\,\big)}{\cosh\big(\log(\tfrac{M}{b\wedge M})\sqrt{2\mu}\,\big)},\hspace{1.5cm} 0<b\leq a;\label{bm:prob1}\\
&\E_{\log a}\left[\exp\left\{-\mu\int_0^{T_{\log b}(\beta)}\e^{2\beta_v}\d v\right\}\right]=\frac{ \int_{-1}^1 (1-t^2)^{-1/2}\cosh(a\sqrt{2\mu}t)\d t}
{\int_{-1}^1 (1-t^2)^{-1/2}\cosh(b\sqrt{2\mu}t)\d t},\hspace{1.8cm}0<a\leq b\label{bm:prob2}.
\end{align}
\end{thm}

See  \cite[(2.0.1) on p.513, (2.5.1) on p.207 and (2.10.3) on p.210]{BS:Handbook} for the four identities
in Theorem~\ref{thm:expmom1}. Identities \eqref{bes:hit1} and \eqref{bes:hit2} are obtained in \cite{Kent78,GS}, and we use the integral representations in \eqref{def:K} and \eqref{def:I}.  For the proof of \eqref{bm:prob1} when $a\leq M$, see  \cite[Example~2.1 on page 437 with $\mu=0$]{Knight}. Note that \eqref{bm:prob2}  coincides with \eqref{bes:hit2} since the stopped additive functional in \eqref{bm:prob2} has the same law as the hitting time of the two-dimensional Bessel process by the skew product representation of two-dimensional Brownian motion. See also \cite{HM} for related formulas.

To extend the formulas in Theorem~\ref{thm:expmom1} to positive exponents, we first recall a basic result of analytic characteristic functions \cite[Section~2.3 starting on p.20, especially Theorem~2.3.2 on p.22]{Lukacs}.

\begin{prop}\label{prop:extension}
Suppose that the characteristic function $\phi$ of a random variable $X$ admits an analytic extension in the open disc $D(0,R)$, with radius $R\leq \infty$ and center at the origin, in $\Bbb C$. Then $\E[\exp a|X|]<\infty$ for all $0\leq a<R$. 
\end{prop}

In contrast to the rigidity of \eqref{bes:hit1} due to the use of the Macdonald function $K_0$, the other three formulas \eqref{bes:hit2}, \eqref{bm:prob1} and \eqref{bm:prob2} can be extended analytically from positive $\mu$ to negative $\mu$ by the analogous property of the hyperbolic cosine. We show the precise result after the following lemmas.

\begin{lem}\label{lem:Phi}
For $0\neq w\in \Bbb C$, the entire function $\Phi_w(z)\defeq\sum_{n=0}^\infty  (w/2)^{2n}z^n/\Gamma(n+1)^2$ satisfies the following properties:
\begin{enumerate}
\item [\rm (1$\cc$)] For all $z\in \Bbb C$, $\Phi_w(z)=\cosh(w\sqrt{z}\,)$. Here, $\sqrt{z}=\sqrt{|z|}\e^{\i \arg(z)/2}$, for $\arg(z)\in (-\pi,\pi]$,  is well-defined on $\Bbb C$ and is analytic on $\Bbb C\setminus(-\infty,0]$ such that $(\sqrt{z}\,)^2=z$. Hence, $\Phi_w$ is an entire extension of $z\mapsto \cosh(w\sqrt{z}\,)$.
\item [\rm (2$\cc$)] For all $z\in \R_+$, $\Phi_w(-z)=\cos(w\sqrt{z}\,)$.
\item [\rm (3$\cc$)] The set of zeros of $\Phi_w$ is 
$\{-w^{-2}(n\pi+ \pi/2)^2;n\in \Bbb Z_+\}$.
\end{enumerate}
\end{lem}
\begin{proof}
The properties in (1$\cc$) follow from the Taylor expansion of the hyperbolic cosine and a use of the principal branch of logarithm. For (2$\cc$), we use the Taylor expansion of the cosine to write
\[
\forall\;z\in \R_+,\quad \cos(w \sqrt{z}\,)=\sum_{n=0}^\infty (-1)^n\frac{w^{2n}z^n}{(2n)!}=\sum_{n=0}^\infty \frac{w^{2n}(-z)^n}{(2n)!}=\Phi_w(-z).
\]

To see (3$\cc$), note that the zeros of $\cosh(\zeta)$ are given by $\i(n\pi+ \pi/2)$, for all $n\in \Bbb Z$. This property can be seen by solving $\e^\zeta+\e^{-\zeta}=0$ for $\e^\zeta$ and then for $\zeta$. Thus, by the representation of $\Phi_w$ in (1$\cc$), the zero set of $\Phi_w$ is $\{-w^{-2}(n\pi+ \pi/2)^2;n\in \Bbb Z\}$. Moreover, since, for $n\in \Bbb N$, $(-n\pi+\pi/2)^2=(n\pi-\pi/2)^2=((n-1)\pi+\pi/2)^2$, the zero set of $\Phi_w$ simplifies to the required one. 
\end{proof}

\begin{lem}\label{lem:Psi}
For $0\neq w\in \Bbb C$, the entire function $\varphi_w(z)\defeq\sum_{n=0}^\infty (w/2)^{2n}z^{n}/(n!)^2$ satisfies the following properties:
\begin{enumerate}
\item [\rm (1$\cc$)] For all $z\in \Bbb C$, $I_0(w\sqrt{z}\,)=\varphi_w(z)=
J_0(w\sqrt{-z}\,)$.
\item [\rm (2$\cc$)] The zero set of $\varphi_w$ is 
$\{- w^{-2}j_{0,n}^2;n\in \Bbb N\}$, where $0<j_{0,1}<j_{0,2}<\cdots$ are the positive zeros of $J_0$. 
\end{enumerate}
\end{lem}
\begin{proof}
The two identities in (1$\cc$) follow from the Taylor expansions  recalled in \eqref{def:JIseries}. 
To see (2$\cc$), recall that the positive zeros of the even function $J_0$ are $0<j_{0,1}<j_{0,2}<\cdots$ \cite[Section~9.5 p.370]{AS:Handbook}, and the zeros of $J_0$ are all real  \cite[Section~9.5 p.372]{AS:Handbook}. Solving the equations $\pm j_{0,n}=w\sqrt{-z_n}$ gives the required zero set of $\varphi_w$. 
\end{proof}

\begin{prop}\label{prop:expmom2}
For any $M,\nu\in (0,\infty)$, it holds that
\begin{align}
\begin{split}
&\E_{\log a}\left[\exp\left\{\nu\int_0^{T_{\log b}(\beta)}\1_{(-\infty,\log M]}(\beta_v)\d v\right\}\right]=\frac{\cos\big(\log (\tfrac{M}{a\wedge M})\sqrt{2\nu}\,\big)}{\cos\big(\log(\tfrac{M}{b\wedge M})\sqrt{2\nu}\,\big)},\hspace{.3cm}M\e^{-\frac{\pi}{2\sqrt{2\nu}}}<b\leq a;\label{bm:prob3}
\end{split}\\
\begin{split}
&\E_{\log a}\left[\exp\left\{\nu\int_0^{T_{\log b}(\beta)}\e^{2\beta_v}\d v\right\}\right]=\frac{ \int_{-1}^1 (1-t^2)^{-1/2}\cos(a\sqrt{2\nu}t)\d t}
{\int_{-1}^1 (1-t^2)^{-1/2}\cos(b\sqrt{2\nu}t)\d t},\hspace{.6cm}0<a\leq b<\frac{j_{0,1}}{\sqrt{2\nu}},\label{bm:prob4}
\end{split}
\end{align}
where $j_{0,1}=2.40483...$ is the smallest positive zero of $J_0$.  
\end{prop}
\begin{proof}
We start with the proof of \eqref{bm:prob3}. If we restrict attention to $0<b\leq a$, it is enough to further assume $b<M$. In this case, it follows from Proposition~\ref{prop:extension} and Lemma~\ref{lem:Psi} (3$\cc$) that \eqref{bm:prob1} can be extended analytically to the half-space 
\begin{align}\label{def:halfspace}
\Big\{\mu\in \Bbb C\,;\,\Re(2\mu)>-\Big(\log \frac{M}{b}\Big)^{-2}\cdot \Big(\frac{\pi}{2}\Big)^2\Big\}.
\end{align}
Moreover, Lemma~\ref{lem:Phi} (3$\cc$) shows that the explicit form of the analytic extension to $\mu$ such that $\Re(\mu)<0$ can be obtained by replacing the hyperbolic cosine with the cosine. For real $\nu=-\mu> 0$, 
\[
\cosh\big(\log (\tfrac{M}{a})\sqrt{2\mu}\,\big)=\cos\big(\log (\tfrac{M}{a})\sqrt{-2\mu}\,\big)\quad\mbox{and}\quad \cosh\big(\log (\tfrac{M}{b})\sqrt{2\mu}\,\big)=\cos\big(\log (\tfrac{M}{b})\sqrt{-2\mu}\,\big).
\] 
To see the lower bound of $b$ in \eqref{bm:prob3}, note that the defining condition of the half-space \eqref{def:halfspace} requires
\[
2\nu<\Big(\log \frac{M}{b}\Big)^{-2}\cdot \Big(\frac{\pi}{2}\Big)^2\Longleftrightarrow M\e^{-\frac{\pi}{2\sqrt{2\nu}}}<b.
\]
We have proved all of the properties stated in \eqref{bm:prob3}. 

For the proof of \eqref{bm:prob4}, Proposition~\ref{prop:extension} and Lemma~\ref{lem:Phi} again allow for analytic extensions of $\mu\mapsto I_0(a\sqrt{2\mu}\,)$ and $\mu\mapsto I_0(b\sqrt{2\mu}\,)$ to $\Re(\mu)<0$ since $\int_0^1(1-t^2)^{-1/2}\d t<\infty$ (the precise value is $\pi/2$). For real $\nu=-\mu>0$, we can write
\[
I_0(a\sqrt{2\mu}\,)=\int_{-1}^1 (1-t^2)^{-1/2}\cosh(a\sqrt{2\mu}t)\d t=\int_{-1}^1 (1-t^2)^{-1/2}\cos(a\sqrt{2\nu}t)\d t.
\]
A similar integral representation holds for $I_0(b\sqrt{2\mu}\,)$, $\mu<0$. To get an upper bound for the eligible $b$, now we have to consider the zeros of $I_0$ instead by using Lemma~\ref{lem:Psi} (3$\cc$). Hence, \eqref{bm:prob2} can be extended analytically to the half-space
$\{\mu\in \Bbb C\,;\,\Re(2\mu)>-b^{-2}j_{0,1}^2\}$. 
The defining condition of the foregoing half-space requires
\[
2\nu<b^{-2} j_{0,1}^2\Longleftrightarrow b<\frac{j_{0,1}}{\sqrt{2\nu}}.
\]
The formula in \eqref{bm:prob4} follows.
\end{proof}

\bibliographystyle{abbrv}

\end{document}